\documentclass{article}
\usepackage{graphicx}
\usepackage[english]{babel}
\usepackage{caption}
\usepackage{subcaption}
\usepackage{amsmath, amsfonts, amssymb, amsthm, bm}  
\usepackage{mathrsfs}
\usepackage{url}
\usepackage{hyperref}

\theoremstyle{plain}
\newtheorem{theorem}{Theorem}[section]
\newtheorem{proposition}[theorem]{Proposition}
\newtheorem{lemma}[theorem]{Lemma}
\newtheorem{corollary}[theorem]{Corollary}
\newtheorem{remark}[theorem]{Remark}

\theoremstyle{definition}

\newtheorem{example}{Example}

\newcommand{\bydef}{\stackrel{\textnormal{\tiny def}}{=}}
\newcommand{\R}{\mathbb{R}}
\newcommand{\LP}{G_{\text{\textup{LP}}}}
\newcommand{\bc}{\mathbf{c}}
\newcommand{\talpha}{{\tilde{\alpha}}}
\newcommand{\tx}{{\tilde{\chi}}}
\newcommand{\bx}{{\bar{\chi}}}
\newcommand{\order}{K}
\newcommand{\cheb}{\mathcal{T}}
\newcommand{\conjugate}{\mathcal{S}}

\begin{document}
\title{Constructive proofs for localized radial solutions \\ of semilinear elliptic systems on $\R^d$}
\author{
Jan Bouwe van den Berg \thanks{VU Amsterdam, Department of Mathematics, De Boelelaan 1081, 1081 HV Amsterdam, The Netherlands. {\tt janbouwe@few.vu.nl}}
\and
Olivier Hénot \thanks{McGill University, Department of Mathematics and Statistics, 805 Sherbrooke Street West, Montreal, QC, H3A 0B9, Canada. {\tt olivier.henot@mail.mcgill.ca}}
\and
Jean-Philippe Lessard \thanks{McGill University, Department of Mathematics and Statistics, 805 Sherbrooke Street West, Montreal, QC, H3A 0B9, Canada. {\tt jp.lessard@mcgill.ca}}
}

\date{}

\maketitle

% \tableofcontents

\begin{abstract}
Ground state solutions of elliptic problems have been analyzed extensively in the theory of partial differential equations, as they represent fundamental  spatial patterns in many model equations. While the results for scalar equations, as well as certain specific classes of elliptic systems, are comprehensive, much less is known about these localized solutions in generic systems of nonlinear elliptic equations. In this paper we present a general method to prove constructively the existence of localized radially symmetric solutions of elliptic systems on $\R^d$. Such solutions are essentially described by systems of non-autonomous ordinary differential equations. We study these systems using dynamical systems theory and computer-assisted proof techniques, combining a suitably chosen Lyapunov-Perron operator with a Newton-Kantorovich type theorem. We demonstrate the power of this methodology by proving specific localized radial solutions of the cubic Klein-Gordon equation on $\R^3$, the Swift-Hohenberg equation on $\R^2$, and a three-component FitzHugh-Nagumo system on $\R^2$. These results illustrate that ground state solutions in a wide range of elliptic systems are tractable through constructive proofs.
\end{abstract}

\begin{center}
\end{center}

%%%%%%%%%%%%%
%% INTRODUCTION %%
%%%%%%%%%%%%%

\section{Introduction} \label{sec:introduction}
%!TEX root = radial_elliptic.tex

The study of solutions of semilinear elliptic partial differential equations (PDEs) on $\R^d$ is vast and has received a tremendous amount of attention since the middle of the twentieth century. The fact that these PDEs are posed on all of $\R^d$ often forces \emph{radially symmetric} solutions to exist, that is to say solutions of the form $U(x) = u((x_1^2 + \ldots + x_d^2)^{1/2})$. A canonical example of such a result is given by the scalar equation 
\begin{equation}\label{eq:classical}
\Delta U + U^p=0, \qquad U=U(x) \in \R, \quad x \in \R^d,
\end{equation}
for which it was proven in the case $d \ge 3$ and critical exponent $p=(d+2)/(d-2)$ that all the positive solutions are radially symmetric (e.g. see \cite{MR634248,MR982351,MR1121147}). Existence of radial solutions for more general scalar models
\[
\Delta U + \mathbf{N}(U)=0, \qquad U=U(x) \in \R, \quad x \in \R^d,
\]
with $\mathbf{N}$ nonlinear, was proven using variational methods \cite{MR454365,MR695535,MR695536,MR165176}, dynamical systems techniques \cite{MR846391}, shooting methods \cite{MR931016,MR1033193} and moving plane techniques \cite{MR1158936}, while uniqueness was also extensively investigated \cite{MR886933,MR333489,MR682268,MR829369}.

The generalization of these methods to systems of elliptic equations is hindered by the fact that variational techniques, shooting methods and the maximum principle are in general only available for scalar equations. Thus, it is not surprising that existence results for radial solutions of systems of semilinear elliptic equations are more sparse. Nevertheless, there are several examples of existence results for radial solutions on $\R^d$ for systems with specific structures. For instance, \emph{cooperative} systems \cite{MR1755067} or systems involving the notion of $c$-positive functions \cite{MR1765572} (in both case, a form of maximum principle holds), Lane-Emden systems \cite{MR1363170,MR1645728} and Hamiltonian systems~\cite{MR1620581}. It remains however a challenging problem to obtain proofs of radial solutions for elliptic systems without any additional structure.

In this paper, we propose a general (computer-assisted) technique to constructively prove the existence of localized radial solutions of elliptic semilinear systems of the form
\begin{equation}\label{eq:elliptic}
\Delta U + \mathbf{N}(U) = 0, \qquad U = U(x) \in \mathbb{R}^q, \quad x \in \mathbb{R}^d,
\end{equation}
where $\mathbf{N}$ is a nonlinear function. For the sake of clarity and uniformity of the presentation, we assume that $\mathbf{N}$ is polynomial; although, as made clear in Remark~\ref{rem:nonlinearity}, this restriction may be lifted to study certain non-polynomial elliptic systems via essentially the same methodology. Furthermore, since our interest lies in localized patterns, we assume the existence of a zero of $\mathbf{N}$, that is to say a constant solution $c \in \mathbb{R}^q$ of \eqref{eq:elliptic}, such that the Jacobian matrix $D\mathbf{N}(c)$ does not have any eigenvalues in $[0,\infty)$. This hypothesis will be justified later. It suffices to say here that it corresponds to a well-defined notion of hyperbolicity of the constant state. The condition is natural from that perspective, but does exclude the case of a monomial nonlinearity like the one in~\eqref{eq:classical}.

We then look for solutions of the form $U(x) = u(r(x))$, where $r(x) \bydef (x_1^2 + \ldots + x_d^2)^{1/2}$, with asymptotic behaviour
\begin{equation}\label{eq:asymptoticBC} 
\lim_{r \to \infty} u(r) = c \quad\text{and}\quad \lim_{r \to \infty} u'(r) = 0.
\end{equation}
The function $u$ solves the second-order non-autonomous system of ordinary differential equations (ODEs)
\begin{equation} \label{eq:radial_elliptic}
\left(\frac{d^2}{d r^2} + \frac{d-1}{r} \frac{d}{d r}\right) u(r) + \mathbf{N}(u(r)) = 0, \qquad u(r) \in \mathbb{R}^q, \quad r \in [0,\infty).
\end{equation}

Our strategy for proving existence of localized radial solutions of \eqref{eq:elliptic} reduces to establishing solutions of \eqref{eq:radial_elliptic} satisfying the boundary condition $u'(0) = 0$ and asymptotic behaviour~\eqref{eq:asymptoticBC}.
We subdivide the domain $[0,\infty)$ of $u$ into the three subdomains $[0,l_1]$, $[l_1,l_2]$ and $[l_2,\infty)$, with $0<l_1<l_2$, on which we solve~\eqref{eq:radial_elliptic} with different methods, namely a Taylor expansion on $[0,l_1]$, a Chebyshev expansion on $[l_1,l_2]$ and a center-manifold parametrization (graph) on $[l_2,\infty)$. Let us briefly describe each method separately, as well as how they fit together.

For $r \in [0,l_1]$ we substitute a Taylor series into \eqref{eq:radial_elliptic} and look for a solution (in the space of Taylor series coefficients) of the resulting recurrence relation, where the leading Taylor coefficient $u(0)$ is a priori unknown. 

For $r \ge l_1$, we rewrite \eqref{eq:radial_elliptic} as a $(2q+1)$-dimensional system of first-order polynomial autonomous ODEs. This technique is commonly used in the literature (e.g.\ see~\cite[\S~3]{MR2475557}). More explicitly, we introduce $w^{(1)}, w^{(2)}, w^{(3)}$ representing the maps $r \mapsto r^{-1}, r \mapsto u(r), r \mapsto \frac{d}{d r} u(r)$ respectively. Denoting $w = (w^{(1)}, w^{(2)}, w^{(3)}) $, we obtain
\begin{equation}\label{eq:autonomous_ode}
\frac{d}{d r} w(r)
= f(w(r)) \bydef
\begin{pmatrix}
-w^{(1)}(r)^2 \\
w^{(3)}(r) \\
- (d-1) w^{(1)}(r) w^{(3)}(r) - \mathbf{N}(w^{(2)}(r))
\end{pmatrix}.
\end{equation}
For a fixed $l_2>l_1$, we solve \eqref{eq:autonomous_ode} on $[l_1,l_2]$ using Chebyshev series. 

Since the equilibrium $\bc = (0,c,0) \in \mathbb{R} \times \mathbb{R}^q \times \mathbb{R}^q$ of \eqref{eq:autonomous_ode} represents the constant equilibrium $c \in \mathbb{R}^q$ of~\eqref{eq:radial_elliptic} in the limit $r \to \infty$, a localized radial solution converges to $\bc$ in system \eqref{eq:autonomous_ode}. As explained in detail in Section~\ref{sec:graph_enclosure}, the assumption that $D\mathbf{N}(c)$ does not have eigenvalues in $[0,\infty)$ implies that $\bc$ has a $q+1$-dimensional center-stable manifold.
Revisiting Chapter 4.1 of \cite{Chicone}, we use a \emph{Lyapunov-Perron} operator to control the center-stable manifold of $\bc$. 
In comparison with the estimates in \cite{Chicone}, in Section~\ref{sec:graph_enclosure} we leverage our precise knowledge of the center manifold, and the dynamics therein, to widen the verifiable domain of the local graph of the center-stable manifold and to improve control on the location of the manifold in terms of (generalized) Lipschitz bounds.
Thus, this method is tailored to satisfy the minimum conditions required by the Newton-Kantorovich type computer-assisted proof method we are aiming for.
Its simplicity comes at the cost of ceasing a $C^1$ control, and potentially relinquishing a larger domain for the local graph.
An alternative perspective is offered by an approach based on \emph{cone conditions} as presented in \cite{MR3463691}, which offers a $C^1$ bound on the local graph of the non-autonomous stable manifold of the constant equilibrium $c$ of \eqref{eq:radial_elliptic}.

The three strategies on the three subintervals are combined to define a projected boundary value problem (BVP) with a boundary condition at $r=l_2$ forcing the solution to lie in a center-stable manifold. The final step is to solve the BVP with a Newton-Kantorovich type theorem of which the hypotheses are verified with the help of a computer. This computer-assisted proof technique, referred to as the \emph{radii polynomial approach}, is set in a Banach space of the coefficients of some series representation of the solution, see e.g.~\cite{MR3917433,MR4417102,MR3323206,MR3454370} for some relevant works. 
As shown in Section~\ref{sec:newton-kantorovich}, the (generalized) Lipschitz bounds on the center-stable manifold that originate from the Lyapunov-Perron method fit seamlessly into the estimates needed for the Newton-Kantorovich type method. This leads to a constructive proof and yields a rigorous $C^0$-error bound on a numerical approximation of the desired localized radial solution of \eqref{eq:elliptic}. 
To illustrate the efficacy of the method we present three examples, for which the details are provided in Section~\ref{sec:applications}.

%%%%%%%%%%%% EXAMPLE KLEIN-GORDON %%%%%%%%%%%

\begin{example}\label{ex:exampleKG}
The cubic \emph{Klein-Gordon} equation
\begin{equation}\label{eq:exampleKG}
U_{tt} = \Delta U - U + \beta_1 U^2 +\beta_2 U^3, \qquad U = U(t, x) \in \mathbb{R}, \quad t \ge 0, \quad x \in \R^3,
\end{equation}
with parameters $\beta_1, \beta_2 \in \R$, is a relativistic wave equation serving as a model in scalar field theory (see e.g. \cite{MR3930829}). Ground states of \eqref{eq:exampleKG} solve~\eqref{eq:elliptic} with $q=1$, $d=3$ and $\mathbf{N}(U)=-U+\beta_1 U^2 +\beta_2 U^3$. It is known that for any positive values of $\beta_2$ and any given integer $m \ge 0$, there exists a localized radial stationary solution of \eqref{eq:exampleKG} with exactly $m$ zeros (see e.g. \cite{MR846391}).
Our method allows the rigorous computation of the profile of such solutions which, to the best of our knowledge, constitutes a novel result.
We prove the following constructive existence result, see Section~\ref{sec:3DKG}.
\begin{theorem}\label{thm:exampleKG}
For $\beta_1 = \beta_2 = 1$, there exist two distinct localized radial stationary solutions of the Klein-Gordon equation~\eqref{eq:exampleKG} on $\R^3$.
Both solutions lie, respectively, within a distance $2.9 \times 10^{-7}$ and $5.5 \times 10^{-6}$ (in the $C^0$-norm) of their numerical approximation depicted in Figure~\ref{fig:exampleKG}.
\end{theorem}	
\end{example}

\begin{figure}
\centering
\begin{subfigure}[b]{0.4\textwidth}
\centering
\includegraphics[width=\textwidth]{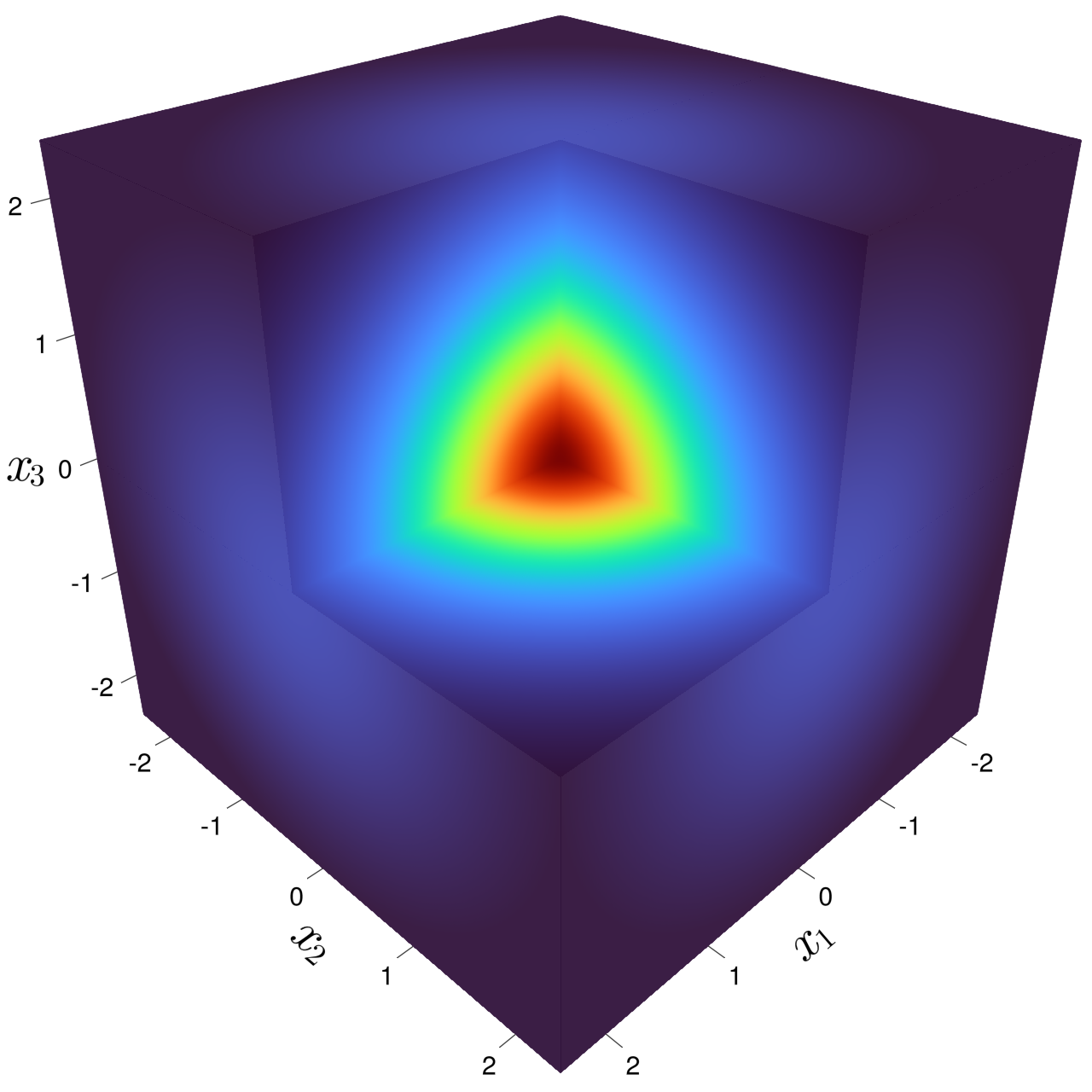}
\caption{}
\end{subfigure}
\hspace{0.1\textwidth}
\begin{subfigure}[b]{0.4\textwidth}
\centering
\includegraphics[width=\textwidth]{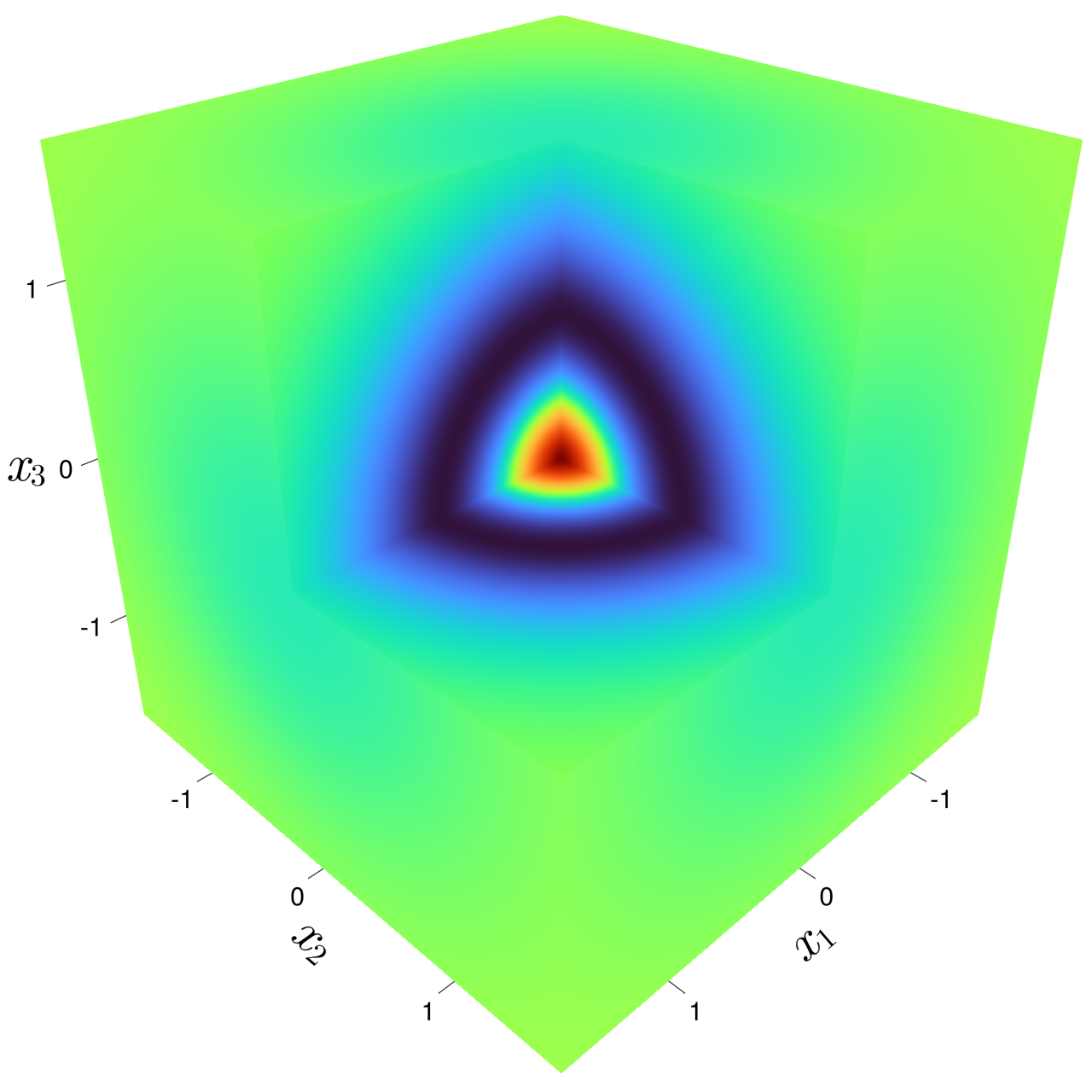}
\caption{}
\end{subfigure}
\caption{Numerical localized radial stationary solutions of the cubic Klein-Gordon equation \eqref{eq:exampleKG} on $\R^3$ at parameter values $\beta_1 = \beta_2 = 1$. (a) This approximation is proven to be $2.9 \times 10^{-7}$ close (in $C^0$-norm) to a strictly positive localized radial stationary solution. (b) This approximation is proven to be $5.5 \times 10^{-6}$ close (in $C^0$-norm) to a localized radial stationary solution with one zero.}
\label{fig:exampleKG}
\end{figure}

%%%%%%%%%%%% EXAMPLE SWIFT-HOHENBERG %%%%%%%%%%%

\begin{example}\label{ex:exampleSH}
The \emph{Swift-Hohenberg} equation on $\R^2$
\begin{equation}\label{eq:exampleSH}
U_t = -(\beta_4 + \Delta)^2 U + \beta_1 U + \beta_2 U^2 + \beta_3 U^3, \qquad U = U(t, x) \in \mathbb{R}, \quad t \ge 0, \quad x \in \R^2,
\end{equation}
with parameters $\beta_1, \beta_2, \beta_3, \beta_4 \in \R$, was introduced in \cite{SwiftHohenberg} to model thermal fluctuations. Notably, the Swift-Hohenberg equation has been studied extensively in the field of pattern formation.
Ground states of \eqref{eq:exampleSH} solve~\eqref{eq:elliptic} with $q=2$, $d=2$ and $\mathbf{N}(U_1,U_2)=(\beta_4 U_1-U_2,\beta_4 U_2 -\beta_1 U_1 -\beta_2 U_1^2 -\beta_3 U_1^3)$.
Existence results for localized stationary \emph{rings} of \eqref{eq:exampleSH} were obtained in \cite{MR2475557} under the assumption $0 < -\beta_1 \ll 1$ (note that $-\beta_1$ is denoted by $\mu$ in \cite{MR2475557}). The amplitude of those solutions is $O((-\beta_1)^{1/2})$. 
Our method is complementary, as we do not impose a condition on the smallness of $-\beta_1$. 
We prove the following constructive existence result for $-\beta_1 = 3/5$, see Section~\ref{sec:2DSH}.
\begin{theorem}\label{thm:exampleSH}
For $(\beta_1, \beta_2, \beta_3, \beta_4) = (-\tfrac{3}{5},\sqrt{6},-\tfrac{1}{10},1)$, there exists a localized radial stationary solution of the Swift-Hohenberg equation~\eqref{eq:exampleSH} on $\R^2$. This solution lies within distance $2.4 \times 10^{-5}$  (in the $C^0$-norm) of the numerical approximation depicted in Figure~\ref{fig:exampleSH}.
\end{theorem}	
\end{example}

\begin{figure}
\centering
\includegraphics[width=0.5\textwidth]{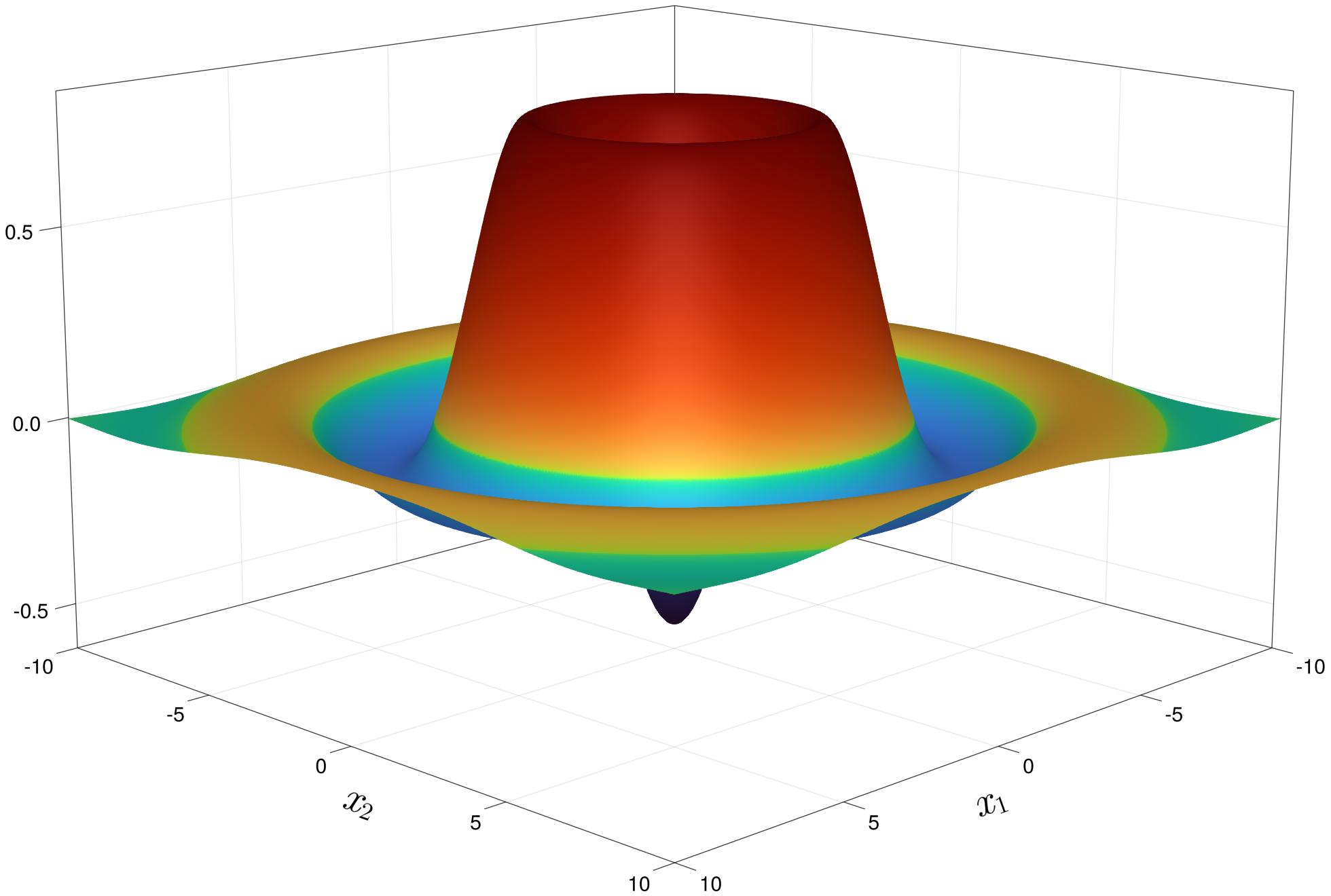}
\caption{Numerical stationary \emph{ring} of the Swift-Hohenberg equation \eqref{eq:exampleSH} on $\mathbb{R}^2$ at parameter values $(\beta_1, \beta_2, \beta_3, \beta_4) = (-\tfrac{3}{5},\sqrt{6},-\tfrac{1}{10},1)$. This approximation is proven to be $2.4 \times 10^{-5}$ close (in $C^0$-norm) to a true localized radial stationary solution.}
\label{fig:exampleSH}
\end{figure}

%%%%%%%%%%%% EXAMPLE FITZHUGH-NAGUMO %%%%%%%%%%%

\begin{example}\label{ex:exampleFN}
Consider the three-component \emph{FitzHugh-Nagumo} type equation 
\begin{equation}\label{eq:exampleFN}
\begin{cases}
\hspace{0.2cm} (U_1)_t = \epsilon^2 \Delta U_1 + U_1 - U_1^3 - \epsilon (\beta_1 + \beta_2 U_2 + \beta_3 U_3), \\
\tau (U_2)_t = \Delta U_2 + U_1 - U_2, \\
\theta (U_3)_t = \beta_4^2 \Delta U_3 + U_1 - U_3,
\end{cases}\quad U_i = U_i(t, x) \text{ for } i=1,2,3, \quad t \ge 0, \quad x \in \R^2,
\end{equation}
with parameters, $\tau, \theta, \epsilon>0$, $\beta_1, \beta_2, \beta_3 \in \R$ and $\beta_4 > 0$.
Originally, the FitzHugh-Nagumo model was introduced in \cite{FitzHugh} as a simple excitable-oscillatory system to describe nerve impulses in an axon. For the three-component variant~\eqref{eq:exampleFN} ground states solve \eqref{eq:elliptic} with $q=3$, $d=2$ and $\mathbf{N}(U) = \bigl(\epsilon^{-2} (U_1 - U_1^3) - \epsilon^{-1} (\beta_1 + \beta_2 U_2 + \beta_3 U_3) , U_1 - U_2 , \beta_4^{-2}(U_1 - U_3) \bigr)$. For $0 < \epsilon \ll 1$ there exist stationary planar radial \emph{spots} of \eqref{eq:exampleFN}, see e.g.\ Theorem 1.3 in \cite{MR3156797}. This result 
requires $\epsilon$ to be asymptotically small.
Our method does not impose any condition on the smallness of $\epsilon$ which, to the best of our knowledge, constitutes a novel result.
We prove the following constructive existence result, see Section \ref{sec:2DFN}.
\begin{theorem}\label{thm:exampleFN}
For $\epsilon = \tfrac{3}{10}$ and $(\beta_1,\beta_2,\beta_3,\beta_4) = (\tfrac{1}{2},\tfrac{1}{2},1,3)$, there exists a localized radial stationary solution of the three-component FitzHugh-Nagumo type equation~\eqref{eq:exampleFN} on $\R^2$. This solution lies within distance $9.8 \times 10^{-7}$ (in the $C^0$-norm) of the numerical approximation depicted in Figure~\ref{fig:exampleFN}.
\end{theorem}	
\end{example}

\begin{figure}
\centering
\begin{subfigure}[b]{0.3\textwidth}
\includegraphics[width=\textwidth]{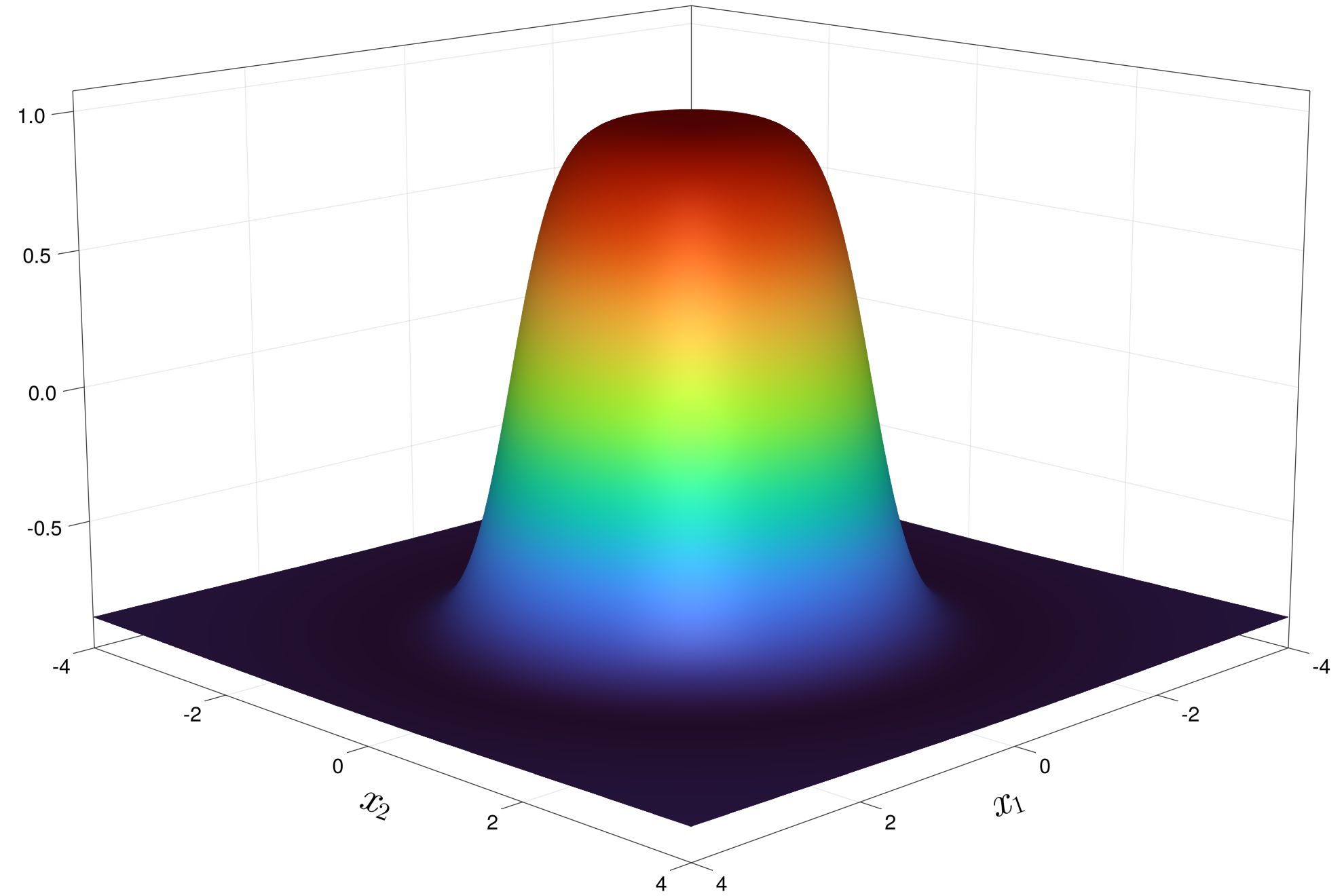}
\caption{}
\end{subfigure}
\hfill
\begin{subfigure}[b]{0.3\textwidth}
\includegraphics[width=\textwidth]{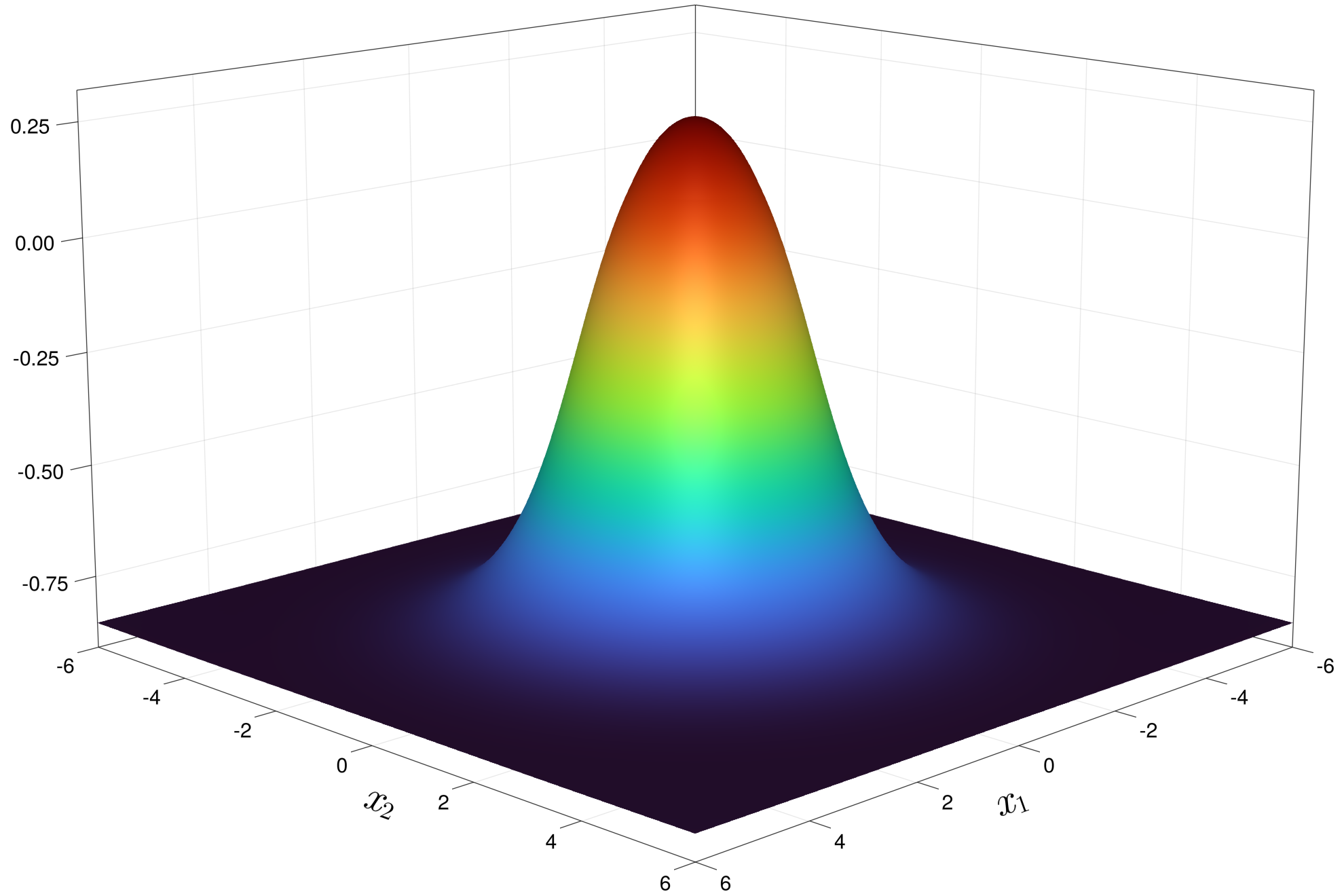}
\caption{}
\end{subfigure}
\hfill
\begin{subfigure}[b]{0.3\textwidth}
\includegraphics[width=\textwidth]{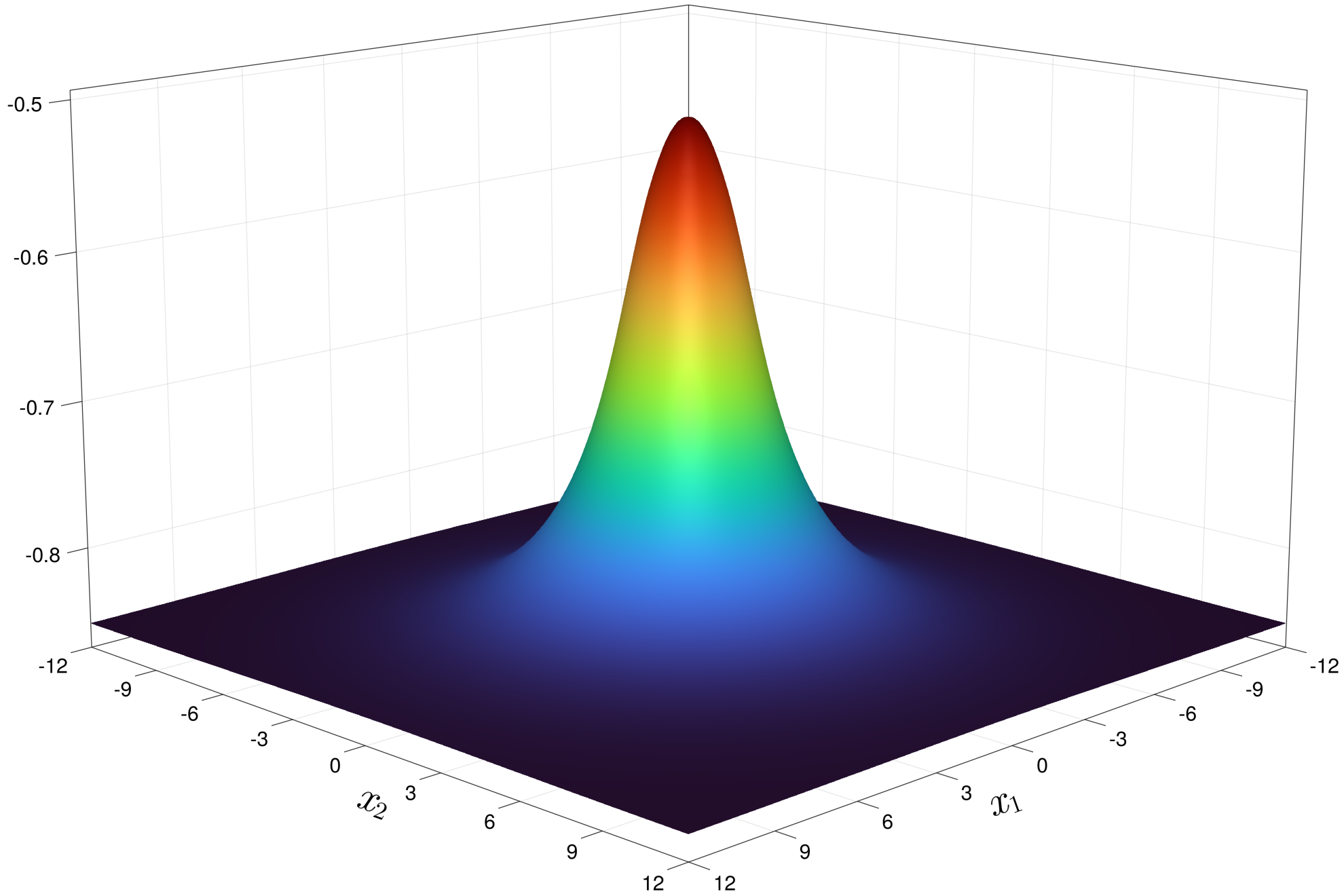}
\caption{}
\end{subfigure}
\caption{Numerical stationary planar radial \emph{spot} of the three-component FitzHugh-Nagumo type equation \eqref{eq:exampleFN} on $\mathbb{R}^2$ at parameter values $\epsilon = \tfrac{3}{10}$ and $(\beta_1,\beta_2,\beta_3,\beta_4) = (\tfrac{1}{2},\tfrac{1}{2},1,3)$; the three components are represented in the subfigures (a), (b) and (c), respectively. This approximation is proven to be $9.8 \times 10^{-7}$ close (in $C^0$-norm) to a true localized radial stationary solution.}
\label{fig:exampleFN}
\end{figure}

%%%%%%%%%% FINAL REMARKS %%%%%%%%%%%

All computational aspects are implemented in
Julia (cf.~\cite{Julia}) via the package \emph{RadiiPolynomial.jl} (cf.~\cite{RadiiPolynomial.jl}) which relies on the package
\emph{IntervalArithmetic.jl} (cf.~\cite{IntervalArithmeticJulia}) for
rigorous floating-point computations. The code is available at \cite{Code}.

\begin{remark}[{\bf The case of non-polynomial nonlinearities}] \label{rem:nonlinearity}
If the nonlinearity $\mathbf{N}$ appearing in \eqref{eq:radial_elliptic} is
non-polynomial and involves elementary functions (namely, exponentials, logarithms, algebraic functions or compositions thereof), then a change of coordinates may be
introduced to transform \eqref{eq:radial_elliptic} into a higher dimensional system of polynomial ODEs (e.g. see \cite{MR633878,MR3545977,MR4292532}). The
approach proposed in the present paper may be readily adapted, although with
some additional work in extending the construction of the center-stable
manifold.
\end{remark}

The paper is organized as follows. In Section~\ref{sec:graph_enclosure} we obtain a Lipschitz error bound for the center-stable manifold via Lyapunov-Perron's method. In Section~\ref{sec:newton-kantorovich} we present the projected boundary value problem on $[0,l_2]$, which connects at $r=l_2$ to the center-stable manifold. This boundary value problem is then split into two subintervals, as explained above, and subsequently solved using a Newton-Kantorovich type theorem.
Section \ref{sec:applications} presents applications where we rigorously compute localized radial solutions of the cubic Klein-Gordon equation on $\mathbb{R}^3$, the Swift-Hohenberg equation on $\mathbb{R}^2$ and a three-components FitzHugh-Nagumo type equation on $\mathbb{R}^2$.

%%%%%%%%%%%%%%%%%%%%
%% CENTER-STABLE MANIFOLD %%
%%%%%%%%%%%%%%%%%%%%

\section{Local enclosure of the center-stable manifold} \label{sec:graph_enclosure}
%!TEX root = radial_elliptic.tex

In this section, we employ the Lyapunov-Perron method to find a Lipschitz bound for the local graph of a center-stable manifold of the equilibrium $\bc = (0, c, 0)$ of \eqref{eq:autonomous_ode}.
We note that we could have opted to apply the results in Chapter 4.1 of \cite{Chicone}, which are valid for general local invariant manifolds, directly to our situation. However, for the center-stable manifold in our setting we can obtain improved bounds by incorporating the specific splitting of the center direction and the stable directions in the estimates. 

Linearizing~\eqref{eq:autonomous_ode} about $\bc$ yields
\[
Df(\bc) =
\begin{pmatrix}
0 & 0 & 0 \\
0 & 0 & I \\
0 & - D\mathbf{N}(c) & 0
\end{pmatrix}.
\]
Evidently, $0$ is an eigenvalue, which originates from the removal of the non-autonomous term $r^{-1}$ in \eqref{eq:radial_elliptic}. 
Hence, a localized solution to \eqref{eq:radial_elliptic} is, in the context of \eqref{eq:autonomous_ode}, a trajectory in a center-stable manifold of $\bc$.

The other eigenvalues of $Df(\bc)$ are the square roots of the eigenvalues of $-D\mathbf{N}(c)$. Since we assumed $D\mathbf{N}(c)$ has no eigenvalues in $[0,\infty)$, the Jacobian $Df(\bc)$ has no eigenvalues on the imaginary axis, apart from the simple $0$ eigenvalue already mentioned.
Furthermore, when $ 0 \neq \lambda \in \mathbb{C}$ is an eigenvalue, so is~$-\lambda$. To reduce technicalities and simplify notation, we continue the analysis under the assumption that all stable and unstable eigenvalues of $Df(\bc)$ are simple.
This is the generic case. We stress that the analysis below can be generalized in a relatively straightforward manner to the case of eigenvalues with higher multiplicity and associated generalized eigenvectors.

Let $\Lambda : \mathbb{C}^q \to \mathbb{C}^q$ denote the diagonal matrix of eigenvalues $\lambda_1,\dots,\lambda_q$ of $Df(\bc)$ with positive real part, and let $\Gamma : \mathbb{C}^q \to \mathbb{C}^q$ be the matrix satisfying $- D \mathbf{N}(c) \Gamma = \Gamma \Lambda^2$. It follows that
\[
Df(\bc)\begin{pmatrix}
0 \\ \Gamma \\ \hspace{0.25cm}\Gamma \Lambda
\end{pmatrix}
=
\begin{pmatrix}
0 \\ \Gamma \Lambda \\ \hspace{0.25cm}\Gamma \Lambda^2
\end{pmatrix}
=
\begin{pmatrix}
0 \\ \Gamma \\ \hspace{0.25cm}\Gamma \Lambda
\end{pmatrix} \Lambda.
\]
For the corresponding eigenvalues $-\Lambda$ with negative real part we have
\[
Df(\bc)\begin{pmatrix}
0 \\ \Gamma \\ -\Gamma \Lambda
\end{pmatrix}
=
\begin{pmatrix}
0 \\ -\Gamma \Lambda \\ \hspace{0.5cm}\Gamma \Lambda^2
\end{pmatrix}
=
\begin{pmatrix}
0 \\ \Gamma  \\ -\Gamma \Lambda
\end{pmatrix} (- \Lambda).
\]
Therefore, the matrix of eigenvectors of $Df(\bc)$ is given by
\[
M = \begin{pmatrix}
1 & 0 & 0 \\
0 & \Gamma & \Gamma \\
0 & - \Gamma \Lambda & \hspace{0.25cm}\Gamma \Lambda
\end{pmatrix} ,
\]
and its inverse is
\[
M^{-1} = \begin{pmatrix}
1 & 0 & 0 \\
0 & \tfrac{1}{2} \Gamma^{-1} & -\tfrac{1}{2} \Lambda^{-1} \Gamma^{-1} \\[1mm]
0 & \tfrac{1}{2} \Gamma^{-1} & \hspace{0.25cm}\tfrac{1}{2} \Lambda^{-1} \Gamma^{-1}
\end{pmatrix}.
\]
Since
\begin{align*}
& M^{-1}  f \left( \bc + M \begin{pmatrix}x\\ y\\ z\end{pmatrix}\right) \\
& \hspace*{2cm}= M^{-1} Df(\bc)  M \begin{pmatrix}x\\ y\\ z\end{pmatrix} + M^{-1} \left(f \left(\bc + M \begin{pmatrix}x\\ y\\ z\end{pmatrix}\right) - Df(\bc) M \begin{pmatrix}x\\ y\\ z\end{pmatrix} \right) \\
& \hspace*{2cm}=
\begin{pmatrix} 0 \\ -\Lambda y \\ \Lambda z \end{pmatrix} +
M^{-1}
\begin{pmatrix}
-x^2 \\
0 \\
- (d-1) x \Gamma \Lambda (-y + z) - \mathbf{N}(c+\Gamma (y + z)) + D \mathbf{N}(c)\Gamma( y+ z)
\end{pmatrix},
\end{align*}
the transformed ``normal'' form of \eqref{eq:autonomous_ode}, under $w = \bc + M \begin{pmatrix}x\\ y\\ z\end{pmatrix}$, reads
\begin{equation}\label{eq:diag_autonomous_ode}
\frac{d}{d r} \begin{pmatrix} x \\ y \\ z \end{pmatrix} =
\begin{pmatrix}
-x^2 \\
-\Lambda y + \psi(x, y, z) \\
\hspace{0.25cm}\Lambda z - \psi(x, y, z) \\
\end{pmatrix},
\end{equation}
where $x \in \R$ is the {\em center variable}, $y,z \in \mathbb{C}^q$, and
\begin{equation}
\psi(x, y, z) \bydef
\frac{1}{2}\Big(
(d-1) x (-y + z) +
\Lambda^{-1} \Gamma^{-1}\Big( \mathbf{N}(c+\Gamma (y +  z)) - D\mathbf{N}(c)\Gamma (y + z) \Big)
\Big) \in \mathbb{C}^q.
\end{equation}

The domain of the local graph of the desired center-stable manifold is a bounded connected subset of $[0, \infty) \times \mathbb{C}^q$ containing the stationary point $0$ of \eqref{eq:diag_autonomous_ode}. We note that the positive $x$-axis is invariant, corresponding to the equilibrium state $u=c$.
The function space used to parameterize the local center-stable manifold in the form $z = \alpha(x, y)$ is given by
\begin{align*}
\mathcal{G}_{\delta, \mu, \mathcal{L}_x, \mathcal{L}_y} \bydef \Big\{ \alpha : \{(x, y) \in [0, \infty) \times \mathbb{C}^q \, :\, \, x \in [0, \delta],\, |y|_{\mathbb{C}^q} \le &\mu\} \to \mathbb{C}^q \, : \\
&\alpha(x, 0) = 0, \\
&|\alpha(x_1, y) - \alpha(x_2, y)|_{\mathbb{C}^q} \le \mathcal{L}_x |y|_{\mathbb{C}^q} |x_1 - x_2|, \\
&|\alpha(x, y_1) - \alpha(x, y_2)|_{\mathbb{C}^q} \le \mathcal{L}_y |y_1 - y_2|_{\mathbb{C}^q} \hspace*{0.7cm} \smash{\Big\}}.
\end{align*}
The positive parameters $\delta$ and $\mu$ controlling the size of the chart, as well as the Lipschitz constants $\mathcal{L}_y$ and the ``mixed term'' Lipschitz constant
$\mathcal{L}_x$, will be chosen later, and the choice will depend on the application. 
Notably, $\mathcal{G}_{\delta, \mu, \mathcal{L}_x, \mathcal{L}_y}$ equipped with the norm
\[
|\alpha|_{\mathcal{G}_{\delta, \mu, \mathcal{L}_x, \mathcal{L}_y}} \bydef \sup_{\begin{smallmatrix}x \in [0, \delta],\\ 0 < |y|_{\mathbb{C}^q} \le \mu\end{smallmatrix}} \frac{|\alpha(x, y)|_{\mathbb{C}^q}}{|y|_{\mathbb{C}^q}} ,
\]
is a Banach space (cf. Chapter 4.1, Proposition A of \cite{Chicone}).

\begin{remark}
Throughout this paper, we always endow $\mathbb{C}^q$ with the infinity norm, which we keep denoting by $| y |_{\mathbb{C}^q} \bydef \max_{i=1,\dots,q} |y_i|$.
Also, we denote by $\mathscr{B}(X, Y)$ the set of bounded linear operators from $X$ to $Y$ and by $| \cdot |_{\mathscr{B}(X, Y)}$ the induced operator norm.
\end{remark}

In the following, we will need estimates on the linear flow. To this end, defining
\begin{equation}\label{eq:lambda_hat}
\hat{\lambda} \bydef \min_{i = 1, \dots, q} \Re(\lambda_i),
\end{equation}
we recall that
\[
|e^{- \Lambda r}|_{\mathscr{B}(\mathbb{C}^q,\mathbb{C}^q)} \le e^{- \hat{\lambda} r},
\qquad \text{for all } r \ge 0.
\]

Under a set of explicit constraints on $\delta, \mu, \mathcal{L}_x, \mathcal{L}_y$, the following proposition guarantees that $\mathcal{G}_{\delta, \mu, \mathcal{L}_x, \mathcal{L}_y}$ contains the local graph of the center-stable manifold of the stationary point $0$ of \eqref{eq:diag_autonomous_ode}. The trajectories in this forward invariant manifold converge to $0$ as $r$ tends to $\infty$.

\begin{proposition}\label{prop:graph}
Let $\delta, \mu, \mathcal{L}_x, \mathcal{L}_y > 0$. Let $\hat{\psi}=\hat{\psi}(\mu,\mathcal{L}_y)>0$ be a bound satisfying
\begin{equation}\label{eq:psi_hat}
\hat{\psi} \ge \frac{1}{2} \sup_{|y|_{\mathbb{C}^q} \le \mu, \, |z|_{\mathbb{C}^q} \le \mathcal{L}_y \mu} \left|
\Lambda^{-1} \Gamma^{-1}\Bigl( D\mathbf{N}(c+\Gamma (y +  z)) - D\mathbf{N}(c)\Bigr)\Gamma
\right|_{\mathscr{B}(\mathbb{C}^q,\mathbb{C}^q)}.
\end{equation}
Assume that the following three inequalities hold:
\begin{subequations}
\label{eq:threeconstraints}
\begin{align}
\label{eq:graph_constraint1}
\hat{\lambda}  &>   (\tfrac{d-1}{2}\delta + \hat{\psi})(1+\mathcal{L}_y),
\\
\label{eq:graph_constraint2}
\mathcal{L}_x &\ge \left(
\frac{1}{2\hat{\lambda} -  (\tfrac{d-1}{2}\delta + \hat{\psi})(1 + \mathcal{L}_y)} + \frac{\left( \tfrac{d-1}{2}\delta + \hat{\psi} \right) (1 + \mathcal{L}_y)}{(2\hat{\lambda} - (\tfrac{3(d-1)}{2} \delta + 2\hat{\psi})(1 + \mathcal{L}_y)) (2\hat{\lambda} - (\tfrac{d-1}{2} \delta + \hat{\psi})(1 + \mathcal{L}_y))}
\right) \nonumber \\
&\hspace*{1cm}\times
\left(\tfrac{d-1}{2} (1 + \mathcal{L}_y) + \left( \tfrac{d-1}{2} \delta + \hat{\psi} \right)\mathcal{L}_x \right),
\\
\label{eq:graph_constraint3}
\mathcal{L}_y &\ge \frac{ (\tfrac{d-1}{2} \delta + \hat{\psi}) (1 + \mathcal{L}_y)}{2\hat{\lambda} - (\tfrac{d-1}{2} \delta + \hat{\psi}) (1 + \mathcal{L}_y)} .
\end{align}
\end{subequations}
Then, $\mathcal{G}_{\delta, \mu, \mathcal{L}_x, \mathcal{L}_y}$ contains the local graph $\tilde{\alpha}$ of a center-stable manifold of the stationary point $0$ of \eqref{eq:diag_autonomous_ode}. Additionally, for any initial condition $(x(0), y(0), z(0)) = (\xi, \eta, \tilde{\alpha}(\xi, \eta))$ such that $\xi \in [0, \delta]$ and $|\eta|_{\mathbb{C}^q} \le \mu$, there exists a unique solution $(x(r), y(r), z(r))$ of \eqref{eq:diag_autonomous_ode} for all $r \ge 0$ which satisfies $z(r) = \tilde{\alpha}(x(r), y(r))$, for all $r \ge 0$, and $\lim_{r \to \infty} (x(r), y(r), z(r)) = 0$.
\end{proposition}

\begin{proof}
We follow the proof strategy of~\cite[\S 4.1]{Chicone} adapted to our context.
Given a function $\alpha \in \mathcal{G}_{\delta, \mu, \mathcal{L}_x, \mathcal{L}_y}$, we denote by $(x(r; \xi), y_\alpha(r; \xi, \eta))$ the unique solution of the initial value problem
\[
\left\{
\begin{aligned}
&\frac{dx}{d r}  = -x^2, \\
&\frac{dy}{d r}  = -\Lambda y + \psi (x, y, \alpha(x, y)), \\
&(x(0), y(0)) = (\xi, \eta),
\end{aligned}
\right.
\]
where $\xi \in [0, \delta]$ and $|\eta|_{\mathbb{C}^q} \le \mu$.

The proof essentially consists in showing that the Lyapunov-Perron operator, defined formally by
\begin{equation}\label{e:defLP}
[\LP(\alpha)](\xi, \eta) \bydef \int_0^{\infty} e^{- \Lambda r} \psi \Bigl(x(r; \xi), y_\alpha(r; \xi, \eta), \alpha \bigl(x(r; \xi), y_\alpha(r; \xi, \eta)\bigr)\Bigr) \, dr,
\end{equation}
satisfies the assumptions of the Banach Fixed-Point Theorem on $\mathcal{G}_{\delta, \mu, \mathcal{L}_x, \mathcal{L}_y}$.

On the one hand, the flow of $x$ is known exactly to be (with initial condition $x(0)=\xi \ge 0$)
\begin{equation}\label{e:xexplicit}
x(r; \xi) = \frac{\xi}{1 + \xi r} \ge 0.
\end{equation}
In principle, one could carry on this expression throughout the estimates. However, the resulting expressions involve exponential integrals and are laborious to work with in practice. Henceforth, we resort to the bounds
\[
x(r; \xi) \le \xi \qquad \text{and} \qquad
|x(r; \xi_1) - x(r; \xi_2)| \le |\xi_1 - \xi_2|, \qquad \text{for all } \xi, \xi_1, \xi_2 \in [0, \delta].
\]
On the other hand, the flow of the $y$-variable (with initial condition $x(0)=\xi$ and $y(0)=\eta$) is only accessible through an implicit variation of constant formula
\begin{equation}\label{eq:variation_of_constant}
y_\alpha (r; \xi, \eta) =
e^{-\Lambda r} \eta + \int_0^r e^{-\Lambda(r-s)} \psi \Bigl( x(s; \xi), y_\alpha (s; \xi, \eta), \alpha \bigl(x(s; \xi), y_\alpha (s; \xi, \eta)\bigr) \Bigr) \, ds.
\end{equation}

We will now derive a series of estimates which will help to establish that $\LP$ is a contraction on $\mathcal{G}_{\delta, \mu, \mathcal{L}_x, \mathcal{L}_y}$.
We start with an estimate on the change in $\psi$ when taking different initial conditions $(\xi_1,\eta_1)$ and $(\xi_2,\eta_2)$ for the same (but arbitrary) $\alpha \in \mathcal{G}_{\delta, \mu, \mathcal{L}_x, \mathcal{L}_y}$.
Indeed, an application of the Mean Value Theorem yields, after some bookkeeping and by using also the triangle inequality, that, as long as $|y_\alpha(r,\xi_i,\eta_i)|_{\mathbb{C}^q} \leq \mu$ for $i=1,2$, we have
\begin{align}
&\Bigl|\psi\Bigl(x(r; \xi_1), y_\alpha(r; \xi_1, \eta_1), \alpha\bigl(x(r; \xi_1), y_\alpha(r; \xi_1, \eta_1)\bigr)\Bigr) - \psi\Bigl(x(r; \xi_2), y_\alpha(r; \xi_2, \eta_2), \alpha\bigl(x(r; \xi_2), y_\alpha(r; \xi_2, \eta_2)\bigr)\Bigr) \Bigr|_{\mathbb{C}^q} \nonumber \\
&\hspace*{0.5cm}\le \tfrac{d-1}{2} (1 + \mathcal{L}_y) 
% \max \bigl\{|y_\alpha(r; \xi_1, \eta_1)|_{\mathbb{C}^q}, |y_\alpha(r; \xi_2, \eta_2)|_{\mathbb{C}^q} \bigr\}
|y_\alpha(r; \xi_2, \eta_2)|_{\mathbb{C}^q}
 |x(r; \xi_1) - x(r; \xi_2)|  \nonumber \\
&\hspace*{1cm} + \left(\tfrac{d-1}{2} \delta + \hat{\psi}\right) |y_\alpha(r; \xi_1, \eta_1) - y_\alpha(r; \xi_2, \eta_2)|_{\mathbb{C}^q} \nonumber \\
&\hspace*{1.5cm} + \left(\tfrac{d-1}{2} \delta + \hat{\psi}\right) |\alpha(x(r; \xi_1), y_\alpha(r; \xi_1, \eta_1)) - \alpha(x(r; \xi_2), y_\alpha(r; \xi_2, \eta_2))|_{\mathbb{C}^q} \nonumber \\
&\hspace*{0.5cm}\le \tfrac{d-1}{2} (1 + \mathcal{L}_y) 
% \max \bigl\{|y_\alpha(r; \xi_1, \eta_1)|_{\mathbb{C}^q}, |y_\alpha(r; \xi_2, \eta_2)|_{\mathbb{C}^q}\bigr\} 
|y_\alpha(r; \xi_2, \eta_2)|_{\mathbb{C}^q}
|x(r; \xi_1) - x(r; \xi_2)| \nonumber \\
&\hspace*{1cm} + \left(\tfrac{d-1}{2} \delta + \hat{\psi}\right) |y_\alpha(r; \xi_1, \eta_1) - y_\alpha(r; \xi_2, \eta_2)|_{\mathbb{C}^q} \nonumber \\
&\hspace*{1.5cm} + \left(\tfrac{d-1}{2} \delta + \hat{\psi}\right) \left( \mathcal{L}_x |y_\alpha(r; \xi_2, \eta_2)|_{\mathbb{C}^q} |x(r; \xi_1) - x(r; \xi_2)| + \mathcal{L}_y |y_\alpha(r; \xi_1, \eta_1) - y_\alpha(r; \xi_2, \eta_2)|_{\mathbb{C}^q} \right) \nonumber \\
&\hspace*{0.5cm}\le \left(\tfrac{d-1}{2} (1 + \mathcal{L}_y) + \left( \tfrac{d-1}{2} \delta + \hat{\psi} \right) \mathcal{L}_x \right) 
% \max \bigl\{|y_\alpha(r; \xi_1, \eta_1)|_{\mathbb{C}^q}, |y_\alpha(r; \xi_2, \eta_2)|_{\mathbb{C}^q}\bigr\} 
|y_\alpha(r; \xi_2, \eta_2)|_{\mathbb{C}^q}
|\xi_1 - \xi_2| \nonumber \\
&\hspace*{1cm} + \left(\tfrac{d-1}{2} \delta + \hat{\psi}\right) (1 + \mathcal{L}_y) \, |y_\alpha(r; \xi_1, \eta_1) - y_\alpha(r; \xi_2, \eta_2)|_{\mathbb{C}^q} .
\label{e:yxieta1xieta2}
\end{align}

Similarly, we estimate the change in $\psi$ when taking different  $\alpha_1,\alpha_2 \in \mathcal{G}_{\delta, \mu, \mathcal{L}_x, \mathcal{L}_y} $ with the same initial condition $(\xi,\eta)$. Again, an application of the Mean Value Theorem implies that, as long as $|y_{\alpha_i}(r,\xi,\eta)|_{\mathbb{C}^q} \leq \mu$ for $i=1,2$, we have 
\begin{align}
&\Bigl|\psi\Bigl(\xi, y_{\alpha_1}(r; \xi, \eta), \alpha_1\bigl(\xi, y_{\alpha_1}(r; \xi, \eta)\bigr)\Bigr) - \psi\Bigl(\xi, y_{\alpha_2}(r; \xi, \eta), \alpha_2\bigl(\xi, y_{\alpha_2}(r; \xi, \eta)\bigr)\Bigr)\Bigr|_{\mathbb{C}^q} \nonumber \\
&\hspace*{0.5cm}\le
\left( \tfrac{d-1}{2} \delta + \hat{\psi} \right) |y_{\alpha_1}(r; \xi, \eta) - y_{\alpha_2}(r; \xi, \eta)|_{\mathbb{C}^q}  \nonumber \\
&\hspace*{1cm} + \left( \tfrac{d-1}{2} \delta + \hat{\psi} \right) |\alpha_1(\xi, y_{\alpha_1}(r; \xi, \eta)) - \alpha_2(\xi, y_{\alpha_2}(r; \xi, \eta))|_{\mathbb{C}^q} \nonumber \\
&\hspace*{0.5cm}\le
\left( \tfrac{d-1}{2} \delta + \hat{\psi} \right) |y_{\alpha_1}(r; \xi, \eta) - y_{\alpha_2}(r; \xi, \eta)|_{\mathbb{C}^q}  \nonumber \\
&\hspace*{1cm} + \left( \tfrac{d-1}{2} \delta + \hat{\psi} \right) \left( \mathcal{L}_y |y_{\alpha_1}(r; \xi, \eta) - y_{\alpha_2}(r; \xi, \eta)|_{\mathbb{C}^q} + |\alpha_1 - \alpha_2|_{\mathcal{G}_{\delta, \mu, \mathcal{L}_x, \mathcal{L}_y}} |y_{\alpha_2}(r; \xi, \eta)|_{\mathbb{C}^q}\right) \nonumber \\
&\hspace*{0.5cm}=
\left( \tfrac{d-1}{2} \delta + \hat{\psi} \right) \left( (1 + \mathcal{L}_y) |y_{\alpha_1}(r; \xi, \eta) - y_{\alpha_2}(r; \xi, \eta)|_{\mathbb{C}^q} + |\alpha_1 - \alpha_2|_{\mathcal{G}_{\delta, \mu, \mathcal{L}_x, \mathcal{L}_y}} |y_{\alpha_2}(r; \xi, \eta)|_{\mathbb{C}^q} \right).
\label{e:yalpha1alpha2}
\end{align}

Then, from the variation of constant formula \eqref{eq:variation_of_constant}, we have
\begin{align}
e^{\hat{\lambda} r} |y_\alpha (r; \xi, \eta)|_{\mathbb{C}^q}
&\le |\eta|_{\mathbb{C}^q} + \int_0^r e^{\hat{\lambda}s} \Bigl|\psi \Bigl( x(s; \xi), y_\alpha (s; \xi, \eta), \alpha \bigl(x(s; \xi), y_\alpha (s; \xi, \eta)\bigr) \Bigr)\Bigr|_{\mathbb{C}^q} \, ds 
\label{eq:rearranged} \\
&\le |\eta|_{\mathbb{C}^q} + \left(\frac{d-1}{2}\delta + \hat{\psi}\right)(1 + \mathcal{L}_y) \int_0^r e^{\hat{\lambda}s} |y_\alpha (r; \xi, \eta)|_{\mathbb{C}^q} \, ds, \nonumber
\end{align}
where the last inequality follows from \eqref{e:yxieta1xieta2} with $\xi_1 = \xi_2 = \xi$, $\eta_1 = \eta$, $\eta_2 = 0$, so that $y_\alpha(r; \xi_2, \eta_2) = 0$ since $\alpha \in\mathcal{G}_{\delta, \mu, \mathcal{L}_x, \mathcal{L}_y}$ and $\psi (x, 0, 0) = 0$ for any $x \in [0,\delta]$.
By Grönwall's inequality, we obtain
\begin{equation}\label{eq:bound_y}
|y_\alpha (r; \xi, \eta)|_{\mathbb{C}^q} \le |\eta|_{\mathbb{C}^q} e^{\left((\frac{d-1}{2}\delta + \hat{\psi})(1 + \mathcal{L}_y) - \hat{\lambda} \right) r}.
\end{equation}
Therefore,
\begin{equation}\label{eq:bound_integral_1}
\int_0^{\infty} e^{-\hat{\lambda} r} |y_\alpha(r; \xi, \eta)|_{\mathbb{C}^q} \, dr
\le \frac{1}{2\hat{\lambda} - (\tfrac{d-1}{2} \delta + \hat{\psi})(1 + \mathcal{L}_y)} |\eta|_{\mathbb{C}^q}.
\end{equation}
Due to the constraint~\eqref{eq:graph_constraint1}, 
we conclude from~\eqref{eq:bound_y} that $y_\alpha$ decays to $0$ at exponential rate, and $|y_\alpha (r; \xi, \eta)|_{\mathbb{C}^q} \leq |\eta|_{\mathbb{C}^q} \leq \mu$ for all $r \geq 0$. We infer that the estimates~\eqref{e:yxieta1xieta2} and~\eqref{e:yalpha1alpha2} hold for all $r \geq 0$. It then follows that $\psi\Bigl(\xi, y_{\alpha}(r; \xi, \eta), \alpha\bigl(\xi, y_{\alpha}(r; \xi, \eta)\bigr)\Bigr)$ tends to $0$ as $r \to \infty$, hence the Lyapunov-Perron operator $\LP$ in~\eqref{e:defLP} is well-defined.

Next, we estimate the difference between two solutions of
$\frac{dy}{d r}  = -\Lambda y + \psi (x, y, \alpha(x, y))$ with different initial conditions. We first take initial data with different values for $\xi$ only.
From the variation of constant formula \eqref{eq:variation_of_constant} we have 
\begin{align*}
&e^{\hat{\lambda} r}|y_\alpha(r; \xi_1, \eta) - y_\alpha(r; \xi_2, \eta)|_{\mathbb{C}^q}\\
&\hspace*{1cm}\le \int_0^r e^{\hat{\lambda} s} \Bigl[
\left(\tfrac{d-1}{2} (1 + \mathcal{L}_y) + \left( \tfrac{d-1}{2} \delta + \hat{\psi} \right) \mathcal{L}_x \right) 
% \max \bigl\{|y_\alpha(r; \xi_1, \eta)|_{\mathbb{C}^q}, |y_\alpha(r; \xi_2, \eta)|_{\mathbb{C}^q}\bigr\}
|y_\alpha(r; \xi_2, \eta)|_{\mathbb{C}^q}
 |\xi_1 - \xi_2| \\
&\hspace*{2cm} + \left(\tfrac{d-1}{2} \delta + \hat{\psi}\right) (1 + \mathcal{L}_y) \, |y_\alpha(r; \xi_1, \eta) - y_\alpha(r; \xi_2, \eta)|_{\mathbb{C}^q} 
\Bigr] \, ds\\
&\hspace*{1cm}\le 
\left(\tfrac{d-1}{2} (1 + \mathcal{L}_y) + \left( \tfrac{d-1}{2} \delta + \hat{\psi} \right) \mathcal{L}_x \right) |\eta|_{\mathbb{C}^q} |\xi_1 - \xi_2| \frac{e^{(\frac{d-1}{2}\delta + \hat{\psi})(1 + \mathcal{L}_y) r} - 1}{(\frac{d-1}{2}\delta + \hat{\psi})(1 + \mathcal{L}_y)} \\
&\hspace*{2cm} + \left(\tfrac{d-1}{2} \delta + \hat{\psi}\right) (1 + \mathcal{L}_y) \int_0^r e^{\hat{\lambda} s} |y_\alpha(r; \xi_1, \eta) - y_\alpha(r; \xi_2, \eta)|_{\mathbb{C}^q} \, ds,
\end{align*}
where the first inequality follows from \eqref{e:yxieta1xieta2} with $\eta_1 = \eta_2 = \eta$ and the second inequality is inferred from \eqref{eq:bound_y}.
By Grönwall's inequality, we obtain
\begin{align*}
&|y_\alpha(r; \xi_1, \eta) - y_\alpha(r; \xi_2, \eta)|_{\mathbb{C}^q} \nonumber\\
&\hspace*{1cm}\le
\left(\tfrac{d-1}{2} (1 + \mathcal{L}_y) + \left( \tfrac{d-1}{2} \delta + \hat{\psi} \right) \mathcal{L}_x \right) |\eta|_{\mathbb{C}^q} |\xi_1 - \xi_2| \frac{e^{(\frac{d-1}{2}\delta + \hat{\psi})(1 + \mathcal{L}_y) r} - 1}{(\frac{d-1}{2}\delta + \hat{\psi})(1 + \mathcal{L}_y)} e^{\left((\frac{d-1}{2} \delta + \hat{\psi}) (1 + \mathcal{L}_y) - \hat{\lambda}\right)r}.
\end{align*}
Therefore,
\begin{align}\label{eq:bound_integral_2}
&\int_0^{\infty} e^{-\hat{\lambda} r} |y_\alpha(r; \xi_1, \eta) - y_\alpha(r; \xi_2, \eta)|_{\mathbb{C}^q} \, dr \nonumber \\
&\hspace*{1cm} \le \frac{\tfrac{d-1}{2} (1 + \mathcal{L}_y) + \left( \tfrac{d-1}{2} \delta + \hat{\psi} \right)\mathcal{L}_x}{(2\hat{\lambda} -  (\tfrac{3(d-1)}{2} \delta + 2\hat{\psi})(1 + \mathcal{L}_y)) (2\hat{\lambda} -  (\tfrac{d-1}{2} \delta + \hat{\psi})(1 + \mathcal{L}_y))} |\eta|_{\mathbb{C}^q} |\xi_1 - \xi_2|.
\end{align}

Similarly, for initial data with different values for $\eta$, the variation of constant formula~\eqref{eq:variation_of_constant} gives
\begin{align*}
&e^{\hat{\lambda} r}|y_\alpha(r; \xi, \eta_1) - y_\alpha(r; \xi, \eta_2)|_{\mathbb{C}^q}\\
&\hspace*{1cm}\le |\eta_1 - \eta_2|_{\mathbb{C}^q} + \left(\tfrac{d-1}{2} \delta + \hat{\psi}\right) (1 + \mathcal{L}_y) \int_0^r e^{\hat{\lambda} s} |y_\alpha(r; \xi, \eta_1) - y_\alpha(r; \xi, \eta_2)|_{\mathbb{C}^q} \, ds,
\end{align*}
where the inequality follows from \eqref{e:yxieta1xieta2} with $\xi_1 = \xi_2 = \xi$.
By Grönwall's inequality, we obtain
\begin{equation*}
|y_\alpha(r; \xi, \eta_1) - y_\alpha(r; \xi, \eta_2)|_{\mathbb{C}^q}
\le
|\eta_1 - \eta_2|_{\mathbb{C}^q} e^{\left((\frac{d-1}{2} \delta + \hat{\psi}) (1 + \mathcal{L}_y) - \hat{\lambda}\right)r}.
\end{equation*}
Therefore,
\begin{equation}\label{eq:bound_integral_3}
\int_0^{\infty} e^{-\hat{\lambda} r} |y_\alpha(r; \xi, \eta_1) - y_\alpha(r; \xi, \eta_2)|_{\mathbb{C}^q} \, dr
\le \frac{1}{2\hat{\lambda} -  (\tfrac{d-1}{2} \delta + \hat{\psi}) (1 + \mathcal{L}_y)} |\eta_1 - \eta_2|_{\mathbb{C}^q}.
\end{equation}

Repeating this process one last time for two solutions with the same initial data but different $\alpha$, from the variation of constant formula \eqref{eq:variation_of_constant} we have that
\begin{align*}
e^{\hat{\lambda} r}|y_{\alpha_1}(r; \xi, \eta) - y_{\alpha_2}(r; \xi, \eta)|_{\mathbb{C}^q}
&\le \int_0^r e^{\hat{\lambda} s}
\left( \tfrac{d-1}{2} \delta + \hat{\psi} \right) \Bigl[ (1 + \mathcal{L}_y) |y_{\alpha_1}(r; \xi, \eta) - y_{\alpha_2}(r; \xi, \eta)|_{\mathbb{C}^q} \\
&\hspace*{5cm}+ |\alpha_1 - \alpha_2|_{\mathcal{G}_{\delta, \mu, \mathcal{L}_x, \mathcal{L}_y}} |y_{\alpha_2}(r; \xi, \eta)|_{\mathbb{C}^q} \Bigr] \, ds\\
&\le \left( \tfrac{d-1}{2} \delta + \hat{\psi} \right) \frac{e^{(\frac{d-1}{2}\delta + \hat{\psi})(1 + \mathcal{L}_y) r} - 1}{(\frac{d-1}{2}\delta + \hat{\psi})(1 + \mathcal{L}_y)} |\alpha_1 - \alpha_2|_{\mathcal{G}_{\delta, \mu, \mathcal{L}_x, \mathcal{L}_y}} |\eta|_{\mathbb{C}^q} \\
&\hspace*{1.5cm}+
\left( \tfrac{d-1}{2} \delta + \hat{\psi} \right) (1 + \mathcal{L}_y) \int_0^r e^{\hat{\lambda} s} |y_{\alpha_1}(r; \xi, \eta) - y_{\alpha_2}(r; \xi, \eta)|_{\mathbb{C}^q} \, ds,
\end{align*}
where the first inequality follows from \eqref{e:yalpha1alpha2} and the second inequality is inferred from \eqref{eq:bound_y}.
By Grönwall's inequality, we obtain
\begin{align*}
&|y_{\alpha_1}(r; \xi, \eta) - y_{\alpha_2}(r; \xi, \eta)|_{\mathbb{C}^q} \\
&\hspace{1cm}\le
\left( \tfrac{d-1}{2} \delta + \hat{\psi} \right) |\alpha_1 - \alpha_2|_{\mathcal{G}_{\delta, \mu, \mathcal{L}_x, \mathcal{L}_y}} |\eta|_{\mathbb{C}^q} \frac{e^{(\frac{d-1}{2}\delta + \hat{\psi})(1 + \mathcal{L}_y) r} - 1}{(\frac{d-1}{2}\delta + \hat{\psi})(1 + \mathcal{L}_y)} e^{\left((\frac{d-1}{2} \delta + \hat{\psi}) (1 + \mathcal{L}_y) - \hat{\lambda}\right)r}.
\end{align*}
Therefore,
\begin{align}\label{eq:bound_integral_4}
&\int_0^{\infty} e^{-\hat{\lambda} r} |y_{\alpha_1}(r; \xi, \eta) - y_{\alpha_2}(r; \xi, \eta)|_{\mathbb{C}^q} \, dr \nonumber \\
&\qquad \le \frac{\tfrac{d-1}{2}\delta + \hat{\psi}}{(2\hat{\lambda} -  (\tfrac{3(d-1)}{2} \delta + 2\hat{\psi})(1 + \mathcal{L}_y)) (2\hat{\lambda} -  (\tfrac{d-1}{2} \delta + \hat{\psi})(1 + \mathcal{L}_y))} |\alpha_1 - \alpha_2|_{\mathcal{G}_{\delta, \mu, \mathcal{L}_x, \mathcal{L}_y}} |\eta|_{\mathbb{C}^q}.
\end{align}

We are now ready to show that $\LP$ is a contraction on $\mathcal{G}_{\delta, \mu, \mathcal{L}_x, \mathcal{L}_y}$. We start by establishing that $\LP$ maps $\mathcal{G}_{\delta, \mu, \mathcal{L}_x, \mathcal{L}_y}$ to itself.
Since $\alpha \in\mathcal{G}_{\delta, \mu, \mathcal{L}_x, \mathcal{L}_y}$ and $\psi (x, 0, 0) = 0$, we have that $y_\alpha = 0$ whenever $\eta = 0$. It follows that
\[
[\LP(\alpha)](\xi, 0) = \int_0^{\infty} e^{-\Lambda r} \psi (x(r; \xi), 0, \alpha(x(r; \xi), 0)) \, dr = \int_0^{\infty} e^{-\Lambda r} \psi (x(r; \xi), 0, 0) \, dr = 0.
\]
Furthermore, we have that
\begin{align*}
 &\bigl|[\LP(\alpha)](\xi_1, \eta) - [\LP(\alpha)](\xi_2, \eta)\bigr|_{\mathbb{C}^q} \\
 &\hspace*{0.5cm}\le
\left(\tfrac{d-1}{2} (1 + \mathcal{L}_y) + \left(\tfrac{d-1}{2} \delta + \hat{\psi}\right) \mathcal{L}_x \right) |\xi_1 - \xi_2| \int_0^{\infty} e^{-\hat{\lambda} r} 
% \max \bigl\{ |y_\alpha(r; \xi_1, \eta)|_{\mathbb{C}^q}, |y_\alpha(r; \xi_2, \eta)|_{\mathbb{C}^q} \bigr\} 
|y_\alpha(r; \xi_2, \eta)|_{\mathbb{C}^q}
\, dr \\
 &\hspace*{1.5cm} + \left(\tfrac{d-1}{2} \delta + \hat{\psi}\right) (1 + \mathcal{L}_y) \int_0^{\infty} e^{-\hat{\lambda} r} |y_\alpha(r; \xi_1, \eta) - y_\alpha(r; \xi_2, \eta)|_{\mathbb{C}^q} \, dr
  \\
&\hspace*{0.5cm}\le 
\left(
\frac{1}{2\hat{\lambda} - (\tfrac{d-1}{2}\delta + \hat{\psi})(1 + \mathcal{L}_y)} +  \frac{\left( \tfrac{d-1}{2}\delta + \hat{\psi} \right) (1 + \mathcal{L}_y)}{(2\hat{\lambda} - (\tfrac{3(d-1)}{2} \delta + 2\hat{\psi})(1 + \mathcal{L}_y)) (2\hat{\lambda} - (\tfrac{d-1}{2} \delta + \hat{\psi})(1 + \mathcal{L}_y))}
\right) \\
& \hspace*{1.5cm} \times
\left(\tfrac{d-1}{2} (1 + \mathcal{L}_y) + \left( \tfrac{d-1}{2} \delta + \hat{\psi} \right)\mathcal{L}_x \right)
 |\eta|_{\mathbb{C}^q} |\xi_1 - \xi_2| \\
&\hspace*{0.5cm}\le \mathcal{L}_x |\eta|_{\mathbb{C}^q} |\xi_1 - \xi_2|,
\end{align*}
where the first inequality follows from \eqref{e:yxieta1xieta2} with $\eta_1 = \eta_2 = \eta$, the second inequality follows from \eqref{eq:bound_integral_1} and \eqref{eq:bound_integral_2}, while the last inequality follows from the constraint \eqref{eq:graph_constraint2}. Similarly,
\begin{align*}
\bigl[\LP(\alpha)](\xi, \eta_1) - [\LP(\alpha)](\xi, \eta_2)\bigr|_{\mathbb{C}^q} 
&\le \left(\tfrac{d-1}{2} \delta + \hat{\psi}\right) (1 + \mathcal{L}_y) \int_0^{\infty} e^{-\hat{\lambda} r} |y_\alpha(r; \xi, \eta_1) - y_\alpha(r; \xi, \eta_2)|_{\mathbb{C}^q} \, dr \\
&\le \frac{(\tfrac{d-1}{2} \delta + \hat{\psi}) (1 + \mathcal{L}_y)}{2\hat{\lambda} - (\tfrac{d-1}{2} \delta + \hat{\psi}) (1 + \mathcal{L}_y)} |\eta_1 - \eta_2|_{\mathbb{C}^q} \\
&\le\mathcal{L}_y |\eta_1 - \eta_2|_{\mathbb{C}^q},
\end{align*}
\noindent where the first inequality follows from \eqref{e:yxieta1xieta2} with $\xi_1 = \xi_2 = \xi$, the second inequality follows from \eqref{eq:bound_integral_3} and the last inequality follows from the constraint \eqref{eq:graph_constraint3}. Hence, $\LP$ maps $\mathcal{G}_{\delta, \mu, \mathcal{L}_x, \mathcal{L}_y}$ to itself.

Next, we prove that $\LP$ is a contraction on $\mathcal{G}_{\delta, \mu, \mathcal{L}_x, \mathcal{L}_y}$:
\begin{align*}
&\bigl| [\LP(\alpha_1)](\xi, \eta) - [\LP(\alpha_2)](\xi, \eta) \bigr|_{\mathbb{C}^q} \\
&\hspace*{0.5cm}\le \left( \tfrac{d-1}{2} \delta + \hat{\psi} \right) 
(1 + \mathcal{L}_y) \int_0^{\infty} e^{-\hat{\lambda} r} |y_{\alpha_1}(r; \xi, \eta) - y_{\alpha_2}(r; \xi, \eta)|_{\mathbb{C}^q} \, dr  \\
&\hspace*{1cm}+ \left( \tfrac{d-1}{2} \delta + \hat{\psi} \right)  |\alpha_1 - \alpha_2|_{\mathcal{G}_{\delta, \mu, \mathcal{L}_x, \mathcal{L}_y}} \int_0^{\infty} e^{-\hat{\lambda} r} |y_{\alpha_2}(r; \xi, \eta)|_{\mathbb{C}^q} \, dr
 \\
&\hspace*{0.5cm}\le 
\left(
\frac{\left( \tfrac{d-1}{2}\delta + \hat{\psi} \right) (1 + \mathcal{L}_y)}{(2\hat{\lambda} - (\tfrac{3(d-1)}{2} \delta + 2\hat{\psi})(1 + \mathcal{L}_y)) (2\hat{\lambda} - (\tfrac{d-1}{2} \delta + \hat{\psi})(1 + \mathcal{L}_y))} + \frac{1}{2\hat{\lambda} - (\tfrac{d-1}{2}\delta + \hat{\psi})(1 + \mathcal{L}_y)}
\right) \\
&\hspace*{1cm}\times 
\left(\tfrac{d-1}{2} \delta + \hat{\psi} \right)
|\alpha_1 - \alpha_2|_{\mathcal{G}_{\delta, \mu, \mathcal{L}_x, \mathcal{L}_y}} |\eta|_{\mathbb{C}^q} \\
&\hspace*{0.5cm}
\leq 
\left[ 1-\frac{\tfrac{d-1}{2} (1+\mathcal{L}_y)}{\mathcal{L}_x} \right] 
|\alpha_1 - \alpha_2|_{\mathcal{G}_{\delta, \mu, \mathcal{L}_x, \mathcal{L}_y}} |\eta|_{\mathbb{C}^q},
\end{align*}
where the first inequality follows from \eqref{e:yalpha1alpha2}, the second inequality is inferred from \eqref{eq:bound_integral_1} and \eqref{eq:bound_integral_4}, while the last inequality follows from \eqref{eq:graph_constraint2}.

Therefore, we can apply the Banach Fixed-Point Theorem to $\LP$, which yields a unique fixed-point $\tilde{\alpha} \in \mathcal{G}_{\delta, \mu, \mathcal{L}_x, \mathcal{L}_y}$.
Since
\begin{align*}
\frac{d}{d r} \Bigl(e^{- \Lambda r} \talpha \bigl(x(r; \xi), &y_\talpha(r; \xi, \eta)\bigr)\Bigr) \\
&= \frac{d}{d r} \Bigl(e^{- \Lambda r} [\LP(\talpha)](x(r; \xi), y_\talpha(r; \xi, \eta)) \Bigr) \\
&= \frac{d}{d r} \int_0^{\infty} e^{- \Lambda (s + r)} \psi \Bigl(x(s + r; \xi), y_\talpha(s + r; \xi, \eta), \talpha \bigl(x(s + r; \xi), y_\talpha(s + r; \xi, \eta)\bigr)\Bigr) \, ds \\
&= \frac{d}{d r} \int_r^{\infty} e^{- \Lambda s} \psi \Bigl(x(s; \xi), y_\talpha(s; \xi, \eta), \talpha \bigl(x(s; \xi), y_\talpha(s; \xi, \eta)\bigr)\Bigr) \, ds \\
&= - e^{- \Lambda r} \psi \Bigl(x(r; \xi), y_\talpha(r; \xi, \eta), \talpha \bigl(x(r; \xi), y_\talpha(r; \xi, \eta)\bigr)\Bigr),
\end{align*}
we have that 
\[
\frac{d}{d r} \talpha \bigl(x(r; \xi), y_\talpha(r; \xi, \eta)\bigr) = \Lambda \talpha \bigl(x(r; \xi), y_\talpha(r; \xi, \eta)\bigr) - \psi \Bigl(x(r; \xi), y_\talpha(r; \xi, \eta), \talpha \bigl(x(r; \xi), y_\talpha(r; \xi, \eta)\bigr)\Bigr).
\]
Hence $\left(x(r; \xi), y_\talpha(r; \xi, \eta), \talpha\bigl(x(r; \xi), y_\talpha(r; \xi, \eta) \bigr)\right)$ solves~\eqref{eq:diag_autonomous_ode} for all $r \ge 0$, which shows that the local graph of $\talpha$ is an invariant manifold.
As mentioned before, it follows from~\eqref{eq:bound_y} and the constraint~\eqref{eq:graph_constraint1} that $y_\talpha (r; \xi, \eta)$ decays to $0$ as $r \to \infty$. Likewise, $\lim_{r \to \infty} x(r; \xi) = 0$ by~\eqref{e:xexplicit}. Finally, by continuity $\lim_{r \to \infty} z(r)=\lim_{r \to \infty} \talpha\bigl(x(r; \xi), y_\talpha (r; \xi, \eta)\bigr) = \talpha(0, 0) = 0$.
\end{proof}

\begin{remark}
The conditions \eqref{eq:graph_constraint1}, \eqref{eq:graph_constraint2}, \eqref{eq:graph_constraint3} connect the size of the domain $\delta, \mu$ and the Lipschitz constants $\mathcal{L}_x, \mathcal{L}_y$ of the local graph with the smallest eigenvalue $\hat{\lambda}$ and the growth of the nonlinearity $\hat{\psi}$.
Both $\hat{\lambda}$ and $\hat{\psi}$ depend on the problem at hand, namely the nonlinear term $\mathbf{N}$ in \eqref{eq:elliptic}.

To fix ideas, let us consider $\mathcal{L}_x, \mathcal{L}_y \leq 1$ only.
Since $\hat{\lambda}$ is bounded away from zero and $\hat{\psi}$ is small for small $\mu$, the right hand sides in~\eqref{eq:threeconstraints} are small for small $\delta,\mu$. We can thus choose $\mathcal{L}_x, \mathcal{L}_y$ sufficiently large for all three inequalities~\eqref{eq:threeconstraints} to hold. The lower bounds on proper choices for $\mathcal{L}_x, \mathcal{L}_y$ tend to $0$ as $\delta,\mu \to 0$, reflecting the tangency of the center-stable manifold to the $(x,y)$-plane at the stationary point $0$ of \eqref{eq:diag_autonomous_ode}.
\end{remark}

In conclusion of this section, we highlight that Proposition~\ref{prop:graph} not just provides bounds on the local graph of the center-stable manifold, but also describes the dynamics inside this invariant manifold: it is comprised of solutions of~\eqref{eq:diag_autonomous_ode} which converge to the stationary point~$0$. Therefore, a localized solution of~\eqref{eq:radial_elliptic}, limiting to the  equilibrium~$c$, yields a trajectory in this local center-stable manifold.

%%%%%%%%%%%%%%%%%%%%%
%% BOUNDARY VALUE PROBLEM  %%
%%%%%%%%%%%%%%%%%%%%%

\section{The Newton-Kantorovich argument} \label{sec:newton-kantorovich}
%!TEX root = radial_elliptic.tex

We begin this section by assuming that there exist $\delta, \mu, \mathcal{L}_x, \mathcal{L}_y > 0$ such that the hypotheses of Proposition~\ref{prop:graph} are verified, yielding the existence of a local graph $\talpha \in \mathcal{G}_{\delta, \mu, \mathcal{L}_x, \mathcal{L}_y}$ of a center-stable manifold of the stationary point $0$ of \eqref{eq:diag_autonomous_ode}. As presented in Section~\ref{sec:zero-finding-problem}, this allows introducing a nonlinear zero-finding problem whose solution corresponds to a localized radial solution of \eqref{eq:elliptic}. Constructive existence of a zero of $F$, defined in \eqref{eq:zero-finding-problem}, is proven using a Newton-Kantorovich argument which we now briefly overview, see Section~\ref{subsec:newton-kantorovich} for details.

The starting point of the argument is a general nonlinear map $F : (X,|\, \cdot \,|_X) \to (Y,|\, \cdot \,|_Y)$ for which a proof of existence of a zero is desired. Considering a finite dimensional projection, one computes a numerical approximation $\bx$ such that $F(\bx) \approx 0$. An injective linear operator $\mathcal{A}$ is then constructed such that $\mathcal{A}F:X \to X$ and such that $| I - \mathcal{A} DF(\bx)|_{\mathscr{B}(X,X)}<1$ (hence, $\mathcal{A}$ serves as an approximate inverse for the Fr\'echet derivative $DF(\bx)$). Next, one defines the Newton-like operator $T(\chi) \bydef \chi -\mathcal{A}F(\chi)$, and proves that there exists a radius $\bar{\rho}>0$ such that $T:\textnormal{cl}(B_{\bar{\rho}}(\bx)) \to \textnormal{cl}(B_{\bar{\rho}}(\bx))$ is a contraction, where $\textnormal{cl}(B_{\bar{\rho}}(\bx))$ is the closed ball of radius $\bar{\rho}$ centered at $\bx$. The Banach Fixed-Point Theorem yields the existence of a unique $\tx \in \textnormal{cl}(B_{\bar{\rho}}(\bx))$ such that $\tx = T(\tx) = \tx-\mathcal{A}F(\tx)$. By injectivity of $\mathcal{A}$, the fixed point $\tx$ of $T$ is the unique zero of $F$ in $\textnormal{cl}(B_{\bar{\rho}}(\bx))$. For this to work in practice, 
sufficient and computable, explicit conditions for the verification that $T(\textnormal{cl}(B_{\bar{\rho}}(\bx)) \subset \textnormal{cl}(B_{\bar{\rho}}(\bx))$ and that $T$ is a contraction need to be derived, see the Newton-Kantorovich Theorem~\ref{thm:radii_polynomial}. The hypotheses of this theorem are finally verified with the computer, yielding a (computer-assisted) proof of existence. 

The above general framework falls in the field of computer-assisted proofs (also called {\em rigorous computations} or {\em rigorous numerics}) in nonlinear analysis. Examples of early pioneering works in the field of computer-assisted proofs in dynamics is the proof of the universality of the Feigenbaum constant \cite{feigenbaum} and the proof of existence of the strange attractor in the Lorenz system \cite{lorenz}. Several computer-assisted proofs of existence of solutions to PDEs have also been presented in the last decades, including eigenvalues enclosure methods \cite{MR2019251,MR1151060}, self-consistent a priori bounds \cite{MR2049869,MR1838755}, a priori error estimates for finite element approximations combined with the Schauder fixed point theorem \cite{MR2161437}, and topological methods based on Conley index theory \cite{MR2136516,MR2329522}. We refer the interested reader to the survey papers \cite{KOCH_ComputeAssisted, NAKAO_VerifiedPDE, TUCKER_ValidatedIntroduction, VANDENBERG_Dynamics,GOMEZ_PDESurvey}, as well as the recent book \cite{MR3971222}.

Let us now introduce a zero-finding problem of the form $F(\,\cdot\,;\talpha)=0$ whose solution corresponds to a localized radial solution of \eqref{eq:elliptic}. 

\subsection{Formulation of the zero-finding problem} \label{sec:zero-finding-problem}

For the formulation of the zero-finding problem we will split the domain $[0,\infty)$ of the function $u$ into subintervals.
First, near $r=0$ we set out to solve \eqref{eq:radial_elliptic} by taking advantage of the analyticity of the solution. Let $\ell>0$ represent an anticipated (lower bound on the) radius of convergence of the power series representation of the localized radial solution (this is confirmed when we find a solution with rescaled Taylor coefficients in the sequence space $\mathscr{T}$ defined in~\eqref{e:defscrT}).
We rescale $[0,\ell]$ to $[0,1]$ by setting $u(r)=v(\ell^{-1} r)$ and write
\[
v(r) \bydef \{v\}_0 + \sum_{n \ge 2} \{v\}_n r^n, \qquad \text{for all } r \in [0, 1].
\]
Solving \eqref{eq:radial_elliptic} with the initial condition $(u,\frac{d}{d r} u)(0)=(\phi ,0)$ for $\phi \in \mathbb{R}^q$, is equivalent to finding the Taylor coefficients $\{v\}_n \in \R^q$, for $n \ge 0$, satisfying
\begin{equation}\label{eq:taylor}
\begin{cases}
\{v\}_0 - \phi = 0, \\
\{v\}_1 = 0, \\
n (n + d - 2) \{v\}_n + \ell^2 \{\mathbf{N}(v)\}_{n-2} = 0, & n \ge 2.
\end{cases}
\end{equation}
We will solve for the sequence of Taylor coefficients in 
\begin{equation}\label{e:defscrT}
\mathscr{T} \bydef \biggl\{ a \in \mathbb{C}^{\mathbb{N}\cup\{0\}} \, : \, |a|_\mathscr{T} \bydef \sum_{n \ge 0} |\{a\}_n| < \infty \biggr\},
\end{equation}
which is a Banach algebra with the Cauchy product
\begin{equation}\label{eq:cauchy_prod}
a *_\mathscr{T} b \bydef \left\{ \sum_{m = 0}^n \{a\}_{n-m} \{b\}_m \right\}_{n \ge 0}.
\end{equation}

We denote $w = (w^{(1)},w^{(2)},w^{(3)})$ with $w^{(1)}=w_1$, $w^{(2)} = (w_2, \dots, w_{1+q})$ and $w^{(3)} = (w_{2+q} ,\dots,w_{1+2q})$. Transforming the coordinates of the ``normal'' form \eqref{eq:diag_autonomous_ode} to the variables $(w^{(1)},w^{(2)},w^{(3)})$, the local graph of the center-stable manifold of the equilibrium $\bc$ of \eqref{eq:autonomous_ode} is given by
\[
(\xi, \eta) \mapsto
\begin{pmatrix}
0 \\
c \\
0
\end{pmatrix} +
\begin{pmatrix}
1 & 0 & 0 \\
0 & \Gamma & \Gamma \\
0 & - \Gamma \Lambda & \hspace{0.25cm}\Gamma \Lambda
\end{pmatrix}
\begin{pmatrix} \xi \\ \eta \\ \talpha(\xi, \eta) \end{pmatrix} ,
\]
where $\talpha \in \mathcal{G}_{\delta, \mu, \mathcal{L}_x, \mathcal{L}_y}$ was obtained in Section~\ref{sec:graph_enclosure}.
In principle, there could exist $r_* \in (0, 1)$ such that $v(r_*)$ already connects to this local graph, but in general one should not expect this.
Thus, we solve \eqref{eq:autonomous_ode} on $[\ell r_*, \ell r_* + L]$ for some $L>0$, subject to the initial condition $(w^{(1)},w^{(2)},w^{(3)})(\ell r_*)=((\ell r_*)^{-1}, v(r_*), \ell^{-1} \frac{d}{d r} v(r_*))$ to ensure that the solution connects smoothly to $v$.
In particular, since $w^{(1)}(r) = \frac{\xi}{\xi (r - r_0) + 1}$ where $w^{(1)}(r_0) = \xi$, it follows from our initial condition $w^{(1)} (\ell r_*) = (\ell r_*)^{-1}$ that $w^{(1)} (r) = \frac{(\ell r_*)^{-1}}{(\ell r_*)^{-1} (r - \ell r_*) + 1} = r^{-1}$ as desired.

Although the first intuition might be to set $r_* = 1$, we will need to choose $r_* < 1$ to guarantee differentiability at $r=r_*$ of the power series $v$ with coefficients in $\mathscr{T}$.
In a nutshell, having fixed a truncation order $n_\mathscr{T} \ge 2$ for the power series $v$, this requires bounding the sequence $n r_*^n$ for all $n > n_\mathscr{T}$; this sequence is strictly decreasing when $r_* \in (0, e^{-1/(n_\mathscr{T}+1)}]$.

For the boundary value problem for $w$ we favour, 
in the spirit of \cite{MR3148084}, the use of Chebyshev polynomials of the first kind given by
\begin{equation}\label{eq:chebyshevpolynomials}
  \cheb_n(s) = \cos( n \arccos(s)), \qquad n=0,1,2,\dots \quad\text{and}\quad  s \in [-1,1].
\end{equation}
Since these are defined on $[-1,1]$, we rescale the domain of $w$ from $[\ell r_*, \ell r_* + L]$ to $[-1,1]$.
Integrating~\eqref{eq:autonomous_ode} using the initial condition mentioned above, yields
\begin{equation}\label{eq:integraleq}
\begin{pmatrix}(\ell r_*)^{-1} \\ v(r_*) \\ \displaystyle \ell^{-1} \frac{d}{d r} v(r_*)\end{pmatrix} + \frac{L}{2} \int_{-1}^s f(w(s')) \, ds' - w(s) = 0, \qquad \text{for all } s \in [-1, 1].
\end{equation}
The rescaling implies that $w^{(2)}(s)=u(\ell r_* + \frac{s+1}{2} L)$ for $s \in [-1,1]$.
Expanding $w$ as a Chebyshev series
\[
w(s) \bydef \{w\}_0 + 2 \sum_{n \ge 1} \{w\}_n \cheb_n (s), \qquad \textnormal{for all } s \in [-1, 1],
\]
the integral equation~\eqref{eq:integraleq} is equivalent to 
\begin{equation}\label{eq:cheb}
\left\{
\begin{aligned}
\begin{pmatrix}
(\ell r_*)^{-1} \\ \sum_{m \ge 0} \{v\}_m r_*^m \\ \sum_{m \ge 1} \frac{m}{\ell} \{v\}_m r_*^m
\end{pmatrix} + \frac{L}{2} \Bigl(\{f(w)\}_0 - \frac{1}{2}\{f(w)\}_1 - 2 \sum_{m \ge 2} \frac{(-1)^m}{m^2 -1} \{f(w)\}_m \Bigr) - \{w\}_0 = 0,
& \qquad n = 0 \\
\frac{L}{2} \frac{\{f(w)\}_{n-1} - \{f(w)\}_{n+1}}{2n} - \{w\}_n = 0, & \qquad n \ge 1.
\end{aligned}
\right.
\end{equation}
We look for Chebyshev coefficients $\{w_i\}$, $i=1,\dots,1+2q$ solving~\eqref{eq:cheb} in the sequence space (for some weight $\nu>1$)
\[
\mathscr{C}_\nu \bydef \left\{ a \in \mathbb{C}^{\mathbb{N}\cup\{0\}} \, : \, |a|_{\mathscr{C}_\nu} \bydef |\{a\}_0| + 2 \sum_{n \ge 1} |\{a\}_n| \nu^n < \infty \right\}.
\]
This is a Banach algebra with the discrete convolution product 
\begin{equation}\label{eq:discrete_conv}
a *_\mathscr{C} b \bydef \left\{ \sum_{m \in \mathbb{Z}} \{a\}_{|n-m|} \{b\}_{|m|} \right\}_{n \ge 0},
\end{equation}
which corresponds to the natural convolution in Fourier space through~\eqref{eq:chebyshevpolynomials}.

\begin{remark}
We expanded $u$ as a Taylor series on $[0, \ell r_*]$ to  deal with the non-autonomous term $r^{-1}$ in \eqref{eq:radial_elliptic}, which is singular at the origin. 
As seen from the set of equations \eqref{eq:taylor}, the apparent singularity, which is naturally caused by the radial coordinate frame, is entirely removed by taking the first–order coefficient of the power series equal to zero.
\end{remark}

At $s=1$ (i.e. $r=\ell r_*+L$) we need the solution to lie in the local center-stable manifold, hence 
\begin{equation}\label{eq:boundary_condition}
\begin{pmatrix} \{w^{(2)}\}_0 + 2\sum_{n \ge 1} \{w^{(2)}\}_n \\ \{w^{(3)}\}_0 + 2\sum_{n \ge 1} \{w^{(3)}\}_n \end{pmatrix} - \begin{pmatrix}
c \\
0
\end{pmatrix} - \begin{pmatrix}
\Gamma & \Gamma \\
-\Gamma \Lambda & \hspace{0.25cm}\Gamma \Lambda
\end{pmatrix} \begin{pmatrix}
\eta \\ \talpha\bigl((\ell r_* + L)^{-1},\eta\bigr)
\end{pmatrix} = 0,
\end{equation}
with $w^{(1)}(1)=(\ell r_*+L)^{-1}$ being satisfied automatically since $w^{(1)}(1)=(\ell r_*)^{-1}$.
Consequently, when $\{v\}$ and $\{w\}$ solve~\eqref{eq:taylor}, \eqref{eq:cheb} and \eqref{eq:boundary_condition}, then the function
\[
u(r) \bydef
\begin{cases}
\displaystyle\sum_{n \ge 0} \{v\}_n (\ell^{-1}r)^n, & r \in [0, \ell r_*], \\
\displaystyle\{w^{(2)}\}_0 + 2 \sum_{n \ge 1} \{ w^{(2)} \}_n \cheb_n (\tfrac{2}{L} (r - \ell r_*) - 1), & r \in (\ell r_*, \ell r_* + L],
\end{cases}
\]
represents a localized radial solution of \eqref{eq:elliptic}: its orbit connects at $r=\ell r_* +L$ to the center-stable manifold of $c$ and converges to $c$ as $r \to \infty$ by Proposition~\ref{prop:graph}.

To formulate the above in the framework of a single zero-finding problem, we define, for some $\nu > 1$, the product Banach space
\[
X \bydef \mathbb{C}^q \times \mathbb{C}^q \times \mathscr{T}^q \times \mathscr{C}^{1 + 2q}_\nu,
\]
equipped with the norm
\[
|\chi|_X \bydef \max \Bigl\{ |\eta|_{\mathbb{C}^q}, |\phi|_{\mathbb{C}^q}, \max_{i=1,\dots,q} |v_i|_\mathscr{T}, \max_{i=1,\dots,1+2q} |w_i|_{\mathscr{C}_\nu} \Bigr\},
\qquad \text{for all } \chi = (\eta, \phi, v, w) \in X.
\]

Given $\talpha \in \mathcal{G}_{\delta, \mu, \mathcal{L}_x, \mathcal{L}_y}$, the local graph of a center-stable manifold of the stationary point $0$ of \eqref{eq:diag_autonomous_ode}, we define $F(\,\cdot\,; \talpha)$ on $X$, without specifying the co-domain, as
\begin{equation} \label{eq:zero-finding-problem}
F (\chi; \talpha) \bydef
\begin{pmatrix}
\{w^{(2)}\}_0 + 2\sum_{n \ge 1 } \{w^{(2)}\}_n - c - \Gamma (\eta + \talpha((\ell r_* + L)^{-1},\eta)) \\
\{w^{(3)}\}_0 + 2\sum_{n \ge 1 } \{w^{(3)}\}_n - \Gamma \Lambda (-\eta + \talpha((\ell r_* + L)^{-1},\eta)) \\
g(\chi) \\
h(\chi)
\end{pmatrix},
\end{equation}
where $g(\chi) \in (\mathbb{C}^{1+2q})^{\mathbb{N}\cup\{0\}}$ is given by
\begin{align*}
\{g\}_n &\bydef
\left\{ 
\begin{aligned}
\{v\}_0 - \phi, &\qquad n = 0, \\
\{v\}_1, &\qquad n = 1, \\
n (n + d - 2) \{v\}_n + \ell^2 \{\mathbf{N}(v)\}_{n-2}, &\qquad n \ge 2,
\end{aligned} 
\right.
\end{align*}
and $h(\chi) \in (\mathbb{C}^{1+2q})^{\mathbb{N}\cup\{0\}}$ is given by
\begin{align*}
\{h\}_n &\bydef 
\left\{ 
\begin{aligned}
\begin{pmatrix}
(\ell r_*)^{-1} \\ \sum_{m \ge 0} \{v\}_m r_*^m \\ \sum_{m \ge 1} \frac{m}{\ell} \{v\}_m r_*^m
\end{pmatrix} + \frac{L}{2} \Bigl(\{f(w)\}_0 - \frac{1}{2}\{f(w)\}_1 - 2 \sum_{m \ge 2} \frac{(-1)^m}{m^2 -1} \{f(w)\}_m \Bigr) - \{w\}_0, & \qquad n =0, 
\\
\frac{L}{2} \frac{\{f(w)\}_{n-1} - \{f(w)\}_{n+1}}{2n} - \{w\}_n, & \qquad n \ge 1.
\end{aligned}
\right.
\end{align*}
Here,
\begin{itemize}
\item $\mathbf{N}$ is understood as a mapping from $\mathscr{T}^q$ to itself by identifying each multiplication with the Cauchy product $*_\mathscr{T}$ defined in \eqref{eq:cauchy_prod};
\item $f$ is understood as a mapping from $\mathscr{C}_\nu^{1+2q}$ to itself by identifying each multiplication with the discrete convolution $*_\mathscr{C}$ defined in \eqref{eq:discrete_conv}.
\end{itemize}
Furthermore, we introduce the notation $F(\,\cdot\,; 0)$ to mean the mapping $F(\,\cdot\,; \talpha)$ where the terms involving $\talpha$ are removed.

To conduct the computer-assisted proof, we will approximate a finite number of Taylor and Chebyshev coefficients of the solution numerically. Formally, given $n_\textnormal{max} \in \mathbb{N} \cup \{0\}$, we define the truncation operator $\pi^{n_\textnormal{max}} : \mathbb{C}^{\mathbb{N} \cup \{0\}} \to \mathbb{C}^{\mathbb{N} \cup \{0\}}$ by
\[
\{ \pi^{n_\textnormal{max}} a \}_n \bydef
\begin{cases}
\{a\}_n, & n \le n_\textnormal{max}, \\
0, & n > n_\textnormal{max},
\end{cases} \qquad \text{for all } a \in \mathbb{C}^{\mathbb{N} \cup \{0\}}.
\]
Given $n_\mathscr{T}, n_\mathscr{C} \in \mathbb{N} \cup \{0\}$, this operator extends in a natural fashion to $\mathscr{T}^q$, $\mathscr{C}_\nu^{1+2q}$ and $X$ as follows:
\begin{alignat*}{2}
\pi^{n_\mathscr{T}}v &\bydef (\pi^{n_\mathscr{T}} v_1, \dots, \pi^{n_\mathscr{T}} v_q), \qquad &&\text{for all } v = (v_1, \dots, v_q) \in \mathscr{T}^q, \\
\pi^{n_\mathscr{C}}w &\bydef (\pi^{n_\mathscr{C}} w_1, \dots, \pi^{n_\mathscr{C}} w_{1+2q}), \qquad &&\text{for all } w = (w_1, \dots, w_{1+2q}) \in \mathscr{C}_\nu^{1+2q}, \\
\pi^{n_\mathscr{T}, n_\mathscr{C}}\chi &\bydef (\eta, \phi, \pi^{n_\mathscr{T}} v, \pi^{n_\mathscr{C}} w), \qquad &&\text{for all } \chi = (\eta, \phi, v, w) \in X.
\end{alignat*}
The complementary operators representing the tail of the coefficients are denoted by
\begin{alignat*}{2}
\pi^{\infty(n_\textnormal{max})}a &\bydef a - \pi^{n_\textnormal{max}} a, \qquad &&\text{for all } a \in \mathbb{C}^{\mathbb{N} \cup \{0\}}, \\
\pi^{\infty(n_\mathscr{T})}v &\bydef (\pi^{\infty(n_\mathscr{T})} v_1, \dots, \pi^{\infty(n_\mathscr{T})} v_q), \qquad &&\text{for all } v = (v_1, \dots, v_q) \in \mathscr{T}^q, \\
\pi^{\infty(n_\mathscr{C})}w &\bydef (\pi^{\infty(n_\mathscr{C})} w_1, \dots, \pi^{\infty(n_\mathscr{C})} w_{1+2q}), \qquad &&\text{for all } w = (w_1, \dots, w_{1+2q}) \in \mathscr{C}_\nu^{1+2q}, \\
\pi^{\infty(n_\mathscr{T}), \infty(n_\mathscr{C})}\chi &\bydef (0, 0, \pi^{\infty(n_\mathscr{T})} v, \pi^{\infty(n_\mathscr{C})} w), \qquad &&\text{for all } \chi = (\eta, \phi, v, w) \in X.
\end{alignat*}

\subsection{The Newton-Kantorovich Theorem} \label{subsec:newton-kantorovich}

In this section, we formulate the Newton-Kantorovich Theorem which gives sufficient conditions to prove the existence of a localized radial solution of \eqref{eq:elliptic}. In the estimates we will need the monomial coefficients of the polynomials, hence we fix their notation in the next remark.

\begin{remark}
Recall that $\mathbf{N} = (\mathbf{N}_1, \dots, \mathbf{N}_q) : \mathscr{T}^q \to \mathscr{T}^q$ and $f = (f_1, \dots, f_{1+2q}) : \mathscr{C}^{1+2q}_\nu \to \mathscr{C}^{1+2q}_\nu$ are polynomials of the same order $\order \geq 2$. That is to say, there exist coefficients $a^{(i)}_{k_1,\dots,k_q}$ and $b^{(i)}_{k_1,\dots,k_q}$ such that
\begin{align*}
\mathbf{N}_i (v) &= \sum_{\begin{smallmatrix}k_1, \dots, k_q \in \mathbb{N} \cup \{0\} \\ 0 \le k_1 + \ldots + k_q \le \order \end{smallmatrix}} a^{(i)}_{k_1,\dots,k_q} v_1^{k_1} *_\mathscr{T} \ldots *_\mathscr{T} v_q^{k_q}, &&\text{for } v \in \mathscr{T}^q,\quad 1 \leq i \leq q, \\
f_i (w) &= \sum_{\begin{smallmatrix}k_1, \dots, k_{1+2q} \in \mathbb{N} \cup \{0\} \\ 0 \le k_1 + \ldots + k_q \le \order \end{smallmatrix}} b^{(i)}_{k_1,\dots,k_{1+2q}} w_1^{k_1} *_\mathscr{C} \ldots *_\mathscr{C} w_{1+2q}^{k_{1+2q}}, &&\text{for } w \in \mathscr{C}^{1+2q}_\nu, \quad 1 \leq i \leq 1+2q .
\end{align*}
We denote by $\mathbf{N}_\textnormal{abs} = (\mathbf{N}_{\textnormal{abs}, 1}, \dots, \mathbf{N}_{\textnormal{abs}, q}) : \mathbb{R}^q \to \mathbb{R}^q$ and $f_\textnormal{abs} = (f_{\textnormal{abs}, 1}, \dots, f_{\textnormal{abs}, 1+2q}) : \mathbb{R}^{1+2q} \to \mathbb{R}^{1+2q}$ the polynomials given by
\begin{align*}
\mathbf{N}_{\textnormal{abs},i} (\zeta) &= \sum_{\begin{smallmatrix}k_1, \dots, k_q \in \mathbb{N} \cup \{0\} \\ 0 \le k_1 + \ldots + k_q \le \order \end{smallmatrix}} \bigl|a^{(i)}_{k_1,\dots,k_q}\bigr| 
\textstyle\prod\limits_{j=1}^{q} \zeta_j^{k_j},
&& \text{for } \zeta \in\mathbb{R}^q,\quad 1 \leq i \leq q,
\\
f_{\textnormal{abs},i} (\zeta) &= \sum_{\begin{smallmatrix}k_1, \dots, k_{1+2q} \in \mathbb{N} \cup \{0\} \\ 0 \le k_1 + \ldots + k_q \le \order \end{smallmatrix}} \bigl|b^{(i)}_{k_1,\dots,k_{1+2q}}\bigr| 
\textstyle\prod\limits_{j=1}^{1+2q}
\zeta_j^{k_j},
&& \text{for } \zeta \in \mathbb{R}^{1+2q}, \quad 1 \leq i \leq 1+2q.
\end{align*}
\end{remark}

We denote by
\[
\textnormal{cl}(B_\rho(\bx)) \bydef
\{
\chi \in X \, : \, |\chi-\bx|_X \le \rho
\}
\]
the closure of the ball of radius $\rho \ge 0$ in $X$, centered at $\bx \in X$.
In the computer-assisted proof it is natural to take $\bx \in \pi^{n_\mathscr{T}, n_\mathscr{C}} X$ for some choice of $n_\mathscr{T}, n_\mathscr{C} \in \mathbb{N}$. 

\begin{theorem}[\bf Newton-Kantorovich]\label{thm:radii_polynomial}
Denote by $\order \geq 2$ the order of the polynomial $\mathbf{N}$ and fix $\ell, L > 0$, $n_\mathscr{T} \ge 2$, $n_\mathscr{C} \ge 1$, $r_* \in (0, e^{-1 / (n_\mathscr{T} + 1)}]$, $\nu > 1$, $\mathcal{L}_x, \mathcal{L}_y > 0$ and $\varrho > 0$. Let $\bx = (\bar{\xi}, \bar{\eta}, \bar{\phi}, \bar{v}, \bar{w}) \in \pi^{n_\mathscr{T}, n_\mathscr{C}} X$ and let $A : \pi^{n_\mathscr{T}, n_\mathscr{C}} X \to \pi^{n_\mathscr{T}, n_\mathscr{C}} X$ be an injective linear operator.

Suppose Proposition \ref{prop:graph} holds true for $\delta = (\ell r_* + L)^{-1}, \mu = |\bar{\eta}|_{\mathbb{C}^q} + \varrho$ and the chosen Lipschitz constant $\mathcal{L}_x, \mathcal{L}_y$ such that $\mathcal{G}_{\delta, \mu, \mathcal{L}_x, \mathcal{L}_y}$ contains the local graph, denoted by $\talpha$, of the center-stable manifold of the stationary point $0$ of \eqref{eq:diag_autonomous_ode}.

Define
\begin{alignat*}{1}
Y & \bydef | A \pi^{n_\mathscr{T}, n_\mathscr{C}} F (\bx; 0) |_X  \\
& \hspace*{1.5cm}+ \max \biggl\{
\frac{\ell^2 | \pi^{\infty(n_\mathscr{T}-2)} \mathbf{N}(\bar{v}) |_{\mathscr{T}^q}}{(n_\mathscr{T}+1)(n_\mathscr{T}+d-1)},
\frac{L(\nu + \nu^{-1})| \pi^{\infty(n_\mathscr{C}-1)} f(\bar{w}) |_{\mathscr{C}_\nu^{1+2q}}}{4(n_\mathscr{C}+1)}
 \biggr\} ,\\
Z_1 &\bydef
| \pi^{n_\mathscr{T}, n_\mathscr{C}} - A \pi^{n_\mathscr{T}, n_\mathscr{C}} DF (\bx; 0) \pi^{n_\mathscr{T}, \order n_\mathscr{C} + 1} |_{\mathscr{B}(X, X)}  \\
& \hspace*{1.5cm}+ \max \biggl\{ \frac{\ell^2 |D\mathbf{N}(\bar{v})|_{\mathscr{B}(\mathscr{T}^q, \mathscr{T}^q)}}{(n_\mathscr{T}+1)(n_\mathscr{T}+d-1)}, \frac{L (\nu + \nu^{-1}) | Df(\bar{w})|_{\mathscr{B}(\mathscr{C}_\nu^{1+2q}, \mathscr{C}_\nu^{1+2q})}}{4(n_\mathscr{C}+1)} \biggr\}  \\
& \hspace*{1.5cm}+| A |_{\mathscr{B}(X, X)} \max  \biggl\{ \frac{2}{\nu^{\order n_\mathscr{C}+2}}, r_*^{n_\mathscr{T} + 1}\max\{ 1, \ell^{-1}(n_\mathscr{T} + 1)\} + \frac{L| Df(\bar{w})|_{\mathscr{B}(\mathscr{C}_\nu^{1+2q}, \mathscr{C}_\nu^{1+2q})}}{\nu^{n_\mathscr{C} + 2}((n_\mathscr{C} + 2)^2 - 1)} \biggr\},\\
Z_2 &\bydef
\left(| A |_{\mathscr{B}(X, X)} + 1\right)
\max\Biggl\{ \ell^2 |D^2 \mathbf{N}_\textnormal{abs}(|\bar{v}_1|_\mathscr{T}+\varrho, \dots, |\bar{v}_q|_\mathscr{T}+\varrho)|_{\mathscr{B}(\mathbb{C}^{q^2}, \mathbb{C}^q)}, \\
& \hspace*{4cm} {\frac{L(1 + \nu)}{2}} |D^2 f_\textnormal{abs}(|\bar{w}_1|_{\mathscr{C}_\nu}+\varrho, \dots, |\bar{w}_{1+2q}|_{\mathscr{C}_\nu}+\varrho)|_{\mathscr{B}(\mathbb{C}^{(1+2q)^2}, \mathbb{C}^{1+2q})}\Biggr\}.
\end{alignat*}
If there exists $\bar{\rho} \in [0, \varrho]$ such that the two inequalities
\begin{subequations}
\begin{align}
\begin{split}\label{eq:radii_constraint1}
Y + |A (\Gamma, \Gamma \Lambda, 0, 0)|_{\mathscr{B}(\mathbb{C}^q, X)} \mathcal{L}_y |\bar{\eta}|_{\mathbb{C}^q} - (1 - Z_1 - |A (\Gamma, \Gamma \Lambda, 0, 0)|_{\mathscr{B}(\mathbb{C}^q, X)} \mathcal{L}_y)\bar{\rho} + \frac{Z_2}{2} \bar{\rho}^2 &\le 0,
\end{split} \\
\begin{split}\label{eq:radii_constraint2}
|A (\Gamma, \Gamma \Lambda, 0, 0)|_{\mathscr{B}(\mathbb{C}^q, X)} \mathcal{L}_y + Z_1 + Z_2 \bar{\rho} &< 1,
\end{split}
\end{align}
\end{subequations}
hold, then there exists a unique $\tx \in \textnormal{cl} (B_{\bar{\rho}} (\bx))$ such that $F(\tx; \talpha)=0$.
\end{theorem}

\begin{proof}
Consider the bounded linear operator $\mathcal{A} \bydef A \pi^{n_\mathscr{T}, n_\mathscr{C}} + \mathcal{A}_\infty$ with $\mathcal{A}_\infty : X \to \pi^{\infty(n_\mathscr{T}), \infty(n_\mathscr{C})} X$ given by
\[
\mathcal{A}_\infty \chi \bydef
(0, 0, v_\infty, -\pi^{\infty(n_\mathscr{C})}w), \qquad \text{for all } \chi = (\eta, \phi, v, w) \in X,
\]
\noindent where
\[
\{v_\infty\}_n \bydef
\begin{cases}
0 & \text{for } n \le n_\mathscr{T}, \\
\displaystyle \frac{\{v\}_n}{n(n + d -2)} & \text{for }  n > n_\mathscr{T}.
\end{cases}
\]

Define $T(\,\cdot\, ; \talpha) : \textnormal{cl} (B_\varrho (\bx)) \to X$ by
\[
T(\chi; \talpha) \bydef \chi - \mathcal{A} F(\chi; \talpha), \qquad \text{for all } \chi \in \textnormal{cl} (B_\varrho (\bx)).
\]
\noindent This operator is well-defined since, for any $\chi = (\eta, \phi, v, w) \in \textnormal{cl} (B_\varrho (\bx))$, we have that $|\eta|_{\mathbb{C}^q} \le |\bar{\eta}|_{\mathbb{C}^q} + \varrho = \mu$ and it follows that $\talpha(\delta, \eta)$ is well-defined. By construction, specifically the action of $\mathcal{A}_\infty$, the operator $\mathcal{A}F(\,\cdot\, ; \talpha)$ maps $\textnormal{cl} (B_\varrho (\bx))$ into $X$. Moreover, $T(\chi ; 0) \bydef \chi - \mathcal{A} F(\chi; 0)$ is twice Fr\'echet differentiable for all $\chi \in \textnormal{cl} (B_\varrho (\bx))$.

Let $\bar{\rho} \in [0, \varrho]$ be such that \eqref{eq:radii_constraint1} and \eqref{eq:radii_constraint2} hold. It is enough to show that $T(\,\cdot\, ; \talpha)$ satisfies the assumptions of the Banach Fixed-Point Theorem in $\textnormal{cl} (B_{\bar{\rho}} (\bx))$. Indeed, $\mathcal{A}$ is injective in view of its definition and injectivity of the matrix $A$, which implies that a fixed-point of $T(\,\cdot\, ; \talpha)$ is a zero of $F(\,\cdot\, ; \talpha)$.

We prove in Appendix~\ref{appendix:YZ1Z2} that
\begin{subequations}
\begin{align}
\begin{split}\label{eq:Y}
Y &\ge |T (\bx; 0) - \bx|_X,
\end{split}\\
\begin{split}\label{eq:Z_1}
Z_1 &\ge |D T (\bx; 0)|_{\mathscr{B}(X, X)},
\end{split}\\
\begin{split}\label{eq:Z_2}
Z_2 &\ge \sup_{\chi \in \textnormal{cl}(B_\varrho(\bx))} |D^2 T (\chi; 0)|_{\mathscr{B}(X, \mathscr{B}(X, X))}.
\end{split}
\end{align}
\end{subequations}
Let $\chi = (\eta, \phi, v, w) \in \textnormal{cl} (B_{\bar{\rho}} (\bx))$. From Taylor's Theorem we have that
\begin{align*}
&|T (\chi; 0) - \bx|_X \\
&\quad= | T (\bx; 0) - \bx + [DT(\bx; 0)](\chi - \bx) + \int_0^1 (1 - t) [D^2 T(\bx + t(\chi - \bx); 0)] (\chi - \bx, \chi - \bx) \, dt |_X \\
&\quad\le | T (\bx; 0) - \bx |_X + |[D T(\bx; 0)](\chi - \bx)|_X + \int_0^1 (1 - t) |[D^2 T(\bx + t(\chi - \bx); 0)] (\chi - \bx, \chi - \bx)|_X \, dt \\
&\quad\le Y + Z_1 \bar{\rho} + Z_2 \bar{\rho}^2 \int_0^1 (1 - t)\, dt \\
&\quad= Y + Z_1 \bar{\rho} + \frac{Z_2}{2} \bar{\rho}^2 .
\end{align*}
Since $F(\chi;\alpha)$ is affine linear in $\alpha$, 
it follows from \eqref{eq:radii_constraint1} that
\begin{align*}
|T (\chi; \talpha) - \bx|_X
&= |T (\chi; 0) - \bx + \mathcal{A}(F(\chi; 0) - F(\chi; \talpha)|_X \\
&\le |T (\chi; 0) - \bx|_X + |\mathcal{A}(F(\chi; 0) - F(\chi; \talpha)|_X \\
&= |T (\chi; 0) - \bx|_X + |A
\begin{pmatrix}
\Gamma \talpha(\delta, \eta) , \Gamma \Lambda \talpha(\delta, \eta) , 0  , 0
\end{pmatrix}
|_X \\
&\le Y + Z_1 \bar{\rho} + \tfrac{1}{2} Z_2 \bar{\rho}^2 + |A (\Gamma, \Gamma \Lambda, 0, 0)|_{\mathscr{B}(\mathbb{C}^q, X)} \mathcal{L}_y (|\bar{\eta}|_{\mathbb{C}^q} + \bar{\rho}) \\
& \le \bar{\rho} ,
\end{align*}
which implies that $T(\,\cdot\, ; \talpha)$ maps $ \textnormal{cl} (B_{\bar{\rho}} (\bx))$ into itself.

Next, let $\chi = (\eta, \phi, v, w), \chi' = (\eta', \phi', v', w') \in \textnormal{cl} (B_{\bar{\rho}} (\bx))$. By the Mean Value Theorem, we have that
\begin{align*}
&|T (\chi; 0) - T (\chi'; 0)|_X \\
&\hspace*{1.5cm}\le \sup_{h \in \textnormal{cl} (B_{\bar{\rho}} (\bx))} | D T (h) |_{\mathscr{B}(X, X)} |\chi - \chi'|_X \\
&\hspace*{1.5cm}\le \sup_{h \in \textnormal{cl} (B_{\bar{\rho}} (\bx))} \left(|DT(\bx)|_{\mathscr{B}(X, X)} + \int_0^1 | D^2 T (\bx + t(h -\bx); 0) |_{\mathscr{B}(X, \mathscr{B}(X, X))} |h - \bx|_X \, dt \right) |\chi - \chi'|_X \\
&\hspace*{1.5cm}\le \left(Z_1 + Z_2 \bar{\rho} \right) |\chi - \chi'|_X .
\end{align*}
We infer that 
\begin{align*}
|T (\chi; \talpha) - T (\chi'; \talpha)|_X
&= |T (\chi; 0) - T (\chi'; 0) + \mathcal{A}(F(\chi ; 0) - F(\chi ; \talpha) - (F(\chi' ; 0) - F(\chi' ; \talpha)))|_X \\
&\le |T (\chi; 0) - T (\chi'; 0)|_X + |\mathcal{A}(F(\chi ; 0) - F(\chi ; \talpha) - (F(\chi' ; 0) - F(\chi' ; \talpha)))|_X \\
&= |T (\chi; 0) - T (\chi'; 0)|_X + |A
\begin{pmatrix}
\Gamma (\talpha(\delta, \eta) - \talpha(\delta, \eta')) , \Gamma \Lambda (\talpha(\delta, \eta) - \talpha(\delta, \eta')) , 0 , 0
\end{pmatrix}
|_X \\
&\le (Z_1 + Z_2 \bar{\rho} + |A (\Gamma, \Gamma \Lambda, 0, 0)|_{\mathscr{B}(\mathbb{C}^q, X)} \mathcal{L}_y) | \chi - \chi' |_X,
\end{align*}
which, in view of \eqref{eq:radii_constraint2}, shows that $T(\,\cdot\, ; \talpha)$ is a contraction on $\textnormal{cl} (B_{\bar{\rho}} (\bx))$.

Therefore, the operator $T(\,\cdot\, ; \talpha)$ satisfies the Banach Fixed-Point Theorem whenever \eqref{eq:radii_constraint1} and \eqref{eq:radii_constraint2} hold.
\end{proof}

There is a subtle technicality arising from our definition of $X$. Indeed, we choose to work in $\mathbb{C}$ to seamlessly allow $\Lambda$ to be comprised of complex eigenvalues. Consequently, we must ensure that $\tx = (\tilde{\eta}, \tilde{\phi}, \tilde{v}, \tilde{w}) \in \textnormal{cl}(B_{\bar{\rho}}(\bx))$, obtained from Theorem \ref{thm:radii_polynomial}, satisfies $\tilde{\phi} \in \mathbb{R}^q, \tilde{v} \in (\mathbb{R}^{\mathbb{N} \cup \{0\}})^q, \tilde{w} \in (\mathbb{R}^{\mathbb{N} \cup \{0\}})^{1+2q}$. Fortunately, this property is ``inherited'' from the numerical approximation $\bx$.

\begin{lemma}\label{lem:sym}
Let $\bx = (\bar{\eta}, \bar{\phi}, \bar{v}, \bar{w})$ and $\tx = (\tilde{\eta}, \tilde{\phi}, \tilde{v}, \tilde{w})$ as in Theorem \ref{thm:radii_polynomial}.
If $\bar{\phi} \in \mathbb{R}^q, \bar{v} \in (\mathbb{R}^{\mathbb{N} \cup \{0\}})^q, \bar{w} \in (\mathbb{R}^{\mathbb{N} \cup \{0\}})^{1+2q}$, then $\tilde{\phi} \in \mathbb{R}^q, \tilde{v} \in (\mathbb{R}^{\mathbb{N} \cup \{0\}})^q, \tilde{w} \in (\mathbb{R}^{\mathbb{N} \cup \{0\}})^{1+2q}$.
\end{lemma}

\begin{proof}
Suppose there are $2q'$ ($\le q$) complex eigenvalues. Without loss of generality, $\Lambda$ and $\Gamma = \begin{pmatrix} \Gamma_1 \cdots \Gamma_q \end{pmatrix}$ satisfy $\Lambda_{i,i} = \Lambda_{i+1,i+1}^\dagger$ and $\Gamma_i = \Gamma_{i+1}^\dagger$ for all $i = 1, 3, \dots, 2q'-1$ where $\dagger$ denotes complex conjugation. Consider $\conjugate : X \to X$ given by
\[
\conjugate(\chi) \bydef (\conjugate_0 (\eta), \phi^\dagger, v^\dagger, w^\dagger), \qquad \text{for all } \chi = (\eta, \phi, v, w) \in X,
\]
where $\conjugate_0 : \mathbb{C}^q \to \mathbb{C}^q$ is defined by
\[
\bigl( \conjugate_0 (\eta) \bigr)_i \bydef
\begin{cases}
\eta_{i+1}^\dagger, & i \in \{1, 3, \dots, 2q'-1\}, \\
\eta_{i-1}^\dagger, & i \in \{2, 4, \dots, 2q'\}, \\
\eta_i^\dagger, & i \in \{2q'+1, \dots, q\},
\end{cases} \qquad \text{for all } \eta \in \mathbb{C}^q.
\]
It is straightforward to see that $\phi \in \mathbb{R}^q, v \in (\mathbb{R}^{\mathbb{N} \cup \{0\}})^q, w \in (\mathbb{R}^{\mathbb{N} \cup \{0\}})^{1+2q}$ whenever $\chi = (\eta, \phi, v, w) \in X$ satisfies $\chi = \conjugate(\chi)$. Hence, the lemma claims that if $\bx = \conjugate(\bx)$, then $\tx = \conjugate(\tx)$.

Observe that for $\chi \in \textnormal{cl}(B_{\bar{\rho}}(\bx))$ we have $\conjugate(\chi) \in \textnormal{cl}(B_{\bar{\rho}}(\bx))$ since $|\conjugate(\chi) - \bx|_X = |\conjugate(\chi) - \conjugate(\bx)|_X = |\conjugate(\chi - \bx)|_X = |\chi - \bx|_X \le \bar{\rho}$.

Furthermore, our ordering of $\Gamma$ and $\Lambda$ yields $\Gamma \conjugate_0(\eta) = \left(\Gamma \eta\right)^\dagger$ and $\Lambda \conjugate_0(\eta) = \conjugate_0( \Lambda \eta)$ for all $\eta \in \mathbb{C}^q$; in particular,
$
\eta^\dagger = (\Gamma \Gamma^{-1} \eta)^\dagger = \Gamma \conjugate_0 (\Gamma^{-1} \eta).
$
Thus, for all $x \in \mathbb{C}$, $y, z \in \mathbb{C}^q$, we have that
\begin{align*}
\psi(x, \conjugate_0(y), \conjugate_0(z))
&=
\frac{1}{2}\Big(
(d-1) x \conjugate_0(-y + z) +
\Lambda^{-1} \Gamma^{-1}\Big( \mathbf{N}(c+\Gamma \conjugate_0(y +  z)) - D\mathbf{N}(c)\Gamma \conjugate_0(y + z) \Big)
\Big) \\
&=
\frac{1}{2}\Big(
(d-1) x \conjugate_0(-y + z) +
\Lambda^{-1} \Gamma^{-1}\Big( \mathbf{N}(c+\Gamma (y +  z)) - D\mathbf{N}(c)\Gamma (y + z) \Big)^\dagger
\Big) \\
&=
\frac{1}{2}\Big(
(d-1) x \conjugate_0(-y + z) +
\Lambda^{-1} \conjugate_0\left( \Gamma^{-1}\Big( \mathbf{N}(c+\Gamma (y +  z)) - D\mathbf{N}(c)\Gamma (y + z) \right)\Big)
\Big) \\
&=
\conjugate_0 ( \psi(x, y, z) ).
\end{align*}

Recall that we place ourselves in the context of Theorem \ref{thm:radii_polynomial}. Hence we consider $\talpha \in \mathcal{G}_{\delta, \mu, \mathcal{L}_x, \mathcal{L}_y}$ to be the local graph of the center-stable manifold as in the theorem.
Define $[\conjugate_1 (\talpha)](\xi, \eta) \bydef \conjugate_0 (\talpha(\xi, \conjugate_0 (\eta)))$ and observe that $\conjugate_1 (\talpha) \in \mathcal{G}_{\delta, \mu, \mathcal{L}_x, \mathcal{L}_y}$. We shall follow the notation introduced in the proof of Proposition~\ref{prop:graph}. Namely, we denote by $(x(r; \xi), y_\alpha(r; \xi, \eta))$ the unique solution of the initial value problem
\[
\left\{
\begin{aligned}
&\frac{dx}{d r}  = -x^2, \\
&\frac{dy}{d r}  = -\Lambda y + \psi (x, y, \alpha(x, y)), \\
&(x(0), y(0)) = (\xi, \eta),
\end{aligned}
\right.
\]
where $\xi \in [0, \delta]$ and $|\eta|_{\mathbb{C}^q} \le \mu$.
Fix $\eta \in \mathbb{C}^q$, then
\begin{align*}
\frac{d}{dr} \conjugate_0 (y_{\conjugate_1(\talpha)} (r; \xi, \eta))
&= \conjugate_0 \Bigl[
-\Lambda y_{\conjugate_1(\talpha)}(r; \xi, \eta)
+ \psi\Bigl(
x(r; \xi), y_{\conjugate_1(\talpha)}(r; \xi, \eta), \conjugate_1 \bigl(
\talpha \bigl( x(r; \xi), y_{\conjugate_1(\talpha)}(r; \xi, \eta) \bigr)
\bigr)
\Bigr)
\Bigr]\\
&= -\Lambda \conjugate_0 (y_{\conjugate_1(\talpha)}(r; \xi, \eta))
+ \psi\Bigl(
x(r; \xi), \conjugate_0 (y_{\conjugate_1(\talpha)}(r; \xi, \eta)),
\talpha \bigl(x, \conjugate_0 (y_{\conjugate_1(\talpha)}(r; \xi, \eta))
\bigr)
\Bigr).
\end{align*}
\noindent Since $\conjugate_0 (y_{\conjugate_1(\talpha)}(0; \xi, \eta)) = \conjugate_0 (\eta)$, it follows, by definition, that $\conjugate_0 (y_{\conjugate_1(\talpha)}(r; \xi, \eta)) = y_\talpha (r; \xi, \conjugate_0 (\eta))$. Hence, $y_{\conjugate_1(\talpha)}(r; \xi, \eta) = \conjugate_0(y_\talpha (r; \xi, \conjugate_0 (\eta)))$. Thus,
\begin{align*}
[\conjugate_1(\talpha)](\xi, \eta)
&=	
\conjugate_0 ( \talpha (\xi, \conjugate_0 (\eta) ) ) \\
&= \conjugate_0 \bigl( [\LP(\talpha)] (\xi, \conjugate_0 (\eta) ) \bigr) \\
&= \conjugate_0 \int_0^\infty e^{-\Lambda r} \psi \Bigl(x(r; \xi), y_\talpha(r; \xi, \conjugate_0(\eta)), \talpha \bigl(x(r; \xi), y_\talpha(r; \xi, \conjugate_0(\eta))\bigr)\Bigr) \, dr \\
&= \int_0^\infty e^{-\Lambda r} \psi \Bigl(x(r; \xi), \conjugate_0 (y_\talpha(r; \xi, \conjugate_0(\eta))), \conjugate_0 \bigl(\talpha \bigl(x(r; \xi), y_\talpha(r; \xi, \conjugate_0(\eta))\bigr)\bigr)\Bigr) \, dr \\
&= \int_0^\infty e^{-\Lambda r} \psi \Bigl(x(r; \xi), y_{\conjugate_1(\talpha)}(r;\xi, \eta), \conjugate_0 \bigl(\talpha \bigl(x(r;\xi), \conjugate_0 y_{\conjugate_1(\talpha)}(r;\xi, \eta) \bigr)\bigr) \Bigr) \, dr \\
&= \int_0^\infty e^{-\Lambda r} \psi \Bigl(x(r; \xi), y_{\conjugate_1(\talpha)}(r;\xi, \eta), [\conjugate_1 (\talpha)] \bigl(x(r;\xi), y_{\conjugate_1(\talpha)}(r;\xi, \eta) \bigr) \Bigr) \, dr \\
&= [\LP(\conjugate_1(\talpha))](\xi, \eta).
\end{align*}
However, according to Proposition \ref{prop:graph}, $\talpha$ is the unique fixed-point of $\LP$ in $\mathcal{G}_{\delta, \mu, \mathcal{L}_x, \mathcal{L}_y}$, which implies that $\conjugate_1 (\talpha) = \talpha$. In particular, $\talpha(\xi, \conjugate_0 (\eta)) = [\conjugate_1 (\talpha)](\xi, \conjugate_0(\eta)) = \conjugate_0 (\talpha(\xi, \conjugate_0(\eta)))$.

Therefore,
\begin{align*}
F(\conjugate(\tx); \talpha)
&=
\begin{pmatrix}
\{\tilde{w}^{(2)}\}_0^\dagger + 2\sum_{n \ge 1 } \{\tilde{w}^{(2)}\}_n^\dagger - c - \Gamma (\conjugate_0(\tilde{\eta}) + \talpha((\ell r_* + L)^{-1},\conjugate_0(\tilde{\eta}))) \\
\{\tilde{w}^{(3)}\}_0^\dagger + 2\sum_{n \ge 1 } \{\tilde{w}^{(3)}\}_n^\dagger - \Gamma \Lambda (-\conjugate_0(\tilde{\eta}) + \talpha((\ell r_* + L)^{-1},\conjugate_0(\tilde{\eta}))) \\
g(\conjugate(\tx)) \\
h(\conjugate(\tx))
\end{pmatrix} \\
&=
\begin{pmatrix}
\left[ \{\tilde{w}^{(2)}\}_0 + 2\sum_{n \ge 1 } \{\tilde{w}^{(2)}\}_n - c - \Gamma (\tilde{\eta} + \talpha((\ell r_* + L)^{-1},\tilde{\eta})) \right]^\dagger \\
\left[ \{\tilde{w}^{(3)}\}_0^\dagger + 2\sum_{n \ge 1 } \{\tilde{w}^{(3)}\}_n^\dagger - \Gamma \Lambda (-\tilde{\eta} + \talpha((\ell r_* + L)^{-1},\tilde{\eta}))  \right]^\dagger \\
g(\tx)^\dagger \\
h(\tx)^\dagger
\end{pmatrix} \\
&= F(\tx; \talpha)^\dagger \\
&= 0.
\end{align*}
Since $\tx$ is the unique zero of $F(\,\cdot\,; \talpha)$ in $\textnormal{cl}(B_{\bar{\rho}}(\bx))$, it follows that $\tx = \conjugate(\tx)$ as desired.
\end{proof}

To end this section, we retrieve 
an a posteriori $C^0$-error bound for a localized radial solution of \eqref{eq:elliptic}.
To formulate the result, we first need an expression for the solution in the local center-stable manifold. We set 
$r_0 \bydef \ell r_* + L$ and
\[
\tilde{\eta} \bydef \textstyle \tfrac{1}{2} \Gamma^{-1} \bigl(\{\tilde{w}^{(2)}\}_0 + 2\sum_{n \ge 1 } \{\tilde{w}^{(2)}\}_n -c\bigr) - \tfrac{1}{2} \Lambda^{-1} \Gamma^{-1} \bigl( \{\tilde{w}^{(3)}\}_0 + 2\sum_{n \ge 1 } \{\tilde{w}^{(3)}\}_n \bigr).
\]
Then the solution in the local-center manifold is given by
\[
\tilde{u}_{\talpha}(r) \bydef c + \Gamma \, \Bigl( y_{\talpha}(r -r_0 ; r_0^{-1} , \tilde{\eta}) + \talpha\bigl(r^{-1},y_{\talpha}(r-r_0 ; r_0^{-1} , \tilde{\eta})\bigr) \Bigr), \qquad \text{for all }r\in (r_0,\infty),
\]
where $y_\alpha$ is given implicitly by the variation of constants formula~\eqref{eq:variation_of_constant}.

\begin{corollary}\label{cor:C0bound}
Let $\tx = (\tilde{\eta}, \tilde{\phi}, \tilde{v}, \tilde{w}) \in \textnormal{cl}(B_{\bar{\rho}}(\bx))$ be the zero of $F(\,\cdot\,; \talpha)$ obtained by applying Theorem~\ref{thm:radii_polynomial} to a case where $\bx = \conjugate(\bx)$. Consider the resulting localized radial solution $r \in [0, \infty) \mapsto \tilde{u}(r)$ of \eqref{eq:elliptic} satisfying 
\begin{align*}
\tilde{u}(r) &=
\begin{cases}
\displaystyle\sum_{n \ge 0} \{\tilde{v}\}_n (\ell^{-1}r)^n, & r \in [0, \ell r_*], \\
\displaystyle\{\tilde{w}^{(2)}\}_0 + 2 \sum_{n \ge 1} \{ \tilde{w}^{(2)} \}_n \cheb_n (\tfrac{2}{L} (r - \ell r_*) - 1), & r \in (\ell r_*, r_0 ],\\
\tilde{u}_{\talpha}(r) & r\in (r_0 ,\infty).
\end{cases}
\end{align*}
Define 
\begin{align*}
\bar{u}(r) \bydef
\begin{cases}
\displaystyle\sum_{n = 0}^{n_\mathscr{T}} \{\bar{v}\}_n (\ell^{-1}r)^n, & r \in [0, \ell r_*], \\
\displaystyle\{\bar{w}^{(2)}\}_0 + 2 \sum_{n = 1}^{n_\mathscr{C}} \{ \bar{w}^{(2)} \}_n \cheb_n (\tfrac{2}{L} (r - \ell r_*) - 1), & r \in (\ell r_*, r_0 ], \\
c + \Gamma e^{-\Lambda (r - r_0)} \bar{\eta} , & r \in (r_0 , \infty).
\end{cases}
\end{align*}
Then,
\begin{equation}\label{eq:C0error}
\sup_{r \in [0, \infty)} |\tilde{u}(r) - \bar{u}(r)|_{\mathbb{C}^q} \le \max\bigl\{\bar{\rho}, |\Gamma|_{\mathscr{B}(\mathbb{C}^q,\mathbb{C}^q)} (\bar{\rho} + \mathcal{L}_y (|\bar{\eta}|_{\mathbb{C}^q} + \bar{\rho}))\bigr\}.
\end{equation}
\end{corollary}

\begin{proof}
On the one hand,
\[
\sup_{r \in [0, r_0]} | \tilde{u}(r) - \bar{u}(r) |_{\mathbb{C}^q}
\le \max\left(|\tilde{v} - \bar{v}|_{\mathscr{T}^q}, |\tilde{w}^{(2)} - \bar{w}^{(2)}|_{\mathscr{C}^q_\nu} \right)
\le
|\tx - \bx|_{X}
\le \bar{\rho},
\]
since the supremum norm of a Taylor or Chebyshev series is bounded by the weighted $\ell^1$ norm of its coefficients (namely, the norm of $\mathscr{T}$ and $\mathscr{C}_\nu$ respectively).

On the other hand,
\begin{align*}
\sup_{r \in (r_0, \infty)} |\tilde{u}(r) - \bar{u}(r)|_{\mathbb{C}^q}
&=
\sup_{r \in (r_0, \infty)} |\tilde{u}_\talpha(r) - \bar{u}(r)|_{\mathbb{C}^q} \\
&\le |\Gamma|_{\mathscr{B}(\mathbb{C}^q,\mathbb{C}^q)} \sup_{r \in (r_0 , \infty)} \Bigl| y_{\talpha}(r - r_0 ; r_0^{-1} , \tilde{\eta}) - e^{-\Lambda (r - r_0)} \bar{\eta} + \talpha\bigl(r^{-1},y_{\talpha}(r-r_0 ; r_0^{-1} , \tilde{\eta})\bigr) \Bigr|_{\mathbb{C}^q} \\
&\le |\Gamma|_{\mathscr{B}(\mathbb{C}^q,\mathbb{C}^q)} (\bar{\rho} + \mathcal{L}_y (|\bar{\eta}|_{\mathbb{C}^q} + \bar{\rho})) ,
\end{align*}
where we used the variation of constant formula \eqref{eq:variation_of_constant} to derive the uniform bound $|y_{\talpha}(r - r_0 ; r_0^{-1} , \tilde{\eta}) - e^{-\Lambda (r - r_0)} \bar{\eta}|_{\mathbb{C}^q} \leq \bar{\rho}$ for all $r \in (r_0, \infty)$, in a similar vein as we did to obtain \eqref{eq:bound_y} in the proof of Proposition \ref{prop:graph}.
\end{proof}

\begin{remark}[Exponential decay]\label{rem:decay}
The localized radial solution $\tilde{u}_\talpha$ of \eqref{eq:elliptic} decays exponentially fast to the constant solution~$c$ of \eqref{eq:elliptic}.
Indeed, for $r > r_0$, it follows from the fixed point construction in Section~\ref{sec:graph_enclosure} that we have $|u_\talpha (r) - c|_{\mathbb{C}^q}
\le |\Gamma|_{\mathscr{B}(\mathbb{C}^q, \mathbb{C}^q)} (1 + \mathcal{L}_y) | y_{\talpha}(r -r_0 ; r_0^{-1} , \tilde{\eta})|_{\mathbb{C}^q}$.
By condition~\eqref{eq:graph_constraint1} the estimate \eqref{eq:bound_y} shows that $|u_\talpha (r) - c|_{\mathbb{C}^q}$ decays exponentially with rate at least $(\frac{d-1}{2}\delta + \hat{\psi})(1+\mathcal{L}_y) - \hat{\lambda} < 0$.

In fact, the asymptotic decay rate is $-\hat{\lambda}$.
First, note that $\psi(x, y, z) = \frac{d-1}{2} x (-y + z) + \Upsilon (y + z)$ where $\Upsilon ( \zeta ) \bydef
\frac{1}{2} \Lambda^{-1} \Gamma^{-1}\Big( \mathbf{N}(c+\Gamma \zeta) - D\mathbf{N}(c)\Gamma \zeta \Big)$.
It follows that there exists $\kappa = \kappa(\mu) > 0$ such that $| D \Upsilon (\zeta)|_{\mathbb{C}^q} \le \kappa (1 + \mathcal{L}_y)^{-1} |\zeta|_{\mathbb{C}^q}$ for all $|\zeta|_{\mathbb{C}^q} \le \mu$.
Thus, by the Mean Value Theorem, cf.~\eqref{e:yxieta1xieta2},
\begin{align*}
&\left|\psi\Bigl(r^{-1}, y_\talpha(r-r_0; r_0^{-1}, \tilde{\eta}), \talpha \bigl(r^{-1}, y_\talpha(r-r_0; r_0^{-1}, \tilde{\eta}) \bigr)\Bigr)\right|_{\mathbb{C}^q} \\
&\hspace*{3cm}= \left|\psi\Bigl(r^{-1}, y_\talpha(r-r_0; r_0^{-1}, \tilde{\eta}), \talpha \bigl(r^{-1}, y_\talpha(r-r_0; r_0^{-1}, \tilde{\eta}) \bigr)\Bigr) - \psi(r^{-1}, 0, 0)\right|_{\mathbb{C}^q} \\
&\hspace*{3cm}\le C(r) |y_\talpha(r-r_0; r_0^{-1}, \tilde{\eta})|_{\mathbb{C}^q}, 
\end{align*}
where $C(r) \bydef (\tfrac{d-1}{2} r^{-1} 
+ \kappa |\tilde{\eta}|_{\mathbb{C}^q} e^{\left((\frac{d-1}{2}\delta 
+ \hat{\psi})(1 + \mathcal{L}_y) - \hat{\lambda} \right) (r-r_0)}) (1 + \mathcal{L}_y)$.
Finally, by using the estimates \eqref{eq:rearranged}, \eqref{eq:bound_y} and Gr\"onwall's inequality, we obtain
\[
|y_\talpha (r-r_0; r_0^{-1}, \tilde{\eta})|_{\mathbb{C}^q} \le |\tilde{\eta}|_{\mathbb{C}^q} e^{\int_{r_0}^r C(s) \, ds - \hat{\lambda} (r-r_0)},
\]
with $(r - r_0)^{-1} \int_{r_0}^r C(s) \, ds \to 0$ as $r \to \infty$.
\end{remark}

\begin{remark}[Transversality]
A proof of a localized radial solution of \eqref{eq:elliptic} via Theorem \ref{thm:radii_polynomial} requires the connecting orbit to be isolated.
This transpires geometrically as a transverse intersection between the set of solutions, i.e. orbits of~\eqref{eq:autonomous_ode} (parameterized by $u(0) = \phi \in \R^q$ at $r=0$, see~\eqref{eq:taylor}), and the center-stable manifold (parameterized by $\eta \in \mathbb{C}^q$ satisfying $\conjugate_0(\eta)=\eta$) at the section $w^{(1)} = r_0^{-1}$. We explore this briefly.

Of course, to properly discuss such a transverse intersection, one first needs to establish the differentiability of the local graph $\talpha$ at $(r_0^{-1}, \tilde{\eta})$. 
It follows from standard ODE theory arguments that, although we initially obtain $\talpha$ in a Lipschitz class in Proposition~\ref{prop:graph}, the graph of $\talpha$ is in fact $C^1$. 
Then, Theorem \ref{thm:radii_polynomial} implies that
\begin{align*}
| I - \mathcal{A} DF(\tx ; \talpha)|_{\mathscr{B}(X, X)}
&= | I - \mathcal{A} DF(\bx ; 0)|_{\mathscr{B}(X, X)} + | \mathcal{A} (DF(\bx ; 0) - DF(\tx ; \talpha))|_{\mathscr{B}(X, X)} \\
&\le |A (\Gamma, \Gamma \Lambda, 0, 0)|_{\mathscr{B}(\mathbb{C}^q, X)} \mathcal{L}_y + Z_1 + Z_2 \bar{\rho} \\
&< 1,
\end{align*}
which implies the injectivity of $DF(\tx ; \talpha)$.
On the other hand, by assuming for contradiction that the intersection introduced above is not transverse, one can construct an element in the kernel of $DF(\tx ; \talpha)$, hence establishing transversality. However, providing the details for these arguments would lead us too far astray from the scope of this article, hence we leave those to the interested reader.
\end{remark}

%%%%%%%%%%%%%
%% APPLICATIONS  %%
%%%%%%%%%%%%%

\section{Applications} \label{sec:applications}
%!TEX root = radial_elliptic.tex

We begin this section with a practical note on Theorem \ref{thm:radii_polynomial}.
In Section \ref{sec:newton-kantorovich}, we demonstrated that a zero of $F(\,\cdot\,; \talpha)$ yields a localized radial solution of \eqref{eq:elliptic} which reaches the local center-stable manifold $\talpha$ of the stationary point $0$ of \eqref{eq:diag_autonomous_ode}.
The main result, Theorem \ref{thm:radii_polynomial}, revolves around a contraction occurring in a vicinity of a numerical approximation $\bx = (\bar{\eta}, \bar{\phi}, \bar{v}, \bar{w})$ of a zero of $F(\,\cdot\,; 0)$.
In other words, we pursue a zero of $F(\,\cdot\,; \talpha)$ by hunting with $F(\,\cdot\,; 0)$, and in particular its finite dimensional truncation.

The ingredients involved in Theorem \ref{thm:radii_polynomial} are correlated.
The larger $\ell$ and $L$, the smaller the center-stable coordinates $\delta$ and $\mu$ will be and the smaller we may choose $\mathcal{L}_x, \mathcal{L}_y$. On the other hand, larger $\ell$ and $L$ results in higher truncation orders $n_\mathscr{T}$ and $n_\mathscr{C}$ for the Taylor and Chebyshev series to represent the desired functions sufficiently well.

To shed some light on applying Theorem \ref{thm:radii_polynomial} in practice, we list the main steps of the procedure:
\begin{enumerate}
\item (\emph{A numerical approximate solution}). By some suitable method, which depends on the PDE under consideration, 
compute numerically a localized solution $u_0: [0, r_0] \to \mathbb{R}^q$ of \eqref{eq:radial_elliptic}, with $u_0(r_0) \approx c$ for some zero $c$ of $\mathbf{N}$. Set $\bar{\phi}=u_0(0)$.

\item (\emph{The eigenvalue problem}). Solve rigorously $- D\mathbf{N}(c) \Gamma = \Gamma \Lambda^2$ for $\Lambda, \Gamma : \mathbb{C}^q \to \mathbb{C}^q$ such that $\Lambda$ is a diagonal matrix whose diagonal entries have strictly positive real parts; in other words, the diagonal entries correspond to the unstable eigenvalues of $Df(\bc)$.
If there are $2q'$ ($\le q$) complex eigenvalues, organize the columns of $\Lambda$ and $\Gamma = \begin{pmatrix} \Gamma_1 & \cdots & \Gamma_q \end{pmatrix}$ such that $\Lambda_{i,i} = \Lambda_{i+1,i+1}^\dagger$ and $\Gamma_i = \Gamma_{i+1}^\dagger$ for all $i = 1, 3, \dots, 2q'-1$.
Then, compute a numerical solution $\bar{\eta} \in \mathbb{C}^q$ of $u_0(r_0) = c + \Gamma \bar{\eta} $. Enforce $\bar{\eta} = \conjugate_0(\bar{\eta})$.

\item (\emph{Series representation}). Fix $\ell \in (0, r_0]$ and approximate $r \in [0, \ell] \mapsto u_0 (r)$ by a Taylor series $r \in [0, 1] \mapsto \bar{v}(r) \approx u_0(\ell^{-1} r)$ of order $n_\mathscr{T} \ge 2$.
Then fix $r_* \in (0, e^{-1/(n_\mathscr{T} + 1)}]$, and approximate $r \in [\ell r_*, r_0] \mapsto u_0 (r)$ by a Chebyshev series $s \in [-1, 1] \mapsto \bar{w}(s) \approx u_0(\frac{1-s}{2} \ell r_* + \frac{s+1}{2}r_0)$ of order $n_\mathscr{C} \ge 1$; in particular, $L = r_0 - \ell r_*$.

\item (\emph{Refine the approximate solution}).
At this stage, we have obtained $\bx = (\bar{\eta}, \bar{\phi}, \bar{v}, \bar{w}) \in \pi^{n_\mathscr{T}, n_\mathscr{C}} X$ such that $\bx$ lies in the scope of Lemma \ref{lem:sym}.
As explained at the beginning of this section, one may refine the constructed approximate zero $\bx$ of the mapping $\chi \in \pi^{n_\mathscr{T}, n_\mathscr{C}} X \mapsto \pi^{n_\mathscr{T}, n_\mathscr{C}} F(\chi ; 0)$ by applying Newton's method.
As a rule of thumb, the truncation orders $n_\mathscr{T}, n_\mathscr{C}$ are chosen such that the last coefficients of the series expansions are of the order of \emph{machine epsilon} (e.g. $\sim 10^{-16}$ in double precision). Nevertheless, in practice, higher truncation orders may be required to bring about the desired contraction.

\item (\emph{The center-stable manifold}). Retrieve $\hat{\lambda}$ as defined in \eqref{eq:lambda_hat}.
Compute $\delta = (\ell r_* + L)^{-1}$. Fix an a priori maximal error bound $\varrho > 0$ and set $\mu = |\bar{\eta}|_{\mathbb{C}^q} + \varrho$. Determine $\hat{\psi}$ satisfying the inequality \eqref{eq:psi_hat}.
Choose $\mathcal{L}_x, \mathcal{L}_y > 0$ and verify the inequalities \eqref{eq:graph_constraint1}, \eqref{eq:graph_constraint2} and \eqref{eq:graph_constraint3} to prove that Proposition \ref{prop:graph} holds.

\item (\emph{Choose $A$ and check the Newton-Kantorovich inequalities}).
Compute a numerical inverse $A$ of $\pi^{n_\mathscr{T}, n_\mathscr{C}} DF(\bx; 0) \pi^{n_\mathscr{T}, n_\mathscr{C}}$.
Choose a decay rate $\nu > 1$ for the Chebyshev sequence space. The approximate Chebyshev coefficients $\bar{w}$ guide this choice: $\nu$ cannot exceed their decay rate.
Finally, compute rigorously $Y$, $Z_1$ and $Z_2$. Note that to satisfy inequality \eqref{eq:radii_constraint1}, then the a priori maximal error $\varrho$ must necessarily be greater than $Y + |A(\Gamma, \Gamma \Lambda, 0, 0)|_{\mathscr{B}(\mathbb{C}^q, X)} \mathcal{L}_y |\bar{\eta}|_{\mathbb{C}^q}$.
Verify the inequalities \eqref{eq:radii_constraint1} and \eqref{eq:radii_constraint2} to prove that Theorem \ref{thm:radii_polynomial} holds.
\end{enumerate}

In the following applications, all the computations have been performed in Julia \cite{Julia} via the package \emph{RadiiPolynomial.jl} (cf.~\cite{RadiiPolynomial.jl}) which relies on the package \emph{IntervalArithmetic.jl} (cf.~\cite{IntervalArithmeticJulia}) for rigorous floating-point computations. The code is available at \cite{Code}. All figures have been generated via the package \emph{Makie.jl} (cf.~\cite{Makie.jl}).

\subsection{Example \ref{ex:exampleKG}: cubic Klein-Gordon equation}
\label{sec:3DKG}

Our first example consists in proving localized radial stationary solutions of the cubic Klein-Gordon equation as described in Theorem \ref{thm:exampleKG}. Recall that the equation \eqref{eq:exampleKG} under study reads
\[
U_{tt} = \Delta U - U + \beta_1 U^2 +\beta_2 U^3, \qquad U = U(t, x) \in \R, \quad t \ge 0, \quad x \in \R^3,
\]
with parameters $\beta_1, \beta_2 \in \R$. Hence, the stationary solutions solve \eqref{eq:elliptic} with $q=1$, $d=3$ and
\[
\mathbf{N}(U) = - U + \beta_1 U^2 +\beta_2 U^3.
\]
For $\beta_1 = \beta_2 = 1$, we compute numerical approximations of two distinct localized radial stationary solutions about the constant equilibrium $c = 0$.

Since $D\mathbf{N}(0) = -1$, the equation $-D\mathbf{N}(0) \Gamma = \Lambda^2 \Gamma$ is satisfied for $\Gamma = \Lambda = 1$; it follows that $\hat{\lambda} = 1$. 
Moreover, for any $\mu > 0$, the inequality \eqref{eq:psi_hat} is satisfied for
\[
\hat{\psi} = \left(|\beta_1| + \frac{3}{2}|\beta_2| (1+\mathcal{L}_y)\mu \right)(1+\mathcal{L}_y)\mu.
\]

Figure \ref{fig:KG} shows the numerical approximations of the two localized radial stationary solution of the cubic Klein-Gordon equation \eqref{eq:exampleKG}. We successfully verified Proposition \ref{prop:graph} and Theorem \ref{thm:radii_polynomial} and obtained a $C^0$-error bound~\eqref{eq:C0error} for each numerical approximation depicted in Figures \ref{fig:exampleKG} and \ref{fig:KG}: $2.9 \times 10^{-7}$ and $5.5 \times 10^{-6}$, respectively.
For both proofs, we chose the Lipschitz constants of the local graph of the center-stable manifold to be $\mathcal{L}_x = 1$ and $\mathcal{L}_y \approx 0.05$. The domain of the local graph is contained in $[0, 0.068] \times [-2.3  \times 10^{-7}, 2.3 \times 10^{-7}]$ and $[0, 0.066] \times [-5.1  \times 10^{-6}, 5.1 \times 10^{-6}]$, respectively.

The proofs show a disparity between the order of the Taylor and Chebyshev series expansions needed to have a good approximation of the solutions and the order required to trigger the contraction involved in Theorem \ref{thm:radii_polynomial}. Specifically, the truncation orders $n_\mathscr{T},n_\mathscr{C}$ used for Theorem \ref{thm:radii_polynomial} are split as follows. The \emph{numerical truncation orders} $n_{\mathscr{T},\textnormal{num}},n_{\mathscr{C},\textnormal{num}}$ mark the limit beyond which the Taylor and Chebyshev series coefficients, respectively, are judged negligible (these coefficients are approximated by zeros).
The \emph{padding orders} $n_{\mathscr{T},\textnormal{pad}},n_{\mathscr{C},\textnormal{pad}}$ correspond to the integers $n_{\mathscr{T},\textnormal{num}},n_{\mathscr{C},\textnormal{num}}$ such that $n_\mathscr{T} = n_{\mathscr{T},\textnormal{num}} + n_{\mathscr{T},\textnormal{pad}}$ and $n_\mathscr{C} = n_{\mathscr{C},\textnormal{num}} + n_{\mathscr{C},\textnormal{pad}}$ are sufficiently large to satisfy all assumptions in Theorem \ref{thm:radii_polynomial}.
This dichotomy is especially noticeable for the Chebyshev series expansion of the second solution of \eqref{eq:autonomous_ode}.
For the proof of the first solution depicted in Figure \ref{fig:KGa}, we used $n_{\mathscr{T},\textnormal{num}} = 60$, $n_{\mathscr{T},\textnormal{pad}} = 100$ and $n_{\mathscr{C},\textnormal{num}} = 85$, $n_{\mathscr{C},\textnormal{pad}} = 200$ (finished in under 1 second on a laptop with an M1 CPU, using 8 threads, with 8 GiB of max.\ available RAM), whereas for the proof of the second solution depicted in Figure \ref{fig:KGb}, we used $n_{\mathscr{T},\textnormal{num}} = 150$, $n_{\mathscr{T},\textnormal{pad}} = 100$ and $n_{\mathscr{C},\textnormal{num}} = 200$, $n_{\mathscr{C},\textnormal{pad}} = 15,000$ (finished in under 2 hours on a server with an Intel E7-4850 v4 Broadwell CPU, using 32 threads, with 2510 GiB of max.\ available RAM).

As noted in Example \ref{ex:exampleKG}, for any given integer $m \ge 0$, there exists a localized radial stationary solution of \eqref{eq:exampleKG} with exactly $m$ zeros. Here, we only looked at the case $m=0$ and $m=1$.
While in principle our methodology seamlessly applies for larger $m$, the profile of these solutions warrants caution. Indeed, numerical simulations suggest that as $m$ grows, the height of the solution at $r=0$ and its steepness increase. This transpires in Theorem \ref{thm:radii_polynomial} as a drastic growth of the $Z_1$ bound which is mitigated by taking $n_\mathscr{C}$ much larger. In such cases, we emphasize that combining the present methodology with the domain decomposition strategy in \cite{MR4292534} should relieve the burden of having only one Chebyshev series representing such a ``wild'' profile. 

\begin{figure}[!ht]
\centering
\begin{subfigure}[b]{0.49\textwidth}
\centering
\includegraphics[width=\textwidth]{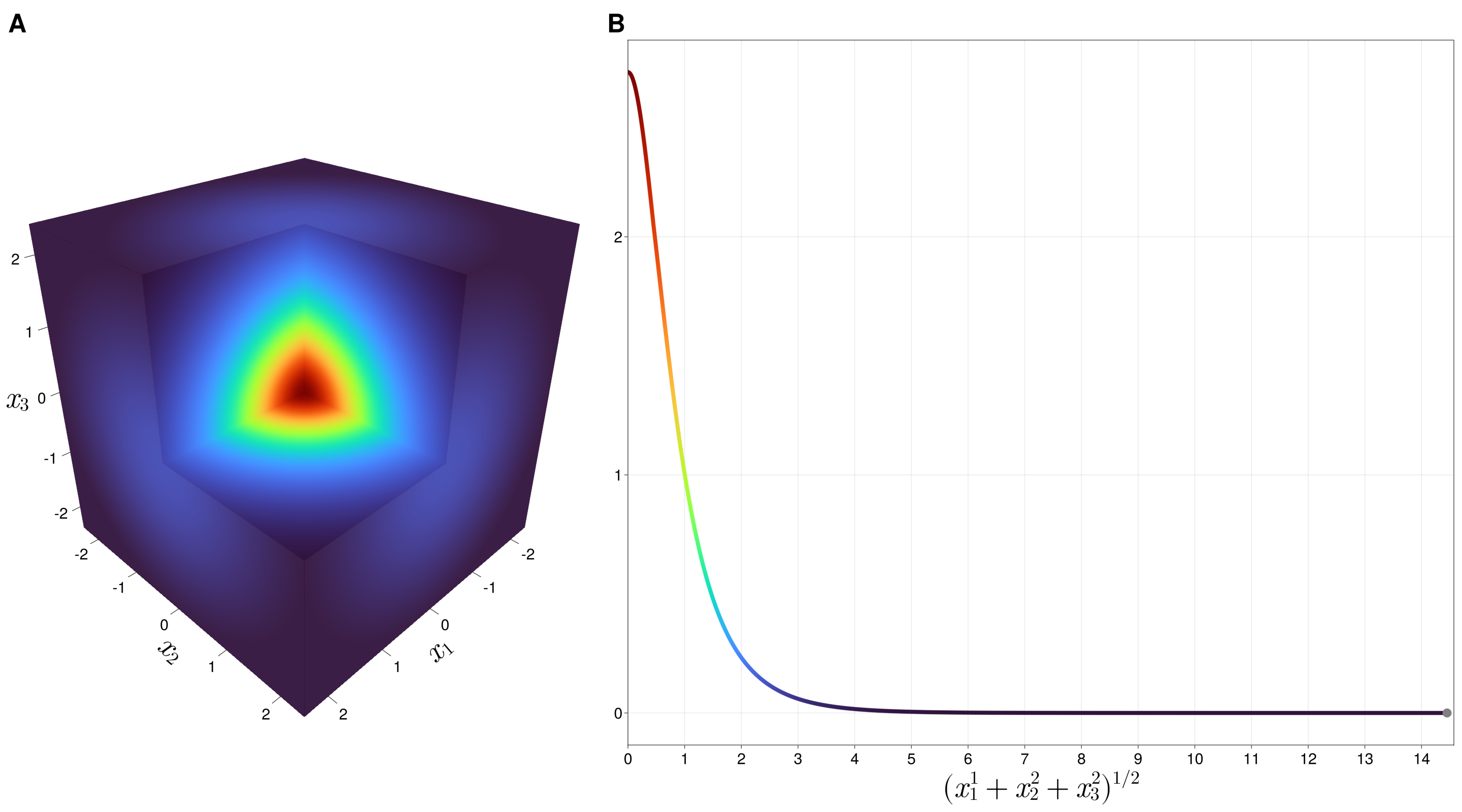}
\caption{}\label{fig:KGa}
\end{subfigure}
\hfill
\begin{subfigure}[b]{0.49\textwidth}
\centering
\includegraphics[width=\textwidth]{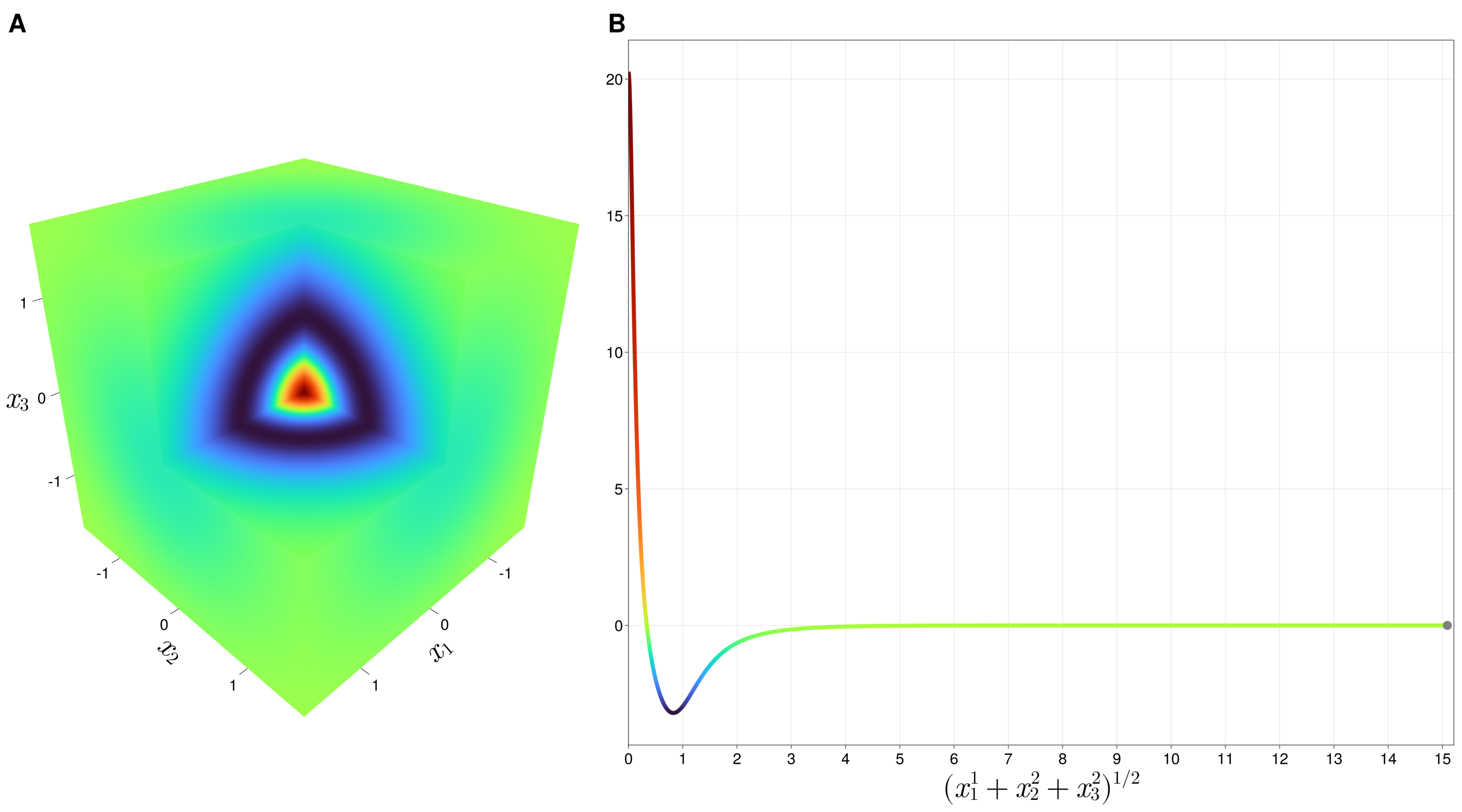}
\caption{}\label{fig:KGb}
\end{subfigure}
\caption{Numerical localized radial stationary solutions of the cubic Klein-Gordon equation \eqref{eq:exampleKG} on $\R^3$ at parameter values $\beta_1 = \beta_2 = 1$. (a) This approximation is proven to be $2.9 \times 10^{-7}$ close (in $C^0$-norm) to a strictly positive localized radial stationary solution. (b) This approximation is proven to be $5.5 \times 10^{-6}$ close (in $C^0$-norm) to a localized radial stationary solution with  one zero. Panel \textbf{A} shows the approximation as a cut-out in $\R^3$. Panel \textbf{B} shows the approximation as a function of the radius $r$. 
}
\label{fig:KG}
\end{figure}

\subsection{Example \ref{ex:exampleSH}: Swift-Hohenberg equation}
\label{sec:2DSH}

Our second example consists in proving localized radial stationary solutions of the Swift-Hohenberg equation as described in Theorem \ref{thm:exampleSH}. Recall that the equation \eqref{eq:exampleSH} under study reads
\[
U_t = -(\beta_4 + \Delta)^2 U + \beta_1 U + \beta_2 U^2 + \beta_3 U^3, \qquad U = U(t, x) \in \R, \quad t \ge 0, \quad x \in \R^2,
\]
with parameters $\beta_1, \beta_2, \beta_3, \beta_4 \in \R$. Observe that if $U$ is a stationary solution of \eqref{eq:exampleSH}, then $(U_1, U_2) = (U, (\beta_4 + \Delta) U)$ solves
\[
\begin{cases}
\Delta U_1 + \beta_4 U_1 - U_2 \hspace{2.8cm}= 0, \\
\Delta U_2 + \beta_4 U_2 - \beta_1 U_1 - \beta_2 U_1^2 - \beta_3 U_1^3 = 0.
\end{cases}
\]
Conversely, if $(U_1, U_2)$ is a solution of the previous system, then $U_1$ is a stationary solution of \eqref{eq:exampleSH}. Hence, the stationary solutions of \eqref{eq:exampleSH} correspond to solutions of \eqref{eq:elliptic} with $q=2$, $d = 2$ and
\[
\mathbf{N}(U_1, U_2) = \begin{pmatrix} \beta_4 U_1 - U_2 \\ \beta_4 U_2 - \beta_1 U_1 - \beta_2 U_1^2 - \beta_3 U_1^3 \end{pmatrix}.
\]
Following the parameter values studied in \cite{MR2475557} (note that in the cited article, the authors consider $\beta_4 \equiv 1$ and use $\mu, \nu, \kappa$ instead of $-\beta_1, \beta_2, -\beta_3$, respectively), for $(\beta_1, \beta_2, \beta_3, \beta_4) = (-\tfrac{3}{5},\sqrt{6},-\tfrac{1}{10},1)$, we compute a numerical approximation of a localized radial stationary solutions oscillating about the constant equilibrium $c = (0, 0)$.

The equation $-D\mathbf{N}(0, 0) \Gamma = \Gamma \Lambda^2$ is satisfied for
\[
\Lambda = \begin{pmatrix}
\sqrt{-\sqrt{\beta_3} - \beta_4} & 0 \\
0 & \sqrt{\sqrt{\beta_3} - \beta_4}
\end{pmatrix}
\quad \text{and} \quad
\Gamma = \begin{pmatrix}
1 & 1 \\
-\sqrt{\beta_3} & \sqrt{\beta_3}
\end{pmatrix},
\]
and it follows that $\hat{\lambda} = \Re(\sqrt{-\sqrt{\beta_3} - \beta_4})$. Moreover, for any $\mu > 0$, the inequality \eqref{eq:psi_hat} is satisfied for
\[
\hat{\psi} = \left(|\beta_2| + \frac{3}{2}|\beta_3| (1+\mathcal{L}_y)\mu \right)(1+\mathcal{L}_y)\mu |
\Lambda^{-1} \Gamma^{-1}|_{\mathscr{B}(\mathbb{C}^2,\mathbb{C}^2)} |\Gamma
|_{\mathscr{B}(\mathbb{C}^2,\mathbb{C}^2)} .
\]

Figure \ref{fig:SH} shows the numerical approximation of a localized radial stationary solution of the Swift-Hohenberg equation \eqref{eq:exampleSH}. We successfully verified Proposition \ref{prop:graph} and Theorem \ref{thm:radii_polynomial} and obtained a $C^0$-error bound for the numerical approximation depicted in Figures \ref{fig:exampleSH} and \ref{fig:SH}: $2.4 \times 10^{-5}$.
For the proof, we chose the Lipschitz constants of the local graph of the center-stable manifold to be $\mathcal{L}_x = 1$ and $\mathcal{L}_y \approx 0.03$. The domain of the local graph is contained in $[0, 0.031] \times \{ y \in \mathbb{C} \, : \, |y| \le 10^{-6} \}^2$.
The proof finished in under 1 minute on a laptop with an M1 CPU, using 8 threads, with 8 GiB of max.\ available memory.

\begin{figure}[!ht]
\centering
\includegraphics[width=0.5\textwidth]{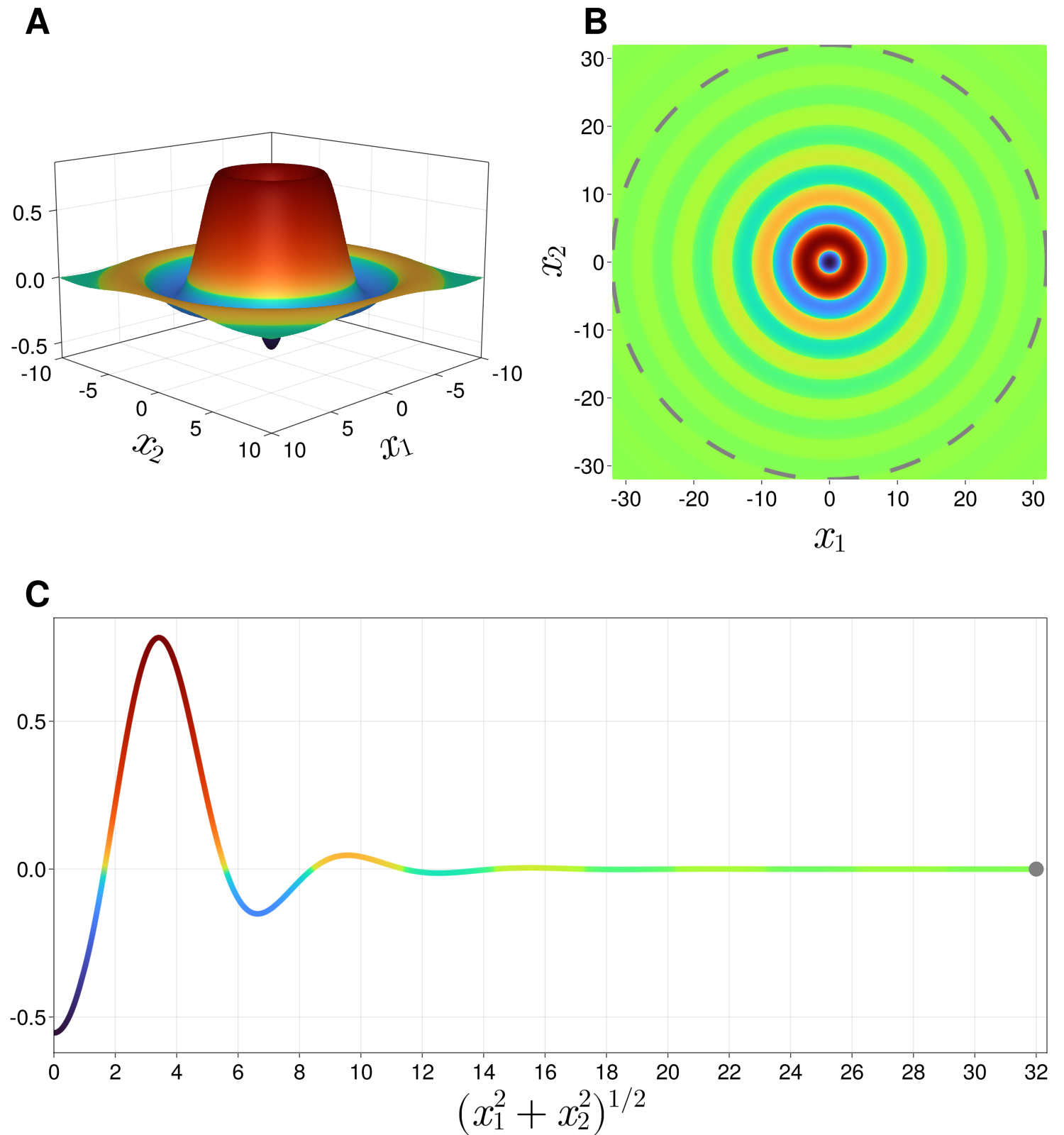}
\caption{Numerical stationary \emph{ring} of the Swift-Hohenberg equation \eqref{eq:exampleSH} on $\mathbb{R}^2$ at parameter values $(\beta_1, \beta_2, \beta_3, \beta_4) = (-\tfrac{3}{5},\sqrt{6},-\tfrac{1}{10},1)$. This approximation is proven to be $2.4 \times 10^{-5}$ close (in $C^0$-norm) to a true localized radial stationary solution. Panel \textbf{A} shows the approximation as a graph over $\mathbb{R}^2$. Panel~\textbf{B} shows its variation on the domain $[- (\ell r_* + L), \ell r_* + L]^2$; the dashed grey circle marks the boundary of the validated local center-stable manifold. Panel \textbf{C} shows the approximation as a function of the radius~$r$.}
\label{fig:SH}
\end{figure}

\subsection{Example \ref{ex:exampleFN}: three-component FitzHugh–Nagumo type equation}
\label{sec:2DFN}

Our third example consists in proving localized radial stationary solutions of the three-component FitzHugh-Nagumo type equation as described in Theorem \ref{thm:exampleFN}. Recall that the equation \eqref{eq:exampleSH} under study reads
\[
\begin{cases}
\hspace{0.2cm} (U_1)_t = \epsilon^2 \Delta U_1 + U_1 - U_1^3 - \epsilon (\beta_1 + \beta_2 U_2 + \beta_3 U_3), \\
\tau (U_2)_t = \Delta U_2 + U_1 - U_2, \\
\theta (U_3)_t = \beta_4^2 \Delta U_3 + U_1 - U_3,
\end{cases}\quad U_i = U_i(t, x) \text{ for } i=1,2,3, \quad t \ge 0, \quad x \in \R^2,
\]
with parameters, $\tau, \theta, \epsilon>0$, $\beta_1, \beta_2, \beta_3 \in \R$ and $\beta_4 > 1$. Hence, the stationary solutions solve \eqref{eq:elliptic} with $q=3$, $d=2$ and
\[
\mathbf{N}(U_1, U_2, U_3) =
\begin{pmatrix}
\epsilon^{-2} (U_1 - U_1^3) - \epsilon^{-1} (\beta_1 + \beta_2 U_2 + \beta_3 U_3) \\
U_1 - U_2 \\
\beta_4^{-2}(U_1 - U_3)
\end{pmatrix}.
\]
Following the parameter values studied in \cite{MR3156797} (note that in the cited article, the authors use $\gamma, \alpha, \beta, D$ instead of $\beta_1, \beta_2, \beta_3, \beta_4$, respectively), for $\epsilon = \tfrac{3}{10}$ and $(\beta_1, \beta_2, \beta_3, \beta_4) = (\tfrac{1}{2},\tfrac{1}{2},1,3)$ we compute a numerical approximation of a localized radial stationary solution about the constant equilibrium $c = (c_*, c_*, c_*)$ where $c_*$ is the smallest root of the cubic polynomial $\upsilon - \upsilon^3 - \epsilon (\beta_1 + \beta_2 \upsilon + \beta_3 \upsilon)$. For the prescribed parameter values, we find $c_* = -20^{-1}(5 + \sqrt{145})$.

To circumvent tedious algebraic manipulations, we solve $- D\mathbf{N}(c) \Gamma = \Gamma \Lambda^2$ via a rigorous numerical method (also based on a Newton-Kantorovich type theorem) whose details is developed in Appendix \ref{appendix:eigenproblem}.
It follows that $\hat{\lambda} \in 0.3687766247191 + [-10^{-14}, 10^{-14}]$. Moreover, for any $\mu > 0$, the inequality \eqref{eq:psi_hat} is satisfied for
\[
\hat{\psi} = \frac{3}{2 \epsilon^2}\bigl((2c_* + (1+\mathcal{L}_y)\mu\bigr)(1+\mathcal{L}_y)\mu |
\Lambda^{-1} \Gamma^{-1}|_{\mathscr{B}(\mathbb{C}^3,\mathbb{C}^3)} |\Gamma
|_{\mathscr{B}(\mathbb{C}^3,\mathbb{C}^3)}.
\]

Figure \ref{fig:FN} shows the numerical approximation of a localized radial stationary solution of the three-component FitzHugh-Nagumo type equation \eqref{eq:exampleFN}. We successfully verified Proposition \ref{prop:graph} and Theorem~\ref{thm:radii_polynomial} and obtained a $C^0$-error bound for the numerical approximation depicted in Figure \ref{fig:exampleFN} and \ref{fig:FN}: $9.8 \times 10^{-7}$.
For the proof, we chose the Lipschitz constants of the local graph of the center-stable manifold to be $\mathcal{L}_x = 1$ and $\mathcal{L}_y \approx 0.02$. The domain of the local graph is contained in $[0, 0.023] \times [-10^{-7}, 10^{-7}]^3$.
Similarly to the proof of the solution with $m=1$ of the cubic Klein-Gordon equation, there is a sizeable gap between the numerical truncation order and the padding order. For the proof of the solution depicted in Figure \ref{fig:FN}, we used $n_{\mathscr{T},\textnormal{num}} = 60$, $n_{\mathscr{T},\textnormal{pad}} = 100$ and $n_{\mathscr{C},\textnormal{num}} = 250$, $n_{\mathscr{C},\textnormal{pad}} = 10,000$ (finished in under 6 hours on a server with an Intel E7-4850 v4 Broadwell CPU, using 32 threads, with 2510 GiB of max.\ available RAM).

Let us now point out the difficult aspects of this computer-assisted proof. The success of Theorem~\ref{thm:radii_polynomial} depends on the roots of the left-hand-side of \eqref{eq:radii_constraint1}, which are real if
\[
(1 - Z_1 - |A(\Gamma, \Gamma \Lambda, 0, 0)|_{\mathscr{B}(\mathbb{C}^q, X)} \mathcal{L}_y)^2 \ge 2 Z_2 (Y + |A(\Gamma, \Gamma \Lambda, 0, 0)|_{\mathscr{B}(\mathbb{C}^q, X)} \mathcal{L}_y |\bar{\eta}|_{\mathbb{C}^q}).
\]
In particular, both the product $Z_2 Y$
and the product $Z_2  \mathcal{L}_y |\bar{\eta}|_{\mathbb{C}^q}$ needs to be sufficiently small.
As explained at the beginning of Section \ref{sec:applications}, successive Newton iterations allow the $Y$ bound to be as small as we need (provided the Taylor and Chebyshev series expansions include enough modes).
However, the values of $|\bar{\eta}|_{\mathbb{C}^q}$ and $\mathcal{L}_y$ are controlled by $L$: the larger the latter, the smaller the former. In turn, this worsens the $Z_1$ bound such that we need to take a larger $n_\mathscr{C}$.

Unfortunately, for our choice of parameters, the real part of the slowest eigenvalue $\hat{\lambda}$ is fairly small which impedes the asymptotic convergence of the localized radial solution to the constant equilibrium~$c$. In other words, $L$ must be substantially increased to overcome $Z_2$. Then, large $n_\mathscr{C}$ needs be chosen in order to attain $Z_1<1$, which is necessary for inequality~\eqref{eq:radii_constraint2} to hold.
Once again, we believe that combining the domain decomposition method developed in this paper with \cite{MR4292534} should alleviate this difficulty.

\begin{figure}[!ht]
\centering
\begin{subfigure}[b]{0.3\textwidth}
\centering
\includegraphics[width=\textwidth]{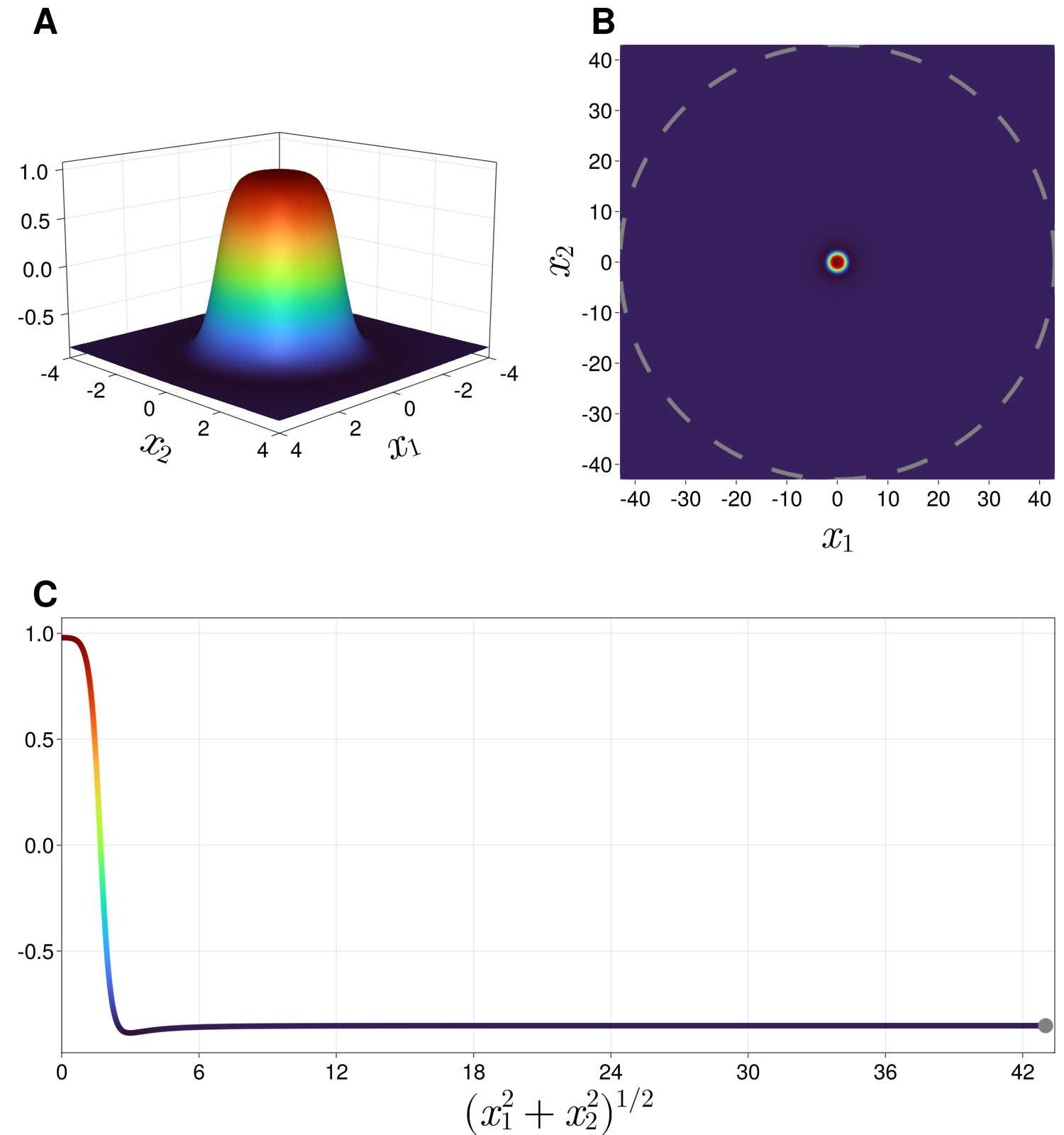}
\caption{}
\end{subfigure}
\hfill
\begin{subfigure}[b]{0.3\textwidth}
\centering
\includegraphics[width=\textwidth]{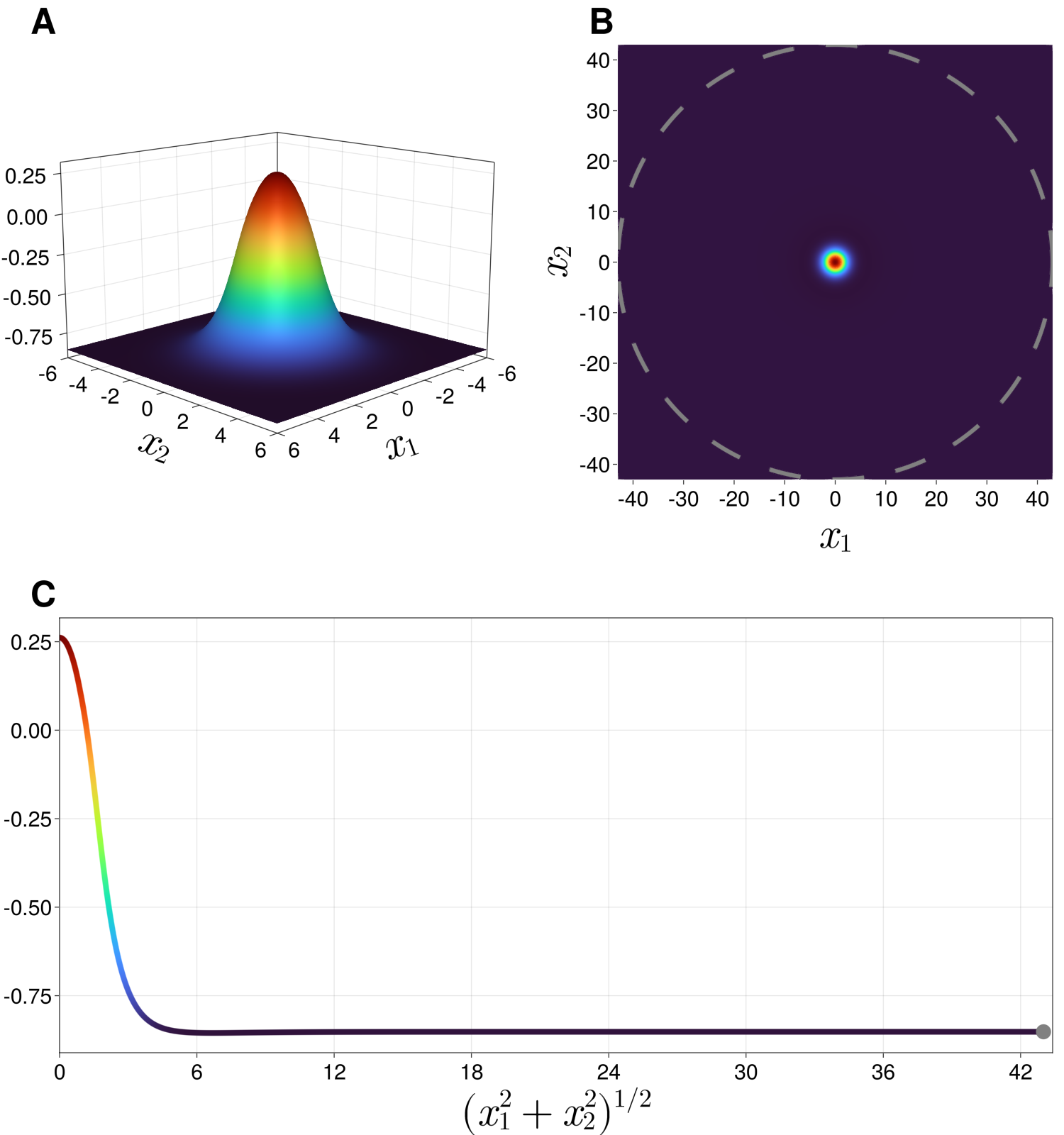}
\caption{}
\end{subfigure}
\hfill
\begin{subfigure}[b]{0.3\textwidth}
\centering
\includegraphics[width=\textwidth]{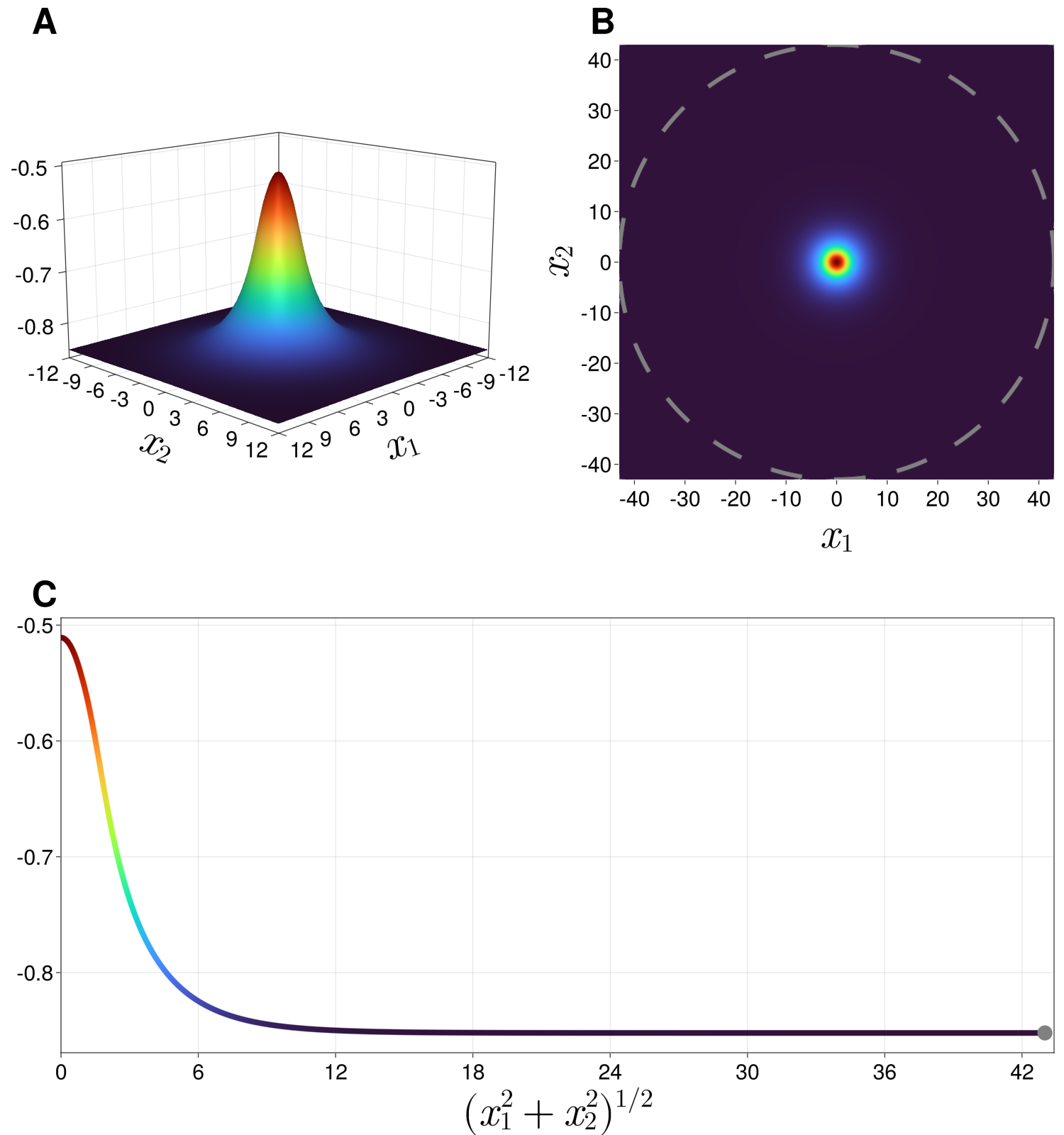}
\caption{}
\end{subfigure}
\caption{Numerical stationary planar radial \emph{spot} of the three-component FitzHugh-Nagumo type equation \eqref{eq:exampleFN} on $\mathbb{R}^2$ at parameter values $\epsilon = \tfrac{3}{10}$ and $(\beta_1, \beta_2, \beta_3, \beta_4) = (\tfrac{1}{2},\tfrac{1}{2},1,3)$; the three components are represented in the subfigures (a), (b) and (c), respectively. This approximation is proven to be $9.8 \times 10^{-7}$ close (in $C^0$-norm) to a true localized radial stationary solution. Panel \textbf{A} shows the approximation as graphs over $\mathbb{R}^2$. Panel~\textbf{B} shows its variation on the domain $[- (\ell r_* + L), \ell r_* + L]^2$; the dashed grey circle marks the boundary of the validated local center-stable manifold. Panel \textbf{C} shows the approximation as a function of the radius~$r$.}
\label{fig:FN}
\end{figure}

%%%%%%%%%%%
%%  APPENDIX  %%
%%%%%%%%%%%

\appendix
 
\section{The Newton-Kantorovich bounds} 
\label{appendix:YZ1Z2}
%!TEX root = radial_elliptic.tex

In this appendix we present the proofs of the inequalities~\eqref{eq:Y}, \eqref{eq:Z_1} and \eqref{eq:Z_2} in the proof of Theorem~\ref{thm:radii_polynomial}. We will use the notation from that proof without reintroducing it here.

Firstly, by the triangle inequality,
\[
|T(\bx; 0) - \bx|_X = |\mathcal{A} F(\bx; 0)|_X \le | A \pi^{n_\mathscr{T}, n_\mathscr{C}} F (\bx; 0) |_X + | \mathcal{A}_\infty F(\bx; 0) |_X.
\]
\noindent Since $\pi^{\infty(n_\mathscr{T}), \infty(n_\mathscr{C})} \bx = 0$, we have
\[
\mathcal{A}_\infty \pi^{\infty(n_\mathscr{T}), \infty(n_\mathscr{C})} F(\bx; 0) = (0, 0, g_\infty, h_\infty),
\]
\noindent where
\begin{align*}
\{g_\infty\}_n &\bydef
\begin{cases}
0, & n \le n_\mathscr{T}, \\
\displaystyle \frac{\ell^2 \{\mathbf{N}(\bar{v})\}_{n-2}}{n (n + d - 2)}, & n > n_\mathscr{T},
\end{cases} \\
\{h_\infty\}_n &\bydef
\begin{cases}
0, & n \le n_\mathscr{C}, \\
\displaystyle \frac{L}{2} \frac{\{f(\bar{w})\}_{n-1} - \{f(\bar{w})\}_{n+1}}{2n}, & n > n_\mathscr{C}.
\end{cases}
\end{align*}
We infer that 
\begin{align*}
|\mathcal{A}_\infty F(\bx; 0)|_X
&= \max(|g|_{\mathscr{T}^q}, |h|_{\mathscr{C}^{1+2q}_\nu}) \\
&\le \max\left(
\frac{\ell^2 | \pi^{\infty(n_\mathscr{T}-2)} \mathbf{N}(\bar{v}) |_{\mathscr{T}^q}}{(n_\mathscr{T}+1)(n_\mathscr{T}+d-1)},
\frac{L(\nu + \nu^{-1})| \pi^{\infty(n_\mathscr{C}-1)} f(\bar{w}) |_{\mathscr{C}_\nu^{1+2q}}}{4(n_\mathscr{C}+1)}
\right).
\end{align*}
By definition of $Y$, this proves \eqref{eq:Y}.

Secondly, by the triangle inequality,
\begin{align*}
|D T(\bx; 0)|_{\mathscr{B}(X, X)} 
&= |I - \mathcal{A} D F(\bx; 0)|_{\mathscr{B}(X, X)} \\
&= |I - (A\pi^{n_\mathscr{T}, n_\mathscr{C}} + \mathcal{A}_\infty) D F(\bx; 0)|_{\mathscr{B}(X, X)} \\
&\le |I - A\pi^{n_\mathscr{T}, n_\mathscr{C}} D F(\bx; 0) \pi^{n_\mathscr{T}, \order n_\mathscr{C} + 1} - \mathcal{A}_\infty D F(\bx; 0) \pi^{\infty(n_\mathscr{T}), \infty(n_\mathscr{C})}|_{\mathscr{B}(X, X)}  \\
&\qquad + |\mathcal{A}_\infty D F(\bx; 0) \pi^{n_\mathscr{T}, n_\mathscr{C}}|_{\mathscr{B}(X, X)} + |A\pi^{n_\mathscr{T}, n_\mathscr{C}} D F(\bx; 0) \pi^{\infty(n_\mathscr{T}), \infty(\order n_\mathscr{C} + 1)}|_{\mathscr{B}(X, X)}.
\end{align*}
We then have
\[
\mathcal{A}_\infty \pi^{\infty(n_\mathscr{T}), \infty(n_\mathscr{C})} D F(\bx; 0) \pi^{\infty(n_\mathscr{T}), \infty(n_\mathscr{C})} \chi =
\begin{pmatrix}
0 \\
0 \\
\pi^{\infty(n_\mathscr{T})} v \\
\pi^{\infty(n_\mathscr{C})} w
\end{pmatrix} = \pi^{\infty(n_\mathscr{T}), \infty(n_\mathscr{C})} \chi,
\]
for all $\chi = (\eta, \phi, v, w) \in X$, which implies
\begin{align*}
&|I - A\pi^{n_\mathscr{T}, n_\mathscr{C}} D F(\bx; 0) \pi^{n_\mathscr{T}, \order n_\mathscr{C} + 1} - \mathcal{A}_\infty D F(\bx; 0) \pi^{\infty(n_\mathscr{T}), \infty(n_\mathscr{C})}|_{\mathscr{B}(X, X)}
\\
&\hspace{7.5cm}=
|\pi^{n_\mathscr{T}, n_\mathscr{C}} - A\pi^{n_\mathscr{T}, n_\mathscr{C}} D F(\bx; 0) \pi^{n_\mathscr{T}, \order n_\mathscr{C} + 1}|_{\mathscr{B}(X, X)}.
\end{align*}
Moreover,
\[
\mathcal{A}_\infty D F(\bx; 0) \pi^{n_\mathscr{T}, n_\mathscr{C}} \chi = (0, 0, g_\infty^{n_\mathscr{T}}, h_\infty^{n_\mathscr{C}}), \qquad \text{for all } \chi = (\eta, \phi, v, w) \in X,
\]
\noindent where
\begin{align*}
\{g_\infty^{n_\mathscr{T}}\}_n &\bydef
\begin{cases}
0, & n \le n_\mathscr{T}, \\
\displaystyle \frac{\ell^2 \{D \mathbf{N}(\bar{v}) \pi^{n_\mathscr{T}} v\}_{n-2}}{n (n + d - 2)}, & n > n_\mathscr{T},
\end{cases} \\
\{h_\infty^{n_\mathscr{C}}\}_n &\bydef
\begin{cases}
0, & n \le n_\mathscr{C}, \\
\displaystyle \frac{L}{2} \frac{\{Df(\bar{w}) \pi^{n_\mathscr{C}} w\}_{n-1} - \{Df(\bar{w}) \pi^{n_\mathscr{C}} w\}_{n+1}}{2n}, & n > n_\mathscr{C}.
\end{cases}
\end{align*}
We infer that
\begin{align*}
|\mathcal{A}_\infty D F(\bx; 0) \pi^{n_\mathscr{T}, n_\mathscr{C}}|_{\mathscr{B}(X, X)} \le
\max&\left(\frac{\ell^2 |D\mathbf{N}(\bar{v})|_{\mathscr{B}(\mathscr{T}^q, \mathscr{T}^q)}}{(n_\mathscr{T}+1)(n_\mathscr{T}+d-1)},
\frac{L(\nu + \nu^{-1}) | Df(\bar{w})|_{\mathscr{B}(\mathscr{C}_\nu^{1+2q}, \mathscr{C}_\nu^{1+2q})}}{4(n_\mathscr{C}+1)}\right).
\end{align*}
\noindent Furthermore,
\begin{align*}
&A\pi^{n_\mathscr{T}, n_\mathscr{C}} D F(\bx; 0) \pi^{\infty(n_\mathscr{T}), \infty(\order n_\mathscr{C} + 1)} \chi =
A \left( 2\sum_{n > \order n_\mathscr{C} + 1 } \{w_2\}_n \, , \, 2\sum_{n > \order n_\mathscr{C} + 1 } \{w_3\}_n \, , \,  0 \, , \, \tilde{h} \right),
\end{align*}
for all $\chi = (\eta, \phi, v, w) \in X$, where
\[
\{\tilde{h}\}_n =
\begin{cases}
\displaystyle
\begin{pmatrix}
0 \\ \sum_{m > n_\mathscr{T}} \{v\}_m r_*^m \\ \ell^{-1} \sum_{m > n_\mathscr{T}} m \{v\}_m r_*^m
\end{pmatrix} - L\sum_{m > n_\mathscr{C} + 1} \frac{(-1)^m}{m^2 -1} \{Df(\bar{w}) \pi^{\infty(\order n_\mathscr{C}+1)} w\}_m, & n = 0, \\
0, & n \ge 1,
\end{cases}
\]
since, according to \eqref{eq:discrete_conv}, $\{Df(\bar{w}) \pi^{\infty(\order n_\mathscr{C}+1)} w\}_n = 0$ for all $n \le n_\mathscr{C} + 1$. 
As it was assumed that $r_* \in (0, e^{-1/(n_\mathscr{T} + 1)}]$, we have that $n r_*^n$ is a decreasing sequence for all $n > n_\mathscr{T}$. 
It follows that
\begin{align*}
&|A\pi^{n_\mathscr{T}, n_\mathscr{C}} D F(\bx; 0) \pi^{\infty(n_\mathscr{T}), \infty(\order n_\mathscr{C} + 1)} |_{\mathscr{B}(X, X)} \\
&\hspace*{2cm}\le  | A |_{\mathscr{B}(X, X)} \max \left( \frac{2}{\nu^{\order n_\mathscr{C}+2}}, r_*^{n_\mathscr{T} + 1}\max(1, \ell^{-1}(n_\mathscr{T} + 1)) + \frac{L| Df(\bar{w})|_{\mathscr{B}(\mathscr{C}_\nu^{1+2q}, \mathscr{C}_\nu^{1+2q})}}{\nu^{n_\mathscr{C} + 2}((n_\mathscr{C} + 2)^2 - 1)} \right).
\end{align*}
\noindent By definition of $Z_1$, this proves \eqref{eq:Z_1}.

Thirdly, by the triangle inequality,
\begin{align*}
\sup_{\chi \in \textnormal{cl}(B_\varrho(\bx))} | D^2 T(\chi; 0) |_{\mathscr{B}(X, \mathscr{B}(X, X))} &= \sup_{\chi \in \textnormal{cl}(B_\varrho(\bx))} | \mathcal{A} D^2 F(\chi; 0) |_{\mathscr{B}(X, \mathscr{B}(X, X))} \\
&\le |\mathcal{A}|_{\mathscr{B}(X, X)} \sup_{\chi \in \textnormal{cl}(B_\varrho(\bx))} |D^2 F(\chi; 0)|_{\mathscr{B}(X, \mathscr{B}(X, X))}.
\end{align*}
\noindent We have
\[
|\mathcal{A}|_{\mathscr{B}(X, X)} \le | A |_{\mathscr{B}(X, X)} + \max\left(\frac{1}{(n_\mathscr{T}+1)(n_\mathscr{T}+d-1)}, 1\right) = | A |_{\mathscr{B}(X, X)} + 1.
\]
Additionally,
\[
[D^2 F(\chi; 0)] (\chi', \chi'') = (0, 0, \ell^2 [D^2 \mathbf{N}(v)](v', v''), \hat{h}),
\]
\noindent for all $\chi = (\eta, \phi, v, w), \chi' = (\eta', \phi', v', w'), \chi'' = (\eta'', \phi'', v'', w'') \in X$, where
\[
\begin{aligned}
&\{ \hat{h} \}_n \bydef \\
&\hspace{0.5cm}
\begin{cases}
\displaystyle
\frac{L}{2} \Bigl(\{[D^2 f(w)](w', w'')\}_0 - \frac{1}{2}\{[D^2 f(w)](w', w'')\}_1 - 2 \sum_{m \ge 2} \frac{(-1)^m}{m^2 -1} \{[D^2 f(w)](w', w'')\}_m \Bigr), & n = 0,
\\
\displaystyle
\frac{L}{2} \frac{\{[D^2 f(w)](w', w'')\}_{n-1} - \{[D^2 f(w)](w', w'')\}_{n+1}}{2n}, & n \ge 1.
\end{cases}
\end{aligned}
\]
Subsequently, by extensively applying the triangle inequality and the Banach algebra property, we obtain
\begin{alignat*}{2}
&|[D^2 \mathbf{N}(v)](v', v'')|_{\mathscr{T}^q} &&= \max_{i=1,\dots,q} \sum_{j,k = 1}^q |\frac{\partial^2}{\partial v_j \partial v_k} \mathbf{N}_i(v) * v'_j * v''_k|_\mathscr{T} \\
& &&\le \max_{i=1,\dots,q} \sum_{j,k = 1}^q \frac{\partial^2}{\partial v_j \partial v_k} \mathbf{N}_{\textnormal{abs},i}(|v_1|_\mathscr{T}, \dots, |v_q|_\mathscr{T}) |v'_j|_\mathscr{T} |v''_k|_\mathscr{T}, \\
&|[D^2 f(w)](w', w'')|_{\mathscr{C}^{1+2q}_\nu} &&= \max_{i=1,\dots,q} \sum_{j,k = 1}^{1+2q} |\frac{\partial^2}{\partial w_j \partial w_k} f_i(w) * w'_j * w''_k|_{\mathscr{C}_\nu} \\
& &&\le \max_{i=1,\dots,q} \sum_{j,k = 1}^q \frac{\partial^2}{\partial w_j \partial w_k} f_{\textnormal{abs},i}(|w_1|_{\mathscr{C}_\nu}, \dots, |w_{1+2q}|_{\mathscr{C}_\nu}) |w'_j|_{\mathscr{C}_\nu} |w''_k|_{\mathscr{C}_\nu}.
\end{alignat*}
Then, introducing $\zeta \bydef [D^2 f(w)](w', w'')$ for convenience, for any $i=1, \dots, 1+2q$ we have
\begin{align*}
|\hat{h}_i|_{\mathscr{C}_\nu}
&= \frac{L}{2} \Bigl( |\{\zeta_i\}_0 - \frac{1}{2}\{\zeta_i\}_1 - 2 \sum_{n \ge 2} \frac{(-1)^n}{n^2 -1} \{\zeta_i\}_n | + \sum_{n \ge 1} \frac{|\{\zeta_i\}_{n-1} - \{\zeta_i\}_{n+1}|}{n} \nu^n
\Bigr)
\\
&= \frac{L}{2} \Bigl( |\{\zeta_i\}_0 - \frac{1}{2}\{\zeta_i\}_1 - 2 \sum_{n \ge 2} \frac{(-1)^n}{n^2 -1} \{\zeta_i\}_n | + \nu \sum_{n \ge 0} \frac{|\{\zeta_i\}_n - \{\zeta_i\}_{n+2}|}{n+1} \nu^n \Bigr)
\\
&\le \frac{L}{2} \Bigl( (1 + \nu)|\{\zeta_i\}_0| + \frac{1 + \nu^2}{2}|\{\zeta_i\}_1| + 2\sum_{n \ge 2} |\{\zeta_i\}_n| \Bigl(\frac{1}{n^2 -1} + \frac{\nu^{n+1}}{n+1}\Bigr) \Bigr)
\\
&= \frac{L(1 + \nu)}{2} \Bigl( |\{\zeta_i\}_0| + \frac{1 + \nu^2}{2(1 + \nu)}|\{\zeta_i\}_1| + 2 \sum_{n \ge 2} \frac{(n+1)\nu^{-n}+(n^2 -1)\nu}{(n+1)(n^2 -1)(1 + \nu)} |\{\zeta_i\}_n| \nu^n \Bigr)
\\
&\le \frac{L(1 + \nu)}{2} \Bigl( |\{\zeta_i\}_0| + 2 |\{\zeta_i\}_1| \nu  + 2 \sum_{n \ge 2} |\{\zeta_i\}_n| \nu^n \Bigr)
\\
&= \frac{L(1 + \nu)}{2} |\zeta_i|_{\mathscr{C}_\nu}.
\end{align*}
It follows that (with the infinity norm on both $\mathbb{C}^q$ and $\mathbb{C}^{1+2q}$)
\begin{align*}
\sup_{\chi \in \textnormal{cl}(B_\varrho(\bx))} |D^2 F(\chi; 0)|_{\mathscr{B}(X, \mathscr{B}(X, X))}
\le
\max\Bigg(&\ell^2 |D^2 \mathbf{N}_\textnormal{abs}(|\bar{v}_1|_\mathscr{T}+\varrho, \dots, |\bar{v}_q|_\mathscr{T}+\varrho)|_{\mathscr{B}(\mathbb{C}^{q^2}, \mathbb{C}^q)},
\\
&{\frac{L(1 + \nu)}{2}} |D^2 f_\textnormal{abs}(|\bar{w}_1|_{\mathscr{C}_\nu}+\varrho, \dots, |\bar{w}_{1+2q}|_{\mathscr{C}_\nu}+\varrho)|_{\mathscr{B}(\mathbb{C}^{(1+2q)^2}, \mathbb{C}^{1+2q})}\Bigg).
\end{align*}
By definition of $Z_2$, this proves \eqref{eq:Z_2}.

\section{Rigorous computation of eigenvalues and eigenvectors for FitzHugh-Nagumo} \label{appendix:eigenproblem}
%!TEX root = radial_elliptic.tex

For the three-component FitzHugh-Nagumo type equation \eqref{eq:exampleFN}, we show how to rigorously compute $\lambda \in \mathbb{C}$ and $g \in \mathbb{C}^3$ satisfying
\begin{equation}\label{eq:eigenproblem_fhn}
D\mathbf{N}(c_*, c_*, c_*) g + g \lambda^2 = 0,
\end{equation}
where
\[
D\mathbf{N}(c_*, c_*, c_*) =
\begin{pmatrix}
\epsilon^{-2}(1 - 3c_*^2) & -\epsilon^{-1} \beta_2 & -\epsilon^{-1} \beta_3 \\
1 & -1 & 0 \\
\beta_4^{-2} & 0 & -\beta_4^{-2}
\end{pmatrix}.
\]
While there exist efficient libraries to rigorously compute eigenvalues and eigenvectors, we include here for completeness a strategy relying on a Newton-Kantorovich argument very similar (and far simpler) in nature to the proof of Theorem \ref{thm:radii_polynomial}.

It is straightforward to verify that if \eqref{eq:eigenproblem_fhn} holds, then the first component $g_1$ of $g$ cannot vanish. A pair $(\lambda, g/g_1) \in \mathbb{C} \times \mathbb{C}^3$ solving \eqref{eq:eigenproblem_fhn} forms an isolated zero of the mapping $F_\textnormal{eig} : \mathbb{C} \times \mathbb{C}^3 \to \mathbb{C} \times \mathbb{C}^3$ given by
\[
F_\textnormal{eig} (\lambda, g) \bydef
\begin{pmatrix}
g_1 - 1 \\
D\mathbf{N}(c_*, c_*, c_*) g + g \lambda ^ 2
\end{pmatrix}, \qquad \text{for all } (\lambda, g) \in \mathbb{C} \times \mathbb{C}^3.
\]
Fix $\varrho_\textnormal{eig} > 0$. Let $\bx_\textnormal{eig} = (\bar{\lambda}, \bar{g})$ be a numerical zero of $F_\textnormal{eig}$ and $A_\textnormal{eig}$ a numerical inverse of $DF_\textnormal{eig} (\bx_\textnormal{eig})$. Define $Y_\textnormal{eig} \bydef |A_\textnormal{eig} F_\textnormal{eig} (\bx_\textnormal{eig})|_{\mathbb{C} \times \mathbb{C}^3}$ and $Z_\textnormal{eig} \bydef \sup_{\chi_\textnormal{eig} \in \textnormal{cl}(B_{\varrho_\textnormal{eig}}(\bx_\textnormal{eig}))}|I - A_\textnormal{eig} DF_\textnormal{eig} (\chi_\textnormal{eig})|_{\mathscr{B}(\mathbb{C} \times \mathbb{C}^3, \mathbb{C} \times \mathbb{C}^3)}$. Observe that $Y_\textnormal{eig}$ and $Z_\textnormal{eig}$ can be rigorously computed using interval arithmetic (see e.g. \cite{IntervalArithmeticJulia}).

If $Z_\textnormal{eig} < 1$ and $\bar{\rho}_\textnormal{eig} \bydef Y_\textnormal{eig}/(1 - Z_\textnormal{eig}) \in [0, \varrho_\textnormal{eig}]$, then the operator $T_\textnormal{eig} \bydef I - A_\textnormal{eig} F_\textnormal{eig}$ is a contraction in $\textnormal{cl}(B_{\varrho_\textnormal{eig}}(\bx_\textnormal{eig}))$. Consequently, $\bar{\rho}_\textnormal{eig}$ is an a posteriori error bound on $\bx_\textnormal{eig}$.

\bibliographystyle{unsrt}
\bibliography{papers}

\end{document}